\documentclass[a4paper,10pt]{amsart}
\usepackage[utf8]{inputenc}
\usepackage{lmodern}
\usepackage[T1]{fontenc}
\usepackage{microtype} 	
\usepackage[all]{xy}
\usepackage{tikz-cd}
\usepackage{amsmath,amssymb,amsfonts,amsthm,latexsym,mathrsfs}
\usepackage{mathtools}
\usepackage{ stmaryrd } 
\usepackage{xcolor}
\usepackage{hyperref}
\usepackage{cleveref}
\hypersetup{colorlinks=true,linkcolor=gray,citecolor=gray}
\usepackage[hyperpageref]{backref} 
\newtheorem{lemma}{Lemma}[section]
\newtheorem{proposition}[lemma]{Proposition}
\newtheorem{theorem}[lemma]{Theorem}
\newtheorem*{theorem*}{Theorem}
\newtheorem{corollary}[lemma]{Corollary}
\newtheorem*{obstruction}{Interpretation obstruction}
\newtheorem*{claim*}{Claim}

\newtheorem{maintheorem}{Theorem}
\newtheorem{mainproposition}[maintheorem]{Proposition}

\theoremstyle{definition}

\newtheorem{definition}[lemma]{Definition}
\newtheorem*{definition*}{Definition}

\newtheorem{question}[lemma]{Question}
\theoremstyle{remark}
\newtheorem{remark}[lemma]{Remark}
\newtheorem{example}[lemma]{Example}

\newtheorem*{remark*}{Remark}
\newtheorem*{example*}{Example}
\newtheorem*{problem*}{Problem}
\newtheorem*{conclusion*}{Conclusion}
\newtheorem*{lremark*}{Literature remark}

\newcommand{\proplabel}[1]{\label{prop:#1}}
\newcommand{\propref}[1]{Proposition~\ref{prop:#1}}

\newcommand{\lemlabel}[1]{\label{lem:#1}}
\newcommand{\lemref}[1]{Lemma~\ref{lem:#1}}

\newcommand{\thelabel}[1]{\label{the:#1}}
\newcommand{\theref}[1]{Theorem~\ref{the:#1}}

\newcommand{\corlabel}[1]{\label{cor:#1}}

 \newcommand\restr[2]{{
   \left.\kern-\nulldelimiterspace 
   #1 
   \right|_{#2} 
   }}

\newcommand\Span{\operatorname{span}}
\newcommand\ann{\operatorname{ann}}
\newcommand\Hom{\operatorname{Hom}}

\newcommand\Ker{\operatorname{Ker}}
\newcommand\coker{\operatorname{coker}}
\newcommand\Ima{\operatorname{Im}}

\newcommand\Mod{\operatorname{mod}}

\newcommand\ot{\otimes}
\newcommand\mc{\mathcal}

\newcommand\Char{\operatorname{char}}

\newcommand\ov{\overline}

\newcommand\End{\operatorname{End}}

\newcommand\wt{\widetilde}

\newcommand\Irr{\operatorname{Irr}}

\newcommand\rep{\operatorname{Rep}}
\newcommand\vect{\operatorname{vect}}
\newcommand\Glid{\operatorname{Glid}}
\newcommand\Inf{\operatorname{Inf}}
\newcommand\Res{\operatorname{Res}}
\newcommand\RRes{\operatorname{\mc{R}es}}
\newcommand\IInd{\operatorname{\mc{I}nd}}
\newcommand\GGlid{\operatorname{\mc{G}lid}}
\newcommand\glid{\operatorname{glid}}
\newcommand\RepR{\operatorname{R}}
\newcommand\RRepR{\operatorname{\ov{R}}}
\newcommand\Preglid{\operatorname{Preglid}}
\newcommand\Ind{\rm Ind}

\newcommand\Gr{\rm Gr}
\newcommand{\N}{{\mathbb N}}
\newcommand{\Z}{{\mathbb Z}}

\newcommand{\Q}{{\mathbb Q}}

\newcommand{\C}{{\mathbb C}}

\newcommand\op{\oplus}
\newcommand\Sub{\rm Sub}
\newcommand\Chain{\rm Chain}
\newcommand\Idemp{\rm Idemp}

\newcommand\hookmapright[1]{\smash{\mathop{\hookrightarrow}\limits^{#1}}}

\usepackage[margin=1.4in]{geometry}
\begin{document}
\title{Glider Representation Rings with a view on distinguishing groups}
\author{Frederik Caenepeel}
\author{Geoffrey Janssens}

\address{(Frederik Caenepeel) \newline Shanghai Center for Mathematical Sciences, Fudan University, $2005$ Songhu Road, Shanghai, China \newline E-mail address: {\tt frederik\_caenepeel@fudan.edu.cn}}
\address{(Geoffrey Janssens) \newline Departement Wiskunde, Vrije Universiteit Brussel,
Pleinlaan $2$, 1050 Elsene, Belgium \newline E-mail address: {\tt geofjans@vub.ac.be}}

\begin{abstract}
Let $G$ be a finite group. In the first part of the paper we develop further the foundations of the youngly introduced glider representation theory. Glider representations encompass filtered modules over filtered rings and as such carry much information of $G$. Therefore the main focus is on the glider representation ring $R_d(\widetilde{G})$, which is shown to be realisable as a concrete subring of the split Grothendieck ring of the monoidal category $\text{glid}_d(G)$ of (Noetherian) glider $\mathbb{C}$-representations of (length $d$) of $G$. 
In the second part we investigate a Wedderburn-Malcev type decomposition of the (infinite-dimensional) $\mathbb{Q}$-algebra $\mathbb{Q}(\widetilde{G}) := \mathbb{Q} \otimes_{\mathbb{Z}}R_1(\widetilde{G})$. The main theorem obtains a  $\mathbb{Q}[G^{ab}]$-module decomposition of $\mathbb{Q}(\widetilde{G})$ relating it in a precise way to $\mathbb{C}$-representation theory of subnormal subgroups in $G$. Under certain vanishing assumptions, which are proven to hold for nilpotent groups (of class $2$), the second main theorem completely describes a $\mathbb{Q}[G^{ab}]$-algebra decomposition. We end with pointing out applications on distinguishing isocategorical groups. 
\end{abstract}
\maketitle

\newcommand\blfootnote[1]{%
  \begingroup
  \renewcommand\thefootnote{}\footnote{#1}%
  \addtocounter{footnote}{-1}%
  \endgroup
}

\blfootnote{Second author is supported by Fonds Wetenschappelijk Onderzoek - Vlaanderen (FWO).}

\tableofcontents

\section{Introduction}
Let $G$ be a finite group and $k$ any field. The guiding goal of this article is to reconstruct (group-theoretical pieces of) $G$ from its representation category $\rep_k(G)$ as monoidal category, however through the lens of glider representations.

Glider representations are generalizations of filtered modules over filtered rings (cf. \cite[example 3.15]{HvR}). This recently introduced theory has been developed by Caenepeel-Van Oystaeyen in a series of papers  \cite{CVo1, CVo2, CVo4} and a full exposition of this young theory can already be found in their book \cite{CvoBook}. 

Recall that, given a chain of (potentially equal) subgroups $G_0 \leq \ldots \leq G_d=G$, a glider representation of this chain consists roughly of a $k[G]$-module $M$ together with a descending chain of $k[G_0]$-submodules $M_j$, $j \in \Z_{\leq 0}$, such that $k[G_{i}] M_j \subseteq M_{j+i}$ for all $i$. Two general aims for which glider representation theory is expected to be a suitable framework for are: 
\begin{enumerate}
\item an 'object-wise dealing' with composition series of $k[G]$-modules when $\Char{(k)} > 0$,\vspace{0,1cm} 
\item a 'relative representation theory' for a pair $(H,G)$ with $H$ a subgroup of $G$ (in the sense of working with $\rep (G\mid \chi)$, the representations of $G$ lying over $\chi \in \rep(H)$).
\end{enumerate}
In \Cref{subsectie constructie en prelim} we recall all the necessary background. 

In recent work of Henrard-van Roosmalen \cite{HvR} the categorical foundations of glider representations for any $\Gamma$-filtered ring with $\Gamma$ an ordered group was developed and put on firm footing.  For example when fixing a chain $G_0 \leq \ldots \leq G$  one considers the associated filtration $0 \subset k[G_0] \subseteq \ldots \subseteq k[G]$ by subalgebras. In this case the category of glider representations is denoted $\Glid\big( F(k[G])\big)$. As an application of their theory, they obtained the striking result \cite[Th. 9.19.]{HvR} that already for the one-step filtration $\{e_G \} < G$ the category $\Glid\big( F(k[G])\big)$, viewed as monoidal category (so without consideration of the symmetry), is rich enough to reconstruct $k[G]$ as bialgebra (and hence also $G$) uniquely from it. Therefore it is natural to consider which information is contained in natural decategorifications connected to $\Glid\big( F(k[G])\big)$.

The main contributions of this paper are two-fold. Firstly, we develop further the foundations of the theory. Among others, we generalize the construction of the 'generalized character ring' from \cite{CVo4} to the general setting (i.e. any field, filtered ring and filtration). The result will be the concept of glider representation ring $R_d(\widetilde{G})$ and its 'reduced' version $\RRepR_d(\widetilde{G})$. If $G$ is non-abelian these rings are in general infinite dimensional. 

 Secondly, in the case of the filtration $\{e_G \} < G$, we obtain a description of the reduced representation ring modulo its Jacobson radical. Also, insight on the latter is obtained, in particular forming a substantial contribution towards a concrete Wedderburn-Malcev type decomposition. In the description just mentioned three concrete modules appear for which we find an interpretation in terms of classical representation theory and group theory.

 We will now review in more detail the main results obtained.
\subsubsection*{Category of gliders and decategorifications}
%
 Such as the categories $\Mod(kG_i)$, the category of gliders $\Glid\big( F(k[G])\big)$ is still a symmetric monoidal additive category, however it is not abelian and even when $\Char{(k)}=0$ it is not semisimple. Nevertheless as shown in \cite[Theorem 5.12]{HvR} it still enjoys a rich categorical structure, being still a complete and cocomplete deflation quasi-abelian category. In \Cref{subsectie constructie en prelim} we recall all the necessary background on gliders and their foundations. 

Let us compactly sketch the construction of the (reduced) glider representation ring. When considering Noetherian gliders $\glid \big(F(k[G])\big)$, i.e. those with total dimension finite (in particular $M_i = 0$ for $i \ll 0$), one obtains a filtration of full subcategories
$$\cdots \subseteq \glid_{d-1} F(k[G]) \subseteq \glid_{d} F(k[G])\subseteq \glid F(k[G])$$
where $\glid_t F(k[G])$ are those for which $M_i = 0$ if $i < -t$. As explained in \Cref{subsectie decateg} those of length $t-1$ also form an (additive) tensor ideal in those of length $t$ and hence one can consider the split Grothendieck ring of the members, resp. the factors, of the filtration above.  The glider representation ring $\RepR_t(\widetilde{G})$, resp. its reduced form $\RRepR_t(\widetilde{G})$, can be realised as the subring generated by the so-called irreducible gliders (in the sense of \cite{CVo1}). Moreover, in \Cref{sectie gliders length 1} the glider character ring is constructed in full generality (generalizing \cite{CVo4}). 

In \Cref{sectie 3 over functors en andere realisatie} we introduce a new subcategory, namely gliders of module type, and thereon give a concrete realisation of the functors of induction and (co)-invariants. Most notably, this subcategory allows us to give another realisation of $\RRepR_t(\widetilde{G})$ in such a way that we can deduce that irreducible gliders form a basis, see \Cref{theorem glider repr as decat}. A result previously only known in the case of an abelian group $G$ and a one-step filtration. 

\subsubsection*{Gliders of length $1$}
From \Cref{sectie SES} on we focus on the filtration $\{e_G\} \subset G$ and $\glid_1 F(k[G])$, the gliders of length $1$. As already alluded to, this case is surprisingly rich as shown by the fact that the monoidal structure of $\glid_1 F(k[G])$ determines $G$ uniquely. The proof of \cite[Th. 9.19.]{HvR} in fact even reconstructs the Fiber functor of $\Mod(k[G])$ in terms of the monoidal structure of $\glid_1 F(k[G])$. However, as in the classical case, its decategorifications are (a priori) much smaller invariants. The focus of the second part of this paper is to understand which information on $G$ and $\rep_{\C}(G)$ is still retained by the ring structure of the glider representation ring $\RRepR_1(\widetilde{G})$. 

Firstly in \Cref{subsectie length 1} we obtain a parametrization of the isomorphism classes of the irreducible gliders of length $1$. For this, let $\Gr(U) = \sqcup_{j=1}^d \Gr(j,U)$ where $U \in \Irr(G)$ and $\dim U = d$. Further denote by $\mc{S}_G$ the set of subsets $B \subseteq \sqcup_{U \in \Irr{G}~\vline~\dim(U) > 1} \Gr(U)$, such that for all $U$ the intersection $B \cap \Gr(j,U)$ is non-empty for at most one $1 \leq j \leq \dim(U)$ and for this $j$ it is in fact a singleton. Then,

\begin{mainproposition}[\Cref{parametrisatie irreducible}]
Let $G$ be a finite group. There is a bijection 
$$\frac{ \{ \text{ irreducible } (k \subseteq kG)-\text{gliders } \}}{\cong}  \xleftrightarrow{1-1} \{ (A,B) \in \mathcal{P}(G/G') \times \mathcal{S}_G \}.$$
\end{mainproposition} 
In the previous result $G'$ is the commutator subgroup of $G$ and $\mathcal{P}(G/G')$ the power set of the abelianization $G^{ab} = G/G'$. 

\subsubsection*{The structure of the (reduced) glider representation ring of length $1$}
Let $\mathbb{Q}(\widetilde{G}) := \mathbb{Q} \otimes_{\mathbb{Z}} \RRepR_1(\widetilde{G})$ be the reduced representation ring extended to a $\mathbb{Q}$-algebra. In general this algebra is infinite-dimensional and we don't know yet whether it is nevertheless Artinian (cf. \Cref{vraag over f.g and Noetherian en etc}). In particular one has not the typical structural results such as Wedderburn-Malcev's decomposition at hand. Nevertheless, the results obtained, and outlined below, indicate that an analogue decomposition might still exist. 

Denote by $N(G)$ the nil-radical of $\mathbb{Q}(\widetilde{G})$. As pointed out in \Cref{sectie SES}, $\mathbb{Q}(\widetilde{G}) / N(G)$ has the structure of a $\Q[G^{ab}]$-algebra. Our first main structural theorem obtains a short exact sequence which relates $\mathbb{Q}(\widetilde{G}) / N(G)$ to the same object for a class of subgroups $H$ of $G$ and three concrete $\Q[G^{ab}]$-submodules which we interpret afterwards.

\begin{maintheorem}[\theref{ses}]
Let $G$ be a finite group. We have the following short exact sequence of $\mathbb{Q}[G^{ab}]$-modules
\begin{equation}\nonumber
\resizebox{0.9 \hsize}{!}{$\xymatrix{ 0 \ar[r] &\frac{P}{P\cap N} + \frac{E}{E \cap N} + \sum_{G' \leq H \lneq G} \ov{\Phi_H^G}\big(\mathbb{Q}(\wt{H})/N(H)\big) \ar[rr]^<<<<<<<<<\Psi &&\mathbb{Q}(\wt{G})/N(G)  \ar[r] & \frac{R}{R \cap N} \ar[r] & 0}$}
\end{equation}
for concretely defined $\mathbb{Q}[G^{ab}]$-submodules $P,E,R$ of $\Q(\widetilde{G})$ and morphism $\Psi$.
\end{maintheorem}

The morphism $\ov{\Phi_H^G}(\cdot)$ is induced from the induction-functor $\IInd^G_H(\cdot)$ which is introduced in \Cref{sectie inductie en restrictie functor}. A remarkable property, which is in sharp contrast to the classical setting of $\rep(G)$, is that for gliders of length $1$ induction gives rise to a ring-monomorphism between the reduced glider representation rings (and in fact already on a larger object). This is an important tool and is proven in \Cref{ind and res induce ring map}.

At first, the modules $P,E,R$ have concrete definitions in the language of gliders. However, in \Cref{sectie interpretaties}, we connect the modules $P$ and $R$ to natural questions in classical representation theory and in group theory. For instance, 
\begin{itemize}
\item As explained in \Cref{subsectie over R interpretatie}, the module $R$ is tightly connected to both the subring of $K_0(\rep_{\C}(G))$ formed by the permutation modules and to the properties of maximal subgroups. For example in \Cref{R niet triviaal alleen voor non-nilpotent} we prove that $G$ is nilpotent if and only if $R= \Q[G^{ab}]$. In general, $\Q[G^{ab}] \subseteq R$ and hence in the nilpotent case $R$ is as small as it can be. 
\item In \Cref{sectie intepretatie P} we discuss the module $P$. More concretely, given a cyclic module $U = \Q[G] u$, $P$ is connected to the sequence of modules $\big( K[G]u^{\ot n} \big)_{n\in \N}$ and to when the trivial module is obtained as a summand.
\end{itemize}

We call these three modules 'obstruction modules' because, as we show in \Cref{sectie decompositie zonder de obstructies}, when they vanish we obtain a precise description of the maximal semisimple quotient of $\mathbb{Q}(\widetilde{G}) = \mathbb{Q} \otimes_{\mathbb{Z}} \RRepR_1(\widetilde{G})$. More concretely,

\begin{maintheorem}[\Cref{gencharring} \& \Cref{both radical equal under assumption coro}]
Let $G$ be a finite nilpotent group such that $P=0=E$. Then, as $\Q[G^{ab}]$-algebra,
$$\mathbb{Q}(\wt{G})/N(G) \cong \bigoplus_{H \leq \Sub(G)} \Q[H^{ab}]$$
where the direct sum runs over a certain class of subnormal subgroups $H$ of $G$. Moreover $N(G) = J(\mathbb{Q}(\wt{G}))$ the Jacobson radical of $\mathbb{Q}(\wt{G})$.
\end{maintheorem}

In \Cref{the set Sub(G)} we ask whether for a solvable group $G$, $\Sub(G)$ in fact contains all subnormal subgroups. The proof consists in constructing an idempotent $\epsilon_H$ for every member $H$ of $\Sub(G)$ for which we obtain a canonical form in \Cref{subsectie ortho idemp}. Herein we also prove that the idempotents $\epsilon_H$ are orthogonal. We would like to emphasize that all this works for general groups $G$. The conditions in the theorem above are needed to prove that these idempotents form a full set of orthogonal idempotents adding up to $1$. However we expect that the condition $P=0=E$ is always fulfilled for nilpotent groups. Furthermore for $G$ non-nilpotent we believe that the modules $R(H)$ should play the role of the summands $\Q[H^{ab}]$.

Using the aformentioned interpretations of $P,E,R$, obtained in \Cref{sectie interpretaties}, we verify all our questions for nilpotent groups of class $2$.

\begin{maintheorem}[\Cref{P=0}, \Cref{R niet triviaal alleen voor non-nilpotent} \& \Cref{represention ring for class 2}]
Let $G$ be a finite nilpotent group of class $2$. Then, as $\Q[G^{ab}]$-algebra,
$$\mathbb{Q}(\wt{G})/J \cong \bigoplus_{H \leq G} \mathbb{Q}[H^{ab}].$$
\end{maintheorem}

Finally, in \Cref{subsectie isocategorical groups} we shortly consider isocategorical groups in the sense of Etingof-Gelaki \cite{EtGe}. Recall that groups $G_1$ and $G_2$ are called isocategorical if $\rep_{\C}(G_1)$ and $\rep_{\C}(G_2)$ are equivalent as tensor category (so without consideration of the symmetry of their monoidal structure). In \cite[Section 4]{MaHi} an (infinite) family of non-isomorphic but isocategorical groups $G^{m}$ and $G_b^m$, with $3 \leq m \in \mathbb{N}$, was constructed. Despite that they have isomorphic representation rings we show that $\RRepR_1(G^m) \ncong \RRepR_1(G_b^m)$. This is done via the decompositions above. Interestingly, in this way new non-monoidal Morita invariants pop up. \vspace{0,2cm}

\noindent \textbf{Acknowledgment.} We would very much like to thank Andreas B\"achle for his thoughts that gave rise to \cref{characterization class 2}. We also warmly thank Ruben Henrard and Adam-Christiaan van Roosmalen for interesting discussions and many remarks on an earlier version of the paper and especially for sharing and discussing the preliminary version of their categorical framework. Finally, we thank Fred Van Oystaeyen for inventing glider representations.\vspace{0,2cm}

\noindent {\it Conventions.}
Throughout the full paper we will assume the following (except stated explicitly otherwise):
\begin{itemize}
\item $k$ will be used for an arbitrary field and $K$ for an algebraically closed field of characteristic $0$,
\item all groups, denoted with the letters $G$ or $H$, will be finite,
\item all $kG$-modules will be left modules,
\item $\mathbb{N}$ denotes the positive integers (with $0$ included).
\item $\subset$ and $<$ will indicate strictly smaller. 
\end{itemize} 

\section{Glider representation rings}
In \Cref{subsectie glider repr} we start by recalling all the necessary background on glider representation theory, specified to the setting of this paper (for a full account we refer to the book by Caenepeel-Van Oystaeyen \cite{CvoBook}). Thereafter, in \Cref{subsectie categorical approach}, we will recall some essentials of the recent categorical foundations laid down by Henrard-van Roosmalen in \cite{HvR}. This will allow us to consider in \Cref{subsectie decateg} the split Grothendieck ring of the various categories arising and certain quotients thereof. The main protagonists of this paper, the glider representation rings, introduced in \Cref{sectie gliders length 1}, can be realised as the subring generated by the so-called irreducible gliders in the aforementioned rings. 

\subsection{Preliminaries and decategorifications} \label{subsectie constructie en prelim} 

\subsubsection{Background on glider representations}\label{subsectie glider repr}
Given a finite group $G$ and a chain of subgroups $G_0 \leq G_1 \leq \ldots \leq G_d = G$, one obtains in a natural way a filtration $F(k[G])$, by subalgebras, of the group algebra $k[G]$ by defining
\begin{equation}\label{onze filtratie}
F_{-n}(k[G]) = 0, F_0(k[G]) = KG_0, F_n(k[G]) = k[G_n]
\end{equation}
for $n >0$ and where $G_{n} = G$ if $n \geq d$.

\begin{definition}
A {\it glider representation over $F(k[G])$} consists of a $k[G]$-module $\Omega$ together with a $\Z_{\leq 0}$-indexed descending chain sequence of $KG_0$-submodules
$$
\ldots \subseteq M_{-2} \subseteq M_{-1} \subseteq M_{0} \subseteq \Omega
$$
such that the $G$-action $k[G] \otimes \Omega \rightarrow \Omega$ induces maps $F_{i}(k[G]) \otimes M_{-j} \rightarrow M_{-j+i}$ for all $i \leq j$. This glider is denoted shortly by $M_{\bullet} \subseteq \Omega$.
\end{definition} 

Given a ring $R$ one can actually define $FR$-glider representations for any filtration $FR$ of $R$ via so-called $FR$-fragments \cite{CVo1}. In the next subsection we will briefly recall this, however following \cite{HvR}. In this article we will only consider the algebra filtration above and hence simply speak about a {\it glider (representation) of $G$}.

Let $M_{\bullet} \subseteq \Omega_{M}$ and $N_{\bullet} \subseteq \Omega_{N}$ be two glider representations of $G$.
\begin{definition}
 A morphism of glider representations
$f_{\bullet}:\left(M_{\bullet} \subseteq \Omega_{M}\right) \rightarrow\left(N_{\bullet} \subseteq \Omega_{N}\right)$ is given by a $k$-linear map $f: M_{0} \rightarrow N_{0}$ satisfying the following conditions:
\begin{enumerate}
    \item for all $j \in \N$, we have $f\left(M_{-j}\right) \subseteq N_{-j}$,
    \item for all $ i \leq j$, the following diagram commutes:

\begin{displaymath}
\xymatrix{
F_{i}(k[G]) \otimes M_{-j} \ar[r] \ar@<2ex>[d]_{1 \ot f}& M_{-j+i} \ar[d]_{f}\\
F_{i}(k[G]) \otimes N_{-j} \ar[r] & N_{-j+i} \\
}
\end{displaymath}
or, equivalently, for all $r \in F_{i} (k[G])$ and $m \in M_{j},$ we have $f(r m)=r f(m)$
\end{enumerate}
\end{definition}
If $f$ is a monomorphism, then $M_{\bullet} \subseteq \Omega_M$ is called {\it a sub-glider} of $N_{\bullet} \subseteq \Omega_N$.
Note that a glider morphism $f_{\bullet}$ gives rise to a sequence of maps $f_{-i} = f_{|_{M_{-i}}}: M_{-i} \rightarrow N_{-i}$ such that $ f_{-j+i}(\alpha_i m_{-j}) = \alpha_i f_j (m_{-j})$ for all $\alpha_i \in k[G_i], m_j \in M_j$ and $i \leq j$, giving an impression of morphisms of quiver representations. This flavour will be made concrete with the approach of \cite{HvR} recalled in the next subsection. 

\begin{remark}\label{remark over de punt of oneindig}
The map $f: M_0 \rightarrow N_0$ can not necessarily be extended to a $k[G]$-module morphism $f_{\Omega}: \Omega_M \rightarrow \Omega_N$ (see \cite[example 3.16]{HvR}). However, as proven in \cite[proposition 4.9]{HvR}, this can be done in a canonical way after possibly changing $M_{\bullet} \subseteq \Omega_{M}$ to an isomorphic copy (more precisely, up to re-choosing $\Omega_{M}$).
\end{remark}

\begin{lremark*}
In the original definition of a glider representation the module $\Omega$ was not included in the data and only its existence was assumed. This new point of view was introduced in \cite{CVo4} and in \cite{HvR} it is in fact proven that both approaches are equivalent. However, the definition above avoids certain (categorical) issues at the level of morphisms, explaining the current choice.
\end{lremark*}

Glider representations of $G$ form a category which we denote by $\Glid \big( F(k[G]) \big)$. Furthermore it inherits a monoidal structure from $\text{Mod}(k[G])$.
\begin{definition}\label{definitie tensor product gliders}
Let $M_{\bullet} \subseteq \Omega_{M}$ and $N_{\bullet} \subseteq \Omega_{N}$ be $F(k[G])$-gliders. Their tensor product is the descending chain
$$ \ldots \subseteq M_{-1} \ot_k N_{-1} \subseteq M_0 \ot_k N_0 \subseteq \Omega_M \ot_k \Omega_N$$
where $k[G]$ acts via the comultiplication map $\Delta(g)= g \ot g$ of $k[G]$.
\end{definition}

Finally, $\Glid \big( F(k[G]) \big)$ also has {\it direct sums} by setting 
\begin{equation}\label{def direct sum in glid}
\big( (M_{\bullet}\subseteq \Omega_M ) \op  (N_{\bullet}\subseteq \Omega_N)\big)_i = M_i \op N_i
\end{equation}
for all $i \leq 0$ and with overlying module $\Omega_M \op \Omega_N$. With the above the category $\Glid \big( F(k[G]) \big)$ is a monoidal additive category \cite[prop. 9.5.]{HvR}. Unfortunately, in contrast to $\Mod(KG)$, the category is not abelian \cite[section 5.4.]{HvR}.

\begin{lremark*}
The categorical direct sum defined in (\ref{def direct sum in glid}) has usually been called in the literature the strict fragment direct sum and is then seen as an extra property of the 'sum of two gliders' (e.g. see \cite[pg 53]{CvoBook}). By the latter is meant the smallest glider $P_{\bullet} \subseteq \Omega$ containing $M_{\bullet}\subseteq \Omega_M$ and $N_{\bullet}\subseteq \Omega_N$. The, non-categorical, notion of a weak fragment direct sum also exists.
\end{lremark*}

In this article we will only be interested in {\it Noetherian gliders}. Recall that an object in a category is called Noetherian if any ascending sequence of subobjects of X is stationary.
\begin{proposition}[\cite{HvR}]
Let $(M_{\bullet} \subseteq \Omega_{M}) \in \Glid\big( F(k[G]) \big)$. Then the following are equivalent:
\begin{enumerate}
    \item $M_{\bullet} \subseteq \Omega_{M}$ is Noetherian
    \item the $k[G_0]$-module $\bigoplus_{i \in \Z_{\leq 0}} M_i$ is Noetherian
    \item the $k[G_0]$-module $M_0$ is Noetherian and $M_i = 0$ for $i \ll 0$.
\end{enumerate}
\end{proposition}
In other words, as $k[G]$ is finite dimensional, in our context Noetherian gliders are exactly those with finite total dimension. The full subcategory of Noetherian gliders is denoted by $\glid \big( F(k[G]) \big)$. Note that Noetherian gliders have trivial body, i.e. the $k[G]$-module $B(M) := \cap_{i \leq 0} M_i$ is zero. If $M_t \neq 0$ but $M_{t-1} = 0$, then $|t|$ is called the {\it (essential) length of the glider.} The full subcategory consisting of the gliders of essential length {\it at most} $t$ will be denoted by $\glid_{\leq t} F(k[G])$ and {\it exactly} $t$ by $\glid_{t} F(k[G])$. For such gliders the tail consisting of the subspaces $M_i$ equal to $0$ will be omitted.
\begin{remark*}
It was pointed out in \cite[Page 1480]{CVo1} that one can reduce the study of glider representations to those of finite essential length (even length at most $d$ with $G_d = G$) and zero body. Hence, assuming Noetherian is not a lose of generality for the filtration $F(k[G])$.
\end{remark*}

\subsubsection{The category $\Glid_{\Lambda}(FR)$}\label{subsectie categorical approach}
We now very briefly recall, for general filtered rings, the very recent (equivalent) categorical definition by Henrard-van Roosmalen \cite{HvR} of glider representations. This approach will not be explicitly used in the rest of the paper and hence the reader may opt to immediately go to the next subsection. However, it helps to clarify \Cref{remark over de punt of oneindig}. Moreover, the content of \Cref{subsectie decateg}, \ref{subsectie glider repr ring} and \ref{sectie inductie en restrictie functor} will in fact also be meaningful for the generality of this section (but will be written for $\Glid \big( F(k[G]) \big)$. 

Let $(\Gamma, \leq)$ be an ordered group and $R$ a $\Gamma$-filtered unital ring with filtration denoted $FR$ (i.e. a set $\{F_{\gamma}R \}_{\gamma \in \Gamma}$ of additive subgroups such that $1_R \in F_eR,\, F_{\alpha}R \subseteq F_{\beta}R$ if $\alpha \leq \beta$ and $(F_{\alpha}R)(F_{\beta}R) \subseteq F_{\alpha\beta}R$).

\begin{definition*}
Let $\Lambda \subseteq \Gamma$ be any subset. The extended $\Lambda$-filtered companion category $\overline{\mathcal{F}}_{\Lambda} R$ of $FR$ is defined on objects by $\operatorname{Ob}\left(\overline{\mathcal{F}}_{\Lambda} R\right)=\Lambda \coprod\{\infty\}$ and the morphisms are given by
$$
\operatorname{Hom}_{\overline{\mathcal{F}}_{\Lambda} R}(\alpha, \beta)=\left\{\begin{array}{ll}
F_{\beta\alpha^{-1}}R & \alpha \leq \beta \in \Lambda \\
R & \beta=\infty \\
0 & \text{otherwise }
\end{array}\right.
$$
The composition is given by the multiplication in $R$.
\end{definition*}
For each $\alpha, \beta \in \mathrm{Ob}\left(\overline{\mathcal{F}}_{\Lambda}\right)$ such that $\alpha \leq \beta$ or $\beta=\infty$ there is an element $1_{R} \in \operatorname{Hom}(\alpha, \beta)$ which is denoted by $1_{\alpha, \beta}$. Recall that a module over the (small $k$-linear) category $\overline{\mathcal{F}}_{\Lambda} R$ is a covariant $k$-linear functor to $\vect_k$. 

\begin{definition*}
The category $\Preglid_{\Lambda}FR$ of $FR$-pregliders is the full additive subcategory of $\operatorname{Mod}\left(\overline{\mathcal{F}}_{\Lambda} R\right)$ given
by those $M \in \operatorname{Mod}\left(\overline{\mathcal{F}}_{\Lambda} R\right)$ for which the map $M\left(1_{\alpha, \beta}\right): M(\alpha) \rightarrow M(\beta)$ is a monomorphism for all $\alpha \leq \beta$.
\end{definition*}

Consider now the set $\Sigma \subseteq \operatorname{Mor}(\Preglid_{\Lambda}FR)$ of all morphisms $f:N \rightarrow M$ such that for all $\lambda \in \Lambda$ the map $f_{\lambda}: N(\lambda) \rightarrow M(\lambda)$ is an isomorphism (hence no condition is imposed on $f_{\infty}$). In \cite[prop. 4.8.]{HvR} it is shown that the category $\Glid_{\Lambda}FR$, in the spirit of \Cref{subsectie glider repr} or \cite{CvoBook}, can be defined as the localization $(\Preglid_{\Lambda}FR )[\Sigma^{-1}]$. Using this it was obtained in \cite[Theorem 1.2. \& 6.5. and prop. 9.5]{HvR} that, despite not being abelian, the categories above still have a rich categorical structure.

\begin{proposition}[\cite{HvR}]\label{Glid as cateogry}
The category $\Glid_{\Lambda}FR$ is a complete and cocomplete deflation quasi-abelian category which is moreover monoidal if $R$ is a bialgebra. Furthermore, the full subcategory of Noetherian $FR$-gliders $\glid_{\Lambda}FR$ is a Serre subcategory and hence inherits these properties.
\end{proposition}

\begin{example}
In our setting, $R= k[G], \Gamma = \Z, \Lambda = \Z_{\leq 0}$ and $FR$ is the algebra filtration (\ref{onze filtratie}). Note that by taking $\Lambda= \{ 0, \ldots, -t \}$, then $\glid_{\Lambda}FR = \glid_{\leq t} F(k[G])$ is exactly the category of gliders of length at most $t$. With other choices one can also recover the classical notions of filtered $FR$-modules and $\Z$-filtered modules in the sense of \cite{ScSc}, resp. \cite{NvOBook} (see \cite[example 3.5 \& 3.15]{HvR}).
\end{example}
Note that the description of $\Glid\big(F(k[G]) \big)$ as localization of the category of pregliders indeed clarifies \Cref{remark over de punt of oneindig}. More precisely, due to this we see that two gliders $M_{\bullet} \subseteq \Omega_1$ and $M_{\bullet} \subseteq \Omega_2$ (where $\Omega_i$ in fact corresponds with the image of $\infty$ of a module in $\operatorname{Mod}\left(\overline{\mathcal{F}}_{\Lambda} R\right)$) are in fact isomorphic. 

\noindent {\it Notation.} In the sequel we will no longer emphasize the ambient module $\Omega$ and denote gliders simply by $M_{\bullet}$.

\subsubsection{Decategorifications : $\operatorname{K_0^{split}}(\cdot)$ and quotients}\label{subsectie decateg} We again consider the setting of \Cref{subsectie glider repr}, i.e. a chain of finite groups $G_0 \leq G_1 \leq \ldots \leq G_d = G$ and $k[G]$ is equipped with the filtration (\ref{onze filtratie}). Such as for every additive category, we may now consider the split Grothendieck group of the categories $\glid_{\leq t} F(k[G])$. Recall,

 \begin{definition} \label{definitie split grothendieck}
The {\it split Grothendieck group} $\operatorname{K_0^{split}}\big( \glid_{\leq t} F(k[G]) \big)$ is the quotient of the free abelian group generated by the isomorphism classes $[M_{\bullet}]$ of gliders $M_{\bullet} \in \glid_{\leq t} F(k[G])$ with the additive subgroup generated by the elements
$$ \left[ M_{\bullet} \oplus N_{\bullet}\right] - [M_{\bullet}] - [N_{\bullet}] $$
corresponding to the split exact sequences. Furthermore, the monoidal structure of $\glid_{\leq t} F(k[G])$ induces a multiplication: $[M_{\bullet}] \cdot [N_{\bullet}] = [M_{\bullet} \otimes N_{\bullet}].$
 \end{definition}
 
Note that $\Glid\, F(k[G])$ inherits from $\Mod(k[G])$ also a symmetry, turning it into a symmetric monoidal category. In particular \Cref{Glid as cateogry} yields that $\operatorname{K_0^{split}}\big( \glid_{\leq t} F(k[G]) \big)$  is a {\it commutative} unital ring with unit element $T^t_{\bullet}$ defined as $T_i = T$ for $0 \leq i \leq t$ and $0$ otherwise where $T$ denotes the trivial $G$-representation. 

Next we wish to obtain a ring representing the gliders of exactly length $t$. For this, note that we have the filtration
$$ \ldots \subseteq \glid_{\leq n-1} F(k[G]) \subseteq \glid_{\leq n} F(k[G]) \subseteq \glid\, F(k[G])$$
of full subcategories. Furthermore, $\glid_{\leq t-1} F(k[G])$ is an additive tensor ideal in \newline$\glid_{\leq t} F(k[G])$. 
In terms of $\operatorname{K_0^{split}}\big( \glid_{\leq t} F(k[G]) \big)$ this boils down to considering the additive subgroup $\mathcal{G}_{t-1}$ generated by the gliders of essential length at most $t-1$ which is in fact an ideal in $\operatorname{K_0^{split}}\big( \glid_{\leq t} F(k[G]) \big)$. The resulting ring will be denoted by
 \begin{equation}
 \operatorname{K^{\op}}(F(k[G]),t) = \operatorname{K_0^{split}}\big( \glid_{\leq t} F(k[G]) \big) / \mathcal{G}_{t-1}.
 \end{equation}
\begin{remark}
Since $\glid_{\leq t} F(k[G])$ is not a semisimple category the split Grothendieck ring does not coincide with the Grothendieck ring. In fact, due to \cite[Theorem 9.10.]{HvR} the Grothendieck ring $\operatorname{K_0}\big( \glid_{\leq t} F(k[G]) \big)$ is isomorphic, as ring, to $\Z^t$. Therefore it does not reflect $G$, explaining the choice for the split Grothendieck ring. The main purpose of this paper is to show that $\operatorname{K_0^{split}}\big( \glid_{\leq t} F(k[G])\big)$ and particularly $\operatorname{K^{\op}}(F(k[G]),t)$ do retain surprisingly much group theoretical information.
\end{remark}

\subsection{Irreducible gliders and glider representation rings}\label{sectie gliders length 1} 
In this section we will introduce the glider representation and character ring, these are inherent in \cite{CVo4} but there it was only defined in a particular case. The former object arises as the subring of $K^{\op}(F(k[G]),t)$ generated by the irreducible gliders whose definition we will first recall. 
\subsubsection{Glider representation rings}\label{subsectie glider repr ring}
Let $M_{\bullet} \in\glid_{t} F(k[G])$ and $N_{\bullet}$ a subglider. Recall that $N_{\bullet}$ is called {\it a trivial subglider} if $$N_{-i}= 0, \text{ but } M_{-i}\neq 0 \text{ for some } 0 \leq i \leq t.$$
In other words, if $N_{\bullet}$ has strictly smaller length. The glider $M_{\bullet}$ is called {\it irreducible} if all its subgliders are trivial.

For example, the irreducible gliders of length $0$ are those with $M_0$ a simple $k[G_0]$-module and $M_{-i}=0$ for $i > 0$. Those of length $1$ are described in \Cref{sectie R(G) as functor} (see (\ref{cyclicity condition for irr})). In \Cref{R_d as split Grothendieck ring} we will explain what is the categorical interpretation of these objects.
\begin{remark}
Irreducible gliders were introduced in \cite{CVo1} (see also \cite[section 2.1.]{CvoBook}). Therein the definition of an irreducible glider involves three conditions $T_1-T_3$. The condition above is $T_2$, whereas $T_1$ is only of application for non-Noetherian gliders. Condition $T_3$ says that if $N_{-i} = M_{\alpha(-i)}$ for all $i \in \N$ where $\alpha$ is some monotone decreasing function $\alpha: \Z_{\leq 0} \rightarrow \Z_{\leq 0}$ such that $\alpha(-i) + j  \leq \alpha(i+j)$ for all $0 \leq j \leq i$ (i.e. no repetition). However if $M_{\bullet}$ is Noetherian (hence finite length), then it easy to prove that the only way to make a strict subglider of type $T_3$ is to be of type $T_1$. As we are working in $\glid_{t} F(k[G])$ our definition is thus the classical one.
\end{remark}

We denote by $\Irr_n(F(k[G]))$ the set of irreducible gliders of (essential) length exactly $n$.

\begin{definition}\label{def glider repr ring}
The {\it glider representation ring} of length $t$ of the filtration $F(k[G])$, denoted $R_{t}(F(k[G]))$, is the subring of $\operatorname{K_0^{split}}\big( \glid_{\leq t} F(k[G]) \big)$ generated by $\bigcup_{1 \leq n \leq t} \Irr_n(F(k[G]))$. The image of $R_{t}(F(k[G]))$ in $\operatorname{K^{\op}}(F(k[G]),t)$ will be called the {\it reduced glider representation ring} of length $t$ and denoted by $\RRepR_{t}(F(k[G]))$.
\end{definition}

Note that $\RRepR_{t}(F(k[G]))$ is generated by $\Irr_t(F(k[G]))$ and that every class $[M_{\bullet}] \in \RRepR_{t}(F(k[G]))$ contains at most one irreducible representant. A crucial property of irreducible gliders is that they are somehow cyclicly generated (by the lowest non-trivial subspace), see \cite[Lemma 3.1.4.]{CvoBook} or \cite[Lemma 2.5.]{CVo1} 

\begin{proposition}[\cite{CvoBook}]\label{irr dan cyclic}
Let $0 \leq i \leq d$ and $M_{\bullet} \in \Irr_t(F(k[G]))$. Then,
\begin{equation}\label{cyclic condition}
M_{-i} = F_{t-i}(k[G]) M_{-t}= k[G_{t-i}]M_{-t}
\end{equation}
for all $0 \leq i \leq t$. 
\end{proposition}

We will call a glider satisfying (\ref{cyclic condition}) {\it cyclic}. A cyclic glider $M_{\bullet}$ (of length $t$) can be seen as a tuple $(M_0,M_{-t})$ such that $M_0 \in \Mod(k[G_t])$ and $M_{-t}$ is a $k[G_0]$-submodule of $\Res_{G_0}^{G_t}(M_0)$ such that $k[G_t] M_{-t} = M_0$. As explained in \Cref{voorbeeld glider of module type}, irreducible gliders will then correspond with those where $M_{-t}$ is simple.

In fact the previous proposition is not inherent to the filtration $F(k[G])$ but holds for any irreducible $M_{\bullet} \in \glid_{\Lambda} FR$ and any filtration $FR$ as in \Cref{subsectie categorical approach}. Furthermore, interpreting length $t$ as choosing $\Lambda = \{1, \ldots, t\}$, all the constructions from \Cref{subsectie decateg} and this section also identically go through in that generality, {\it yielding the (reduced) glider representation rings $\RepR_t(FR)$ and $\RRepR_t(FR)$.}\vspace{0,1cm}

{\bf Notational conventions.} Usually the chain $G_0 \leq G_1 \leq \ldots \leq G_d = G$, the associated algebra filtration $F(k[G])$ and the ground-field $k$ will be clear from the context and therefore we will often use the abbreviated notations $\RepR_t(\widetilde{G})$ and $\RRepR_t(\widetilde{G})$.

\subsubsection{Glider character ring}

 In \cite{CVo4} character theory for $F(\mathbb{C}[G])$-gliders was introduced. We will now introduce the (reduced) glider character ring over any field $k$ with $\Char(k)=0$. This section will not be used in the remainder of the paper and has been included because it nicely complements the theoretical framework introduced in the previous sections. However the reader interested in the structural properties of $\RRepR_t(\wt{G})$ may decide to skip this section.\vspace{0,1cm}
 
 Recall that we have fixed a chain of subgroups $G_0 \leq \cdots \leq G_d=G$. Note that for $M_{\bullet} \in \glid_{\leq d} F(k[G])$ we obtain $k[G_i]$-modules $G_i M_{-j}$ with associated character $\chi_{i,j}$ for any $0 \leq i \leq j$. 

  \begin{definition}
 Let $M_{\bullet} \in \glid_{\leq d} F(k[G])$. Then the associated {\it glider character} is the map $\chi_{M_{\bullet}}: G \rightarrow k^{n}$ with $n = \frac{(d+1)(d+2)}{2}$ which sends $g \in G_i \setminus G_{i-1}$ to
 $$\chi_{M_{\bullet}}(g) = \left( \begin{matrix}
 0 & 	\ldots  & 0  & 	0 & 0 & \ldots & 0 & 0 \\
    & \ddots & \vdots & \vdots & \vdots & \ddots & \vdots & \vdots \\
    &          & 0 & 0 &0  & \ldots & 0 & 0 \\
    & 		&    &\chi_{i,i}(g) &\chi_{i, i+1}(g) & \ldots &\chi_{i,d-1}(g) & \chi_{i,d}(g) \\
    & 		&    &		  & \chi_{i+1, i+1}(g) & \ldots & \chi_{i+1, d-1}(g) & \chi_{i+1, d}(g)\\
    &		&    &		  &			       & \ddots & \vdots & \vdots \\
    &&		&    & 	 	  &			       &\chi_{d-1, d-1}(g) & \chi_{d-1, d}(g)\\
    && 		&    &		  &			       & 			    & \chi_{d, d}(g)
     \end{matrix} \right)$$
 \end{definition}
 
 The image has been written in matrix form $\chi_{M_{\bullet}}(g)_{i,j}= \chi_{i,j}(g)$, however in fact it truly lives inside $k^{n}$. Note that if $g_1,g_2 \in G_i \setminus G_{i-1}$, then $\chi_{M_{\bullet}}(g_1) = \chi_{M_{\bullet}}(g_2)$ if and only if $h^{-1}g_1h = g_2$ for some $h \in G$. Hence it is an example of a {\it glider class function}. Recall that these are the maps from $G$ to $k^n$ that are constant on $\mathcal{C}_G(g) \cap G_i\setminus G_{i-1}$ for $g \in G_i\setminus G_{i-1}$ and all $0 \leq i \leq d$. The set of glider class functions, denoted $\mathcal{A}(\widetilde{G})$, also carries the structure of a $k$-vector space via componentwise addition and $\lambda \in k$ acts via pointwise multiplication with the function $c_{\lambda}(g)_{h,l} = \lambda$ if $i \leq h \leq l$ and $0$ otherwise, where $g\in G_i \setminus G_{i-1}$ (recall that the elements are tuples in $k^n$, hence the scalar multiplication is the componentwise one in $k^n$ and not matrix multiplication).

\begin{definition} 
Let $t \leq d$ and $$\text{ch}_{t}: \glid_{\leq t} F(k[G]) \rightarrow \mathcal{A}(\widetilde{G}) : M_{\bullet} \mapsto \chi_{M_{\bullet}} $$ 
 be the $k$-linear map sending a glider to its character. Then $\text{Im}(\text{ch}_{t})$ is a ring and the subring generated by $\text{ch}_{t}\big(\bigcup_{0 \leq i \leq t} \Irr_i(F(k[G]))\big)$ is called the {\it glider character ring} of length $t$ over $F(k[G])$ and is denoted by $\text{ch}_{t}(F(k[G]))$.  Furthermore, 
$$ \overline{\text{ch}}_{t}(F(k[G])) = \text{ch}_{t}(F(k[G])) / \text{ch}_{t-1}(F(k[G])) $$
  is called the {\it reduced glider character ring} of length $t$.
 \end{definition}
Again, when the context is clear we will use the abbreviations $\text{ch}_t(\widetilde{G})$ and $\overline{\text{ch}}_t(\widetilde{G})$. A first important difference with classical representation theory is that the map $\text{ch}_t$ is not injective. Indeed, slightly reformulated \cite[Proposition 3.1]{CVo4} tells us the following.
 
 \begin{proposition}[\cite{CVo4}]\label{prop wat karakter geeft}
 Let $M_{\bullet}, N_{\bullet} \in \glid_{\leq d} F(k[G])$ be irreducible gliders. Then the glider character $\chi_{M_{\bullet}}$ uniquely determines the $k[G_i]$-modules $G_i M_{-j}$ for all $0 \leq i \leq j$ except for $(i,j)= (0,d)$.
 \end{proposition}

\begin{lremark*}
 In \cite{CVo2, CVo4} the authors introduced 'generalized characters'  and a ring which they call the 'generalized character ring' for the first time. In the recent monograph \cite[Chapter 5]{CvoBook} the new terminology 'glider characters' and 'glider representation ring' are coined for these objects. The latter is furthermore denoted by $R(G_0 < G_1 < \ldots < G_d)$, or $R(\widetilde{G})$ in short. However the approach in loc.cit. is less general and hence differs from ours. Nevertheless, over a field $k$ of characteristic $0$, their 'glider representation ring $R(\widetilde{G})$' is isomorphic to $\RRepR_{d}(\widetilde{G})$, the reduced glider representation ring of length $d$ in our sense. 
\end{lremark*}

\section{Various classical functors on \texorpdfstring{$\glid_{\leq d}(F(k[G]))$}{} and the case of length \texorpdfstring{$1$}{}}\label{sectie 3 over functors en andere realisatie}
For this section let $G_{\bullet} \, = \,  (G_0 \leq G_1 \leq \ldots \leq G_d = G)$ be a chain of groups with the associated algebra filtration defined in (\ref{onze filtratie}).
The aim of this part is to construct various functors on a new full subcategory of $\glid_{\leq d} F(k[G])$. More concretely in \Cref{sectie inductie en restrictie functor} we introduce so-called gliders of module-type and subsequently define on them the analogue of the induction, restriction and (co-)invariants functors and provide their basic properties. Crucial for the remainder of the paper is \Cref{ind and res induce ring map} which shows that for gliders of length $1$ the induction functor becomes a ring morphism on $\RRepR_1(\wt{G})$. This is in sharp contrast to the classical setting of $\Mod(k[G])$ and its Grothendieck ring. 

Besides, in \Cref{R_d as split Grothendieck ring} we prove that for any length the category of gliders of module type has nice properties and the reduced glider representation ring embeds in the split Grothendieck ring of (a quotient) of it. This allows us to obtain a basis of $\RRepR_t(\wt{G})$.

\subsection{Induction, restriction and (co-)invariants functors for gliders}\label{sectie inductie en restrictie functor}\hspace{0,1cm}\newline\vspace{-0,3cm}

As indicated by the definition of the glider representation ring and \Cref{irr dan cyclic}, especially for gliders of length $1$ (cf. (\ref{decomposition glider length 1 into cyclic})), cyclic gliders and more generally gliders with $M_{-i}$ a $k[G_{t-i}]$-module play a central role. Such gliders will be called of {\it module-type}. Therefore we introduce the full subcategory \vspace{0,1cm}
$$\glid_{\leq d}^m F(k[G]) = \{ M_{\bullet} \in \glid_{t} F(k[G])\mid 0\leq t \leq d,\, M_{-i} \in \Mod(k[G_{t-i}]) \text{ for all } 0 \leq i \leq t\}.\vspace{0,1cm}$$

\begin{example}\label{voorbeeld glider of module type}
\begin{enumerate}
    \item[(i)] A cyclic glider $M_{\bullet}$ is of module-type since by definition $M_{-i}= k[G_{t-i}]M_{-t} \in \Mod(k[G_{t-i}])$ for all $i$. 
    \item[(ii)] If $M_{\bullet} \in \glid_{t} F(k[G])$ and $N$ is a non-trivial $k[G_0]$-submodule of $M_{-t}$, then we can construct the {\it cyclic glider $C_{\bullet}(N)$ generated by $N$} of the same length: $C_{-i}$ is equal to the $k[G_{t-i}]$-submodule of $k[G_{t-i}]M_{-t} \subseteq M_{-i}$ generated by $N$. In case of the choice $N= M_{-t}$ we write $C^M_{\bullet}$ and speak about the {\it canonical cyclic subglider}. 
    \item[(iii)] $M_{\bullet}$ is an irreducible glider if and only if it is cyclic and $M_{-t}$ is a simple $k[G_0]$-module. Indeed, otherwise for any non-trivial $k[G_0]$-submodule $N$ we have the subglider $C_{\bullet}(N)$.
\end{enumerate}
\end{example}

We will now introduce the analogue of the restriction, induction and (co-)invariants functors. To start, let $H$ be a group and $Q$ a normal subgroup. For a $k[H]$-module $M$ denote by $M^Q$ the $Q$-invariants, by $M_Q = M/\langle g.m - m \mid g \in Q \rangle$ the co-invariants and $\pi: M \rightarrow M_Q$ the quotient map. Note that due to the normality of $Q$, $M_Q$ and $M^Q$ are $k[H]$-modules on which $Q$ acts trivially and hence are also $k[H/Q]$-modules. Furthermore, for a morphism $\varphi: H \rightarrow G$ and $Q = \ker(\varphi)$ we consider $k[G]$ canonically as a $k[H/Q]$-module through the identification $H/Q \cong \varphi(H)$. If $M$ is a $k[\varphi(H)]$-module, we will denote the inflation to $H$ by $\Inf^H_{\varphi(H)}(V)$ (i.e. the $k[H]$-module where the $H$-action is induced through the identification $H/Q \cong \varphi(H)$ and the projection $H \rightarrow H/Q$).

Now consider a second chain $H_{\bullet} = (H_0 \leq \cdots \leq H_d=H)$ and denote by $\varphi_{\bullet} : H_{\bullet} \rightarrow G_{\bullet}$ a morphism $\varphi: H \rightarrow G$ such that $\varphi(H_i) \subseteq G_i$ for all $0 \leq i \leq d$. For $0 \leq t \leq d$ we use the notation $H^t_{\bullet}$ for the chain $H_0 \leq \cdots \leq H_t$. Moreover, the glider $kG^t_{\bullet}$ defined by $(kG^t_{\bullet})_{-i} = k[G_{t-i}]$ for $0 \leq i \leq t$ and $(kG^t_{\bullet})_{-i} = 0$ for $i >t$ will be called the {\it regular glider} of $G^t_{\bullet}$.

\begin{definition}[Restriction, induction and (co)-invariants on objects]\label{the various functors}\hspace{0,1cm}\newline
Let $0 \leq t \leq d$, $M_{\bullet} \in \glid_{t}^m F(k[H]), M'_{\bullet} \in \glid_{t}^m F(k[G]), Q \triangleleft H_t$ and $\varphi_{\bullet}$ as above. Then,
\begin{enumerate}
    \item[$\underline{(\cdot)^Q}$:] $(M_{\bullet})^Q$ is the glider with $(M^Q)_0 = (M_0)^Q$ and $(M^Q)_{-i} = \{ \sum_{g\in Q} gm \mid m \in M_{-i} \}$ for $1 \leq i $,\vspace{0,1cm}
    \item[$\underline{(\cdot)_Q}$:] $(M_{\bullet})_Q$ is the glider with $(M_Q)_0 = (M_0)_Q$ and $(M_Q)_{-i} = \pi(M_{-i})$ for $1 \leq i$, \vspace{0,1cm}
\end{enumerate}
which are called respectively the glider of {\it invariants} and {\it co-invariants}. Furthermore, for $Q = \ker(\varphi)\cap H_t$,  $\IInd_H^G(M_{\bullet})$ and $\RRes_H^G(M'_{\bullet})$ are the gliders with \vspace{0,1cm}
\begin{enumerate}
    \item[$\underline{\mc{I}nd}$:] $(\mc{I}nd_H^G(M_{\bullet}))_{-i} = (M_Q)_{-i} \ot_{k[\varphi(H_{t})]} (kG^t_{\bullet})_{-i} = \pi(M_{-i}) \ot_{k[\varphi(H_{t-i})]} k[G_{t}],$ for $0 \leq i \leq t$ and zero for $i >t$. Compactly,
     $$\IInd_H^G(M_{\bullet}) = (M_{\bullet})_Q \ot_{kH_t} kG^t_{\bullet}.$$
    \item[$\underline{\mc{R}es}$:] $(\RRes_H^G(M'_{\bullet}))_{-i} = \Inf^{H_{t-i}}_{\varphi(H_{t-i})} \Res^{G_{t-i}}_{\varphi(H_{t-i})}(M'_{-i})$ for $0 \leq i \leq t$ and zero for $i >t$.\vspace{0,1cm}
\end{enumerate}
and are called {\it induction}, respectively {\it restriction} of gliders.
\end{definition}
 
\begin{example*}\hspace{-0,2cm}
\begin{itemize}
    \item If $\varphi$ is a monomorphism then the expressions are the more intuitive ones of $\IInd_H^G(M_{\bullet}) = M_{\bullet} \ot_{kH_t} kG^t_{\bullet}$ and $(\RRes_H^G(M'_{\bullet}))_{-i} = \Res^{G_{t-i}}_{H_{t-i}}(M'_{-i})$. 
    \item Suppose we have a map $\varphi_{\bullet}: (1 \leq G) \rightarrow (1 \leq G/G')$. This is nothing else than a group homomorphism $\varphi: G \rightarrow G/G'$. If $M \in \Irr(k[G])$ with $\dim_k M \geq 2$. For $0\neq m \in M$ one has that $M= k[G]m$. As $M$ is not $1$-dimensional, $G'$ acts non-trivially on $M$ and hence on $m$. We now see that $M^Q=0= M_Q$. In particular, as one would wish, $\IInd_G^{G/G'}(M_{\bullet})=0$ for any glider $M_{\bullet} \in \glid_d
   ^m F(k[G])$ with $M_0 =M$.
\end{itemize}
\end{example*}

Now also consider $N_{\bullet} \in \glid_{l}^m F(k[H])$, with $l \leq d$, and a non-zero glider morphism $f_{\bullet}: M_{\bullet} \rightarrow N_{\bullet}$. Denote
$$ n= \min \{l, t \} \text{ and } m = \max \{l,t \}.$$
Furthermore, for gliders of module-type the family of maps $\restr{f}{M_{-i}}:M_{-i} \rightarrow N_{-i}$ induced by $f_{\bullet}$ are in fact $k[H_{n-i}]$-module maps.

Let $Q_t \triangleleft H_t$ and $Q_l \triangleleft H_l$ such that $Q_n \subseteq Q_m$. It is easy to see that $f: M_0 \rightarrow N_0$ induces $k[H_n]$-module maps
$$f^{t,n}: (M_0)^{Q_t} \rightarrow (N_0)^{Q_n} \text{ and } f_{n,l}: (M_0)_{Q_n} \rightarrow (N_0)_{Q_l}.$$
Next define the $k[H_n]$-module map 
$$\psi^{n,m}: (N_0)^{Q_n} \rightarrow (N_0)^{Q_m}: x \mapsto \sum_{g \in \mc{T}} g x$$
for a transversal $\mc{T}$ of $Q_n$ in $Q_m$. We also need 
$$\phi_{m,n}: (M_0)_{Q_m} \rightarrow (M_0)_{Q_n}: x \mapsto \sum_{g\in Q_m} g x$$
which is in fact the composition of three maps. The first being the map from $(M_0)_{Q_m}$ to $(M_0)^{Q_m}$ which sends $x$ on $\sum_{g\in Q_m} g x$, followed by the natural embedding in $(M_0)^{Q_n}$ and the natural projection on $(M_0)_{Q_n}$. Finally, define $\iota_{i,j}: k[G_i] \rightarrow k[G_j]$ by $\iota_{i,j}(g) = g$ if $g \in G_j$ and $0$ otherwise (in particular if $i \leq j$, then $\iota_{i,j}$ is simply the embedding). We now extend the constructions from \Cref{the various functors} in the natural way to morphisms.
\begin{definition}[Restriction, induction and (co)-invariants on morphisms]\hspace{0,1cm}\newline
With notations as above and also $M'_{\bullet}, N'_{\bullet} \in \glid_{\leq d}^m F(k[G])$ with $M'_{\bullet}$ of length $t$ and $f': M'_{\bullet} \rightarrow N'_{\bullet}$ a glider morphism, we define the following glider morphisms:
\begin{enumerate}
    \item[$\underline{(\cdot)^Q}$:] $(f_{\bullet})^{Q_t,Q_l}: (M_{\bullet})^{Q_t} \rightarrow (N_{\bullet})^{Q_l}$ is induced by $f^{t,l}$ if $l=\min \{ t,l \}$ and $\psi^{Q_t,Q_l} \circ f^{t,t}$ otherwise.\vspace{0,1cm}
    \item[$\underline{(\cdot)_Q}$:] $(f_{\bullet})_{Q_t,Q_l} :(M_{\bullet})_{Q_t} \rightarrow (N_{\bullet})_{Q_l}$ is induced by $f_{t,l}$ if $t = \min \{ t,l\}$ and $f_{l,l} \circ \phi_{t, l}$ otherwise.\vspace{0,1cm}
    \item[$\underline{\mc{I}nd}$:] $\IInd_H^G(f_{\bullet}) = (f_{\bullet})_{Q_t,Q_l} \ot \iota_{t,l}$ where $Q_t = \ker(\varphi) \cap H_t$ and $Q_l = \ker(\varphi) \cap H_l$.\vspace{0,1Cm}
    \item[$\underline{\mc{R}es}$:] $\RRes_H^G(f'_{\bullet})$ is induced by $\Inf^{H_t}_{\varphi(H_t)} \Res^{G_t}_{\varphi(H_t)}(f')$.
\end{enumerate}
\end{definition}
\noindent It is easily checked that all the above maps indeed yield glider morphisms. If $t=l$ we will simply write $(f_{\bullet})^{Q_t}$ and $(f_{\bullet})_{Q_t}$ and in which case the induction functor takes the familiar form of $\IInd_H^G(f_{\bullet})= (f_{\bullet})_{Q_t} \ot Id_{kG^t_{\bullet}}$. 

If one fixes a normal subgroup $Q$ of $H$, then for any $M_{\bullet}$ one can perform $M_{\bullet}^{Q_t}$ with $Q_t = H_t \cap Q$ where $t$ is the length of the glider. By doing so one obtains an endofunctor of $\glid_{\leq d}^m F(k[H])$ which we denote by $(\cdot)^Q$. Similarly, we have an endofunctor $(\cdot)_Q$. Altogether we have constructed, for any $\varphi_{\bullet}: H_{\bullet} \rightarrow G_{\bullet}$, the following covariant functors\vspace{-0,5cm}
\[
\begin{tikzcd}
	{\glid_{\leq d}^m F(k[H])} \arrow["{(\cdot)^Q,(\cdot)_Q}", loop, distance=5em, in=145, out=215] \arrow[rrr, "{\IInd^G_H(\cdot)}", shift left=1] &  &  & {\glid_{\leq d}^m F(k[G])} \arrow[lll, "{\RRes^G_H(\cdot)}", shift left=1]
\end{tikzcd}\vspace{-0,5cm}
\]
 As indicated by the next proposition, these functors behave as in the classical setting of modules. In particular, induction and restriction again satisfy the Frobenius reciprocity (\ref{Frobenius reciprocity type iso}) and a push-pull type of formula. 

\begin{proposition}\label{ind and res adjoint pair}
The functors $(\cdot)^Q, \, (\cdot)_Q, \, \IInd^G_H(\cdot)$ and $\RRes^G_H(\cdot)$ are additive and the former three preserve cyclic gliders. Furthermore $ \IInd^G_H(\cdot)$ is a left adjoint of $\RRes^G_H(\cdot)$ and satisfies
\begin{equation}\label{Push-pull formula}
\IInd^G_H( \RRes^G_H(N_{\bullet}) \ot M_{\bullet}) \cong N_{\bullet} \ot \IInd^G_H(M_{\bullet})
\end{equation}
for all $M_{\bullet} \in \glid_{\leq d}^m F(k[H])$ and $N_{\bullet} \in \glid_{\leq d}^m F(k[G])$ whenever $\varphi$ is a monomorphism.
\end{proposition}
\begin{proof}
Let $M_{\bullet} \in \glid_{t}^m F(k[H]), N_{\bullet} \in \glid_{l}^m F(k[G])$, with $t,l \leq d$ and  $Q \triangleleft H_t$ arbitrary. Suppose $M_{\bullet}$ is cyclic, i.e. $M_{-i} = k[H_{t-i}]M_{-t}$ for $0 \leq i \leq t$. Obviously $(M_{\bullet})_Q$ is again cyclic. From the definition of the action on a tensor product it is easily checked that $\pi(M_{-t}) \ot_{k[H_t]} k[G_0]$ will indeed generate $\IInd_H^G(M_{\bullet})$. The case of $Q$-invariants follows readily from the observation that for every $0 \leq i \leq t$ the action of $H_{t-i}$ commutes with $\sum_{g\in Q}g(\cdot)$ (due to the normality of $Q$). 

Next, it is easy to see that all the functors are additive. For the adjointness assertion one needs isomorphisms
\begin{equation}\label{Frobenius reciprocity type iso}
\Hom(\IInd^G_H(M_{\bullet}), N_{\bullet}) \cong \Hom(M_{\bullet}, \RRes_H^G(N_{\bullet}))
\end{equation}
for every $M_{\bullet} \in \glid_{\leq d}^m F(k[H])$ and $N_{\bullet} \in \glid_{\leq d}^m F(k[G])$ which are furthermore natural in $M_{\bullet}$ and $N_{\bullet}$. This follows by checking that the well-known bifunctors (e.g. see \cite[section 2.3-2.5]{bookLinc}) realising the adjoint pairs $(\Ind_{H_t}^{G_t}(\cdot), \Res_{H_t}^{G_t}(\cdot))$ and $\big( (\cdot)_Q, \Inf^{H_t}_{H_t/Q_t}(\cdot) \big)$ between $\Mod(k[H_t])$ and $\Mod(k[G_t])$ go through for gliders (i.e. the subspaces and associated actions are preserved). Finally, for (\ref{Push-pull formula}), the injectivity assumption entails that $\IInd
^G_H(\cdot)$ is simply tensoring with the regular representation. For this define the (classical) map
$$(\Res^{G_t}_{H_t} (N_0) \ot M_0) \ot_{k[H_t]} k[G_t] \rightarrow N_0 \ot (M_0 \ot_{k[H_t]} k[G_t] ) $$
by $(n \ot m) \ot x \mapsto nx \ot ( m \ot x)$. It is well-known, e.g. \cite[Theorem 2.2.2.]{bookLinc}, that it is a $k[G_t]$-module isomorphism. Clearly it preserves the necessary subspaces to make it into a glider isomorphism.
\end{proof}

Unfortunately, in general irreducible gliders are not preserved under above functors as will be apparent from the next section. Also, note that $\RRes^G_H(\cdot)$ is a monoidal functor. However, $\IInd^G_H(\cdot)$ is not which stems from the facts that usually $\Ind_H^G(M \ot N) \lneq\Ind_H^G(M) \ot \Ind_H^G(N)$ and $(M\ot N)^Q \gneq M^Q \ot N^Q$. Crucially, as shown in \Cref{ind and res induce ring map}, this problem will disappear by considering the reduced glider representation ring (of length $1$).

\begin{remark}\label{in semisimple case inv and coinv same}
(1) If $\Char(k) \nmid |Q|$ the natural map $\pi: M \rightarrow M_Q$ induces an isomorphism $M^Q \cong M_Q$ which in turn induces a glider isomorphism between $(M_{\bullet})_Q$ and $(M_{\bullet})^Q$. Therefore, in a semisimple setting one could have defined alternatively $\IInd_H^G(M_{\bullet})$ as $(M_{\bullet})^Q \ot_{kH_t} kG_{\bullet}$.

(2) A map $\varphi_{\bullet}: H_{\bullet} \rightarrow G_{\bullet}$ induces a functor between the extended companion categories $\ov{\mc{F}}_{\Lambda} F(k[H])$ and $\ov{\mc{F}}_{\Lambda} F(k[G])$ with $\Lambda= \{0, \leq -d \}$ which are additive and small. As such one has the obvious restriction functor between their module categories and which has a left adjoint, called induction. Our approach above is simply a concrete realization of these functors. 
\end{remark}

\subsection{Reduced glider representation rings versus gliders of module type}\label{R_d as split Grothendieck ring}
Given $0 < t \leq d$, as for $\glid_{\leq t} F(k[G])$, we have the chain of full subcategories 
$$ \ldots \subseteq \glid_{\leq t-1}^m F(k[G]) \subseteq \glid_{\leq t}^m F(k[G]).$$
Interestingly, the gliders of module-type of length at most $t-1$ form a Serre subcategory in those of length at most $t$. This can be proven in similar way as \cite[theorem 6.2]{HvR} or by first remarking that kernels and cokernels for gliders of module type can be described explicitly and are the expected ones.

Now clearly $\glid_{\leq t-1}^m F(k[G])$ is a tensor ideal $\glid_{\leq t}^m F(k[G])$. Consequently, we can form the Serre quotient which we denote by $\GGlid_{t} F(k[G])$. This is still a 'nice' monoidal (additive) category. The next theorem shows that $\RRepR_t(F(k[G])$ embeds nicely in the split Grothedieck ring of the latter, providing a basis. 
First recall that an object $A$ of a pre-abelian category $\mc{C}$, such as $\glid_{\leq t}^m F(k[G]))$ (by \Cref{Glid as cateogry}), is called {\it indecomposable} if $A = B \op C$ implies that $B=A$ or $C=A$. The set of all such objects will be denoted by $\Ind(\mc{C})$.

\begin{theorem}\label{theorem glider repr as decat}
Let $0 \leq t \leq d$. Then, $\glid_{\leq t}^m F(k[G])$ and $\GGlid_{t} F(k[G])$ are Krull-Schmidt symmetric monoidal categories. Furthermore, there exists a monomorphism
$$\RRepR_t(F(k[G])) \hookrightarrow \operatorname{K_0^{split}}\big( \GGlid_{t} F(k[G]) \big) $$
in such a way that $\Irr_t(F(k[G]))$ is sent into $\Ind(\GGlid_{t} F(k[G]))$. In particular, the members of $\Irr_t(F(k[G]))$ are linearly independent. 
\end{theorem}
\begin{proof}
Let $M_{\bullet}\in \glid_{t}^m F(k[G])$.  As all modules are finite dimensional and $M_{\bullet}$ of finite length it can be decomposed in indecomposable objects. Furthermore, the endomorphism ring of every indecomposable object is a local ring. The proof is along the classical lines (e.g. see \cite[section 5]{KSKrause}). More precisely we still have the variant of the Fitting lemma. Recall that for any glider map $f_{\bullet}: M_{\bullet} \rightarrow N_{\bullet}$ the image and kernel are defined as expected: $\Ima(f_{\bullet})$ is the glider defined as $\Ima(f)_{-i}= \Ima(\restr{f}{M_{-i}})$ for all $0 \leq i \leq t$ and $\ker(f_{\bullet})_{-i}= \ker(f) \cap M_{-i}$. 

\noindent {\it Claim }{\bf ('Fitting Lemma'):} Let $M_{\bullet} \in \glid_{\leq t}^m F(k[G])$ and $\psi_{\bullet} \in \End(M_{\bullet})$, then $M_{\bullet} = \Ima(\psi_{\bullet}^n) \op \ker(\psi_{\bullet}^n)$ for $n$ large enough. 

\vspace{0,2cm}\noindent{\it Proof of claim.} Consider now the sequences $\Ima(\psi_{\bullet}) \supseteq \Ima(\psi^2_{\bullet}) \supseteq \cdots$ and $\ker(\psi_{\bullet}) \subseteq \ker(\psi^2_{\bullet}) \subseteq \cdots$. Since $M_{\bullet}$ is Noetherian the both sequences must become stationary. Thus there exists some $n$ such that $\Ima(\psi^{n-1}_{\bullet})=\Ima(\psi^{n}_{\bullet})$ and $\ker(\psi^{n-1}_{\bullet})= \ker(\psi^{n}_{\bullet})$. This entails that $\psi^{n}_{\bullet}: \Ima(\psi^n_{\bullet}) \rightarrow \Ima(\psi^{2n}_{\bullet})$ is an isomorphism. Using this we obtain that $M_{\bullet} = \Ima(\psi^{n}_{\bullet}) \op \ker(\psi^{n}_{\bullet})$.\vspace{0,1cm}

Suppose now that $M_{\bullet} \in \Ind \big( \glid_{\leq t}^m F(k[G])) \big)$. Then it follows directly from the claim that $\psi_{\bullet}$ is either nilpotent or an isomorphism. Consequently, $\End (M_{\bullet})$ is a local ring, as needed. Next, it is easily seen that $\glid_{\leq t}^m F(k[G])$ inherits an additive and symmetric monoidal structure from $\glid_{\leq t} F(k[G])$, hence finishing the first part. Now consider the canonical quotient functor
$$\mc{Q}: \glid_{\leq t}^m F(k[G]) \rightarrow \GGlid_{t} F(k[G])$$
which is additive and monoidal (e.g. see \cite[corollary 1.4.]{Day}). Consequently it induces a ring epimorphism $\operatorname{K_0^{split}}(\mc{Q})$ between their split Grothendieck rings and 
$$\ker \big( \operatorname{K_0^{split}}(\mc{Q}) \big) = \mc{G}_{t-1}\cap \operatorname{K_0^{split}}\big( \glid_{\leq t}^m F(k[G]) \big).$$
Thus $\operatorname{K_0^{split}}(\mc{Q})$ restricted to $\RRepR_t(F(k[G]))$ is indeed a monomorphism. By construction and \Cref{voorbeeld glider of module type}, under $\mc{Q}$ any irreducible glider $M_{\bullet}$ of length $t$ is sent to a simple object of $\GGlid_{t} F(k[G])$. Now recall that for any Krull-Schmidt category $\mc{C}$ the classes of the indecomposable objects form a basis of $\operatorname{K_0^{split}}(\mc{C})$. In particular the elements  $\Irr_t(F(k[G]))$ are also independent and hence form a basis of $\RRepR_t(F(k[G]))$. This finishes the proof.
\end{proof}
It can also be seen that $\glid_{\leq t}^m F(k[G])$ inherits the rest of the structural properties of $\glid_{\leq t} F(k[G])$ mentioned in \Cref{Glid as cateogry}. Due to the explicit description of the kernels and cokernels one can show that $\operatorname{coker(ker)}$ is isomorphic to $\operatorname{ker(coker)}$, hence turning it into an abelian category.

\subsection{Induction as a monoidal functor in case of length \texorpdfstring{$1$}{} }\label{sectie R(G) as functor}

As $\IInd^G_H(\cdot)$ preserves the length of a glider it induces a well-defined additive map from $\operatorname{K^{\op}}(F(k[H]),t)$ to $\operatorname{K^{\op}}(F(k[G]),t)$. However in general it is not multiplicative. Interestingly, as shown in \Cref{ind and res induce ring map} below, when considering gliders of length $1$ this problems vanishes. This will be instrumental in the remainder of the paper.

A glider $M_{\bullet} \in \glid_1 F(k[G])$ simply consists of two $k[G_{0}]$-modules 
$M_{-1} \subseteq M_0$ such that $M_0$ contains the $k[G_1]$-module $k[G_1]M_{-1}$. Therefore we will often write $M_{\bullet}$ more informatively as $(M_{-1} \subseteq M_0)$. In this setting a glider morphism $f: M_{\bullet} \rightarrow N_{\bullet}$ is simply a $k[G_0]$-module map $f: M_0 \rightarrow N_0$ such that $\restr{f}{k[G_1]M_{-1}}$ is a $k[G_1]$-module morphism mapping $M_{-1}$ into $N_{-1}$. Also, following \Cref{voorbeeld glider of module type} $(M_{-1} \subseteq M_0)$ is irreducible exactly when 
\begin{equation}\label{cyclicity condition for irr}
M = k[G_1]M_{-1} \text{ and } M_{-1} \text{ is a simple } k[G_0]\text{-module}.
\end{equation}

Concerning the next theorem, we must mention that we consider all the rings and ring homomorphisms as living in ${\bf Rng}$, the category of rings without necessarily a unit element. Thus ring morphisms are not asked to preserve the unit.
\begin{theorem}\label{ind and res induce ring map}
Let $H_{\bullet}, G_{\bullet}, \varphi_{\bullet}$ be as in \Cref{sectie inductie en restrictie functor} and let $Q = \ker(\varphi) \cap H_1$. Suppose that $\Char(k) \nmid |G_0|, |Q|$. 
Then the map
$$\IInd^G_H: \operatorname{K^{\op}}(F(k[H]),1) \rightarrow \operatorname{K^{\op}}(F(k[G]),1) : [M_{\bullet}] \mapsto [\IInd^G_H(M_{\bullet})]$$ 
is a ring morphism with kernel equal to $\{[M_{\bullet}] \mid (M_0)_Q = 0\}= \{[M_{\bullet}] \mid (M_0)^Q = 0\}$.
\end{theorem}
\begin{proof}
Denote by $\pi: M_0 \rightarrow (M_0)_Q$ the quotient map. As $\IInd^G_H(\cdot)$ preserves the length of a glider it induces a well-defined additive map from $\operatorname{K^{\op}}(F(k[H]),1)$ to $\operatorname{K^{\op}}(F(k[G]),1)$. For the multiplication let $N_{\bullet} \in \glid_{\leq 1}^m F(k[G])$. If $N_{\bullet}$ has length $0$, then $\IInd^G_H(M_{\bullet} \ot N_{\bullet})$ and $\IInd^G_H(M_{\bullet}) \ot \IInd^G_H(N_{\bullet})$ are of length $0$, hence both equal to zero in $\operatorname{K^{\op}}(F(k[G]),1)$. Assume now that $N_{\bullet}$ has length $1$. For this case we need the following two observations where we use the identification between invariants an co-invariants from \Cref{in semisimple case inv and coinv same}.
\begin{enumerate}
    \item \resizebox{0.9 \hsize}{!}{$ \big( (M^Q)_{-1} \ot_{k} (N^{Q})_{-1}\big) \ot_{k[H_1]} k[G_0] \cong \big(  (M^Q)_{-1} \ot_{k[H_1]} k[G_0] \big) \ot_k \big((N^Q)_{-1} \ot_{k[H_1]} k[G_0] \big)  $} \newline are isomorphic as $k[G_0]$-modules via the straightforward mapping $\alpha \ot \beta \ot 1 \mapsto (\alpha \ot 1) \ot (\beta \ot 1)$,
    \item the glider $\big(((M^Q)_{-1} \ot_k (N^Q)_{-1}) \ot_{k[H_1]} k[G_0]\big) \subset (M_0^Q \ot_k N_0^Q) \ot_{k[H_1]} k[G_1]$ is canonically a subglider of $\IInd^G_H(M_{\bullet} \ot N_{\bullet})$. Hereby $(\sum_{g\in Q}gm) \ot (\sum_{g\in Q}gn) \ot \alpha$ is sent to $|Q| (\sum_{q\in Q} q (m\ot n)) \ot \alpha$.
\end{enumerate}
The two points combined imply that 
\begin{displaymath}
\resizebox{1.0 \hsize}{!}{$k[G_1] (\IInd^G_H(M_{\bullet} \ot_k N_{\bullet})_{-1}) \cong k[G_1]((M^Q)_{-1} \ot (N^Q)_{-1} \ot 1)  \cong k[G_1] (\IInd^G_H(M_{\bullet})_{-1} \ot_k \IInd^G_H(N_{\bullet})_{-1})$} 
\end{displaymath}
are isomorphic $k[G_1]$-modules. These isomorphisms directly induce glider isomorphisms between the associated canonical cyclic subgliders (in the sense of \Cref{voorbeeld glider of module type}). Finally note that a glider $M_{\bullet} \in \glid_{1}^m F(k[G])$ is the direct sum of $(M_{-1} \subseteq k[G_1]M_{-1})$ and a glider of length $0$. Hence altogether we obtain that $[\IInd^G_H(M_{\bullet} \ot N_{\bullet})] = [\IInd^G_H(M_{\bullet}) \ot \IInd^G_H(N_{\bullet})]$ in $\operatorname{K^{\op}}(F(k[G]),1)$, as needed. 

Finally, suppose that $0\neq [M_{\bullet}] \in \ker (\IInd^G_H)$. Since this class is non-zero every representative $M_{\bullet}$ has length $1$. Now it is easy to see that $(M^Q)_{-1} \ot k[G_0]$ is zero if and only if $(M^Q)_{-1}$ is. However, it is easily seen that this only happens if $(M_0)^Q =0$. Using that $(M_0)^Q \cong (M_0)_Q$, the latter finishes the proof.
\end{proof}

To end this section we point out, as used in the proof, that as soon as $k[G_0]$ is semisimple a glider $M_{\bullet}$ can conveniently be decomposed as 
\begin{equation}\label{decomposition glider length 1 into cyclic}
M_{\bullet} = (M_{-1} \subseteq k[G_1]M_{-1}) \op (0 \subseteq V \cap M_0)
\end{equation}
for some $k[G_1]$-submodule $V$ of $k[G_1]M_0$ (recall that by definition a glider $M_{\bullet}$ goes along with a $k[G]$-module $\Omega_M$, hence speaking about $k[G_1]M_0$ and such a $V$ makes sense). In particular, every class in $\operatorname{K^{\op}}(F(k[G]),1)$ has a cyclic representative. However, more importantly, the decomposition (\ref{decomposition glider length 1 into cyclic}) is a canonical one in the following sense: if $[M_{\bullet}]$, then there is a unique cyclic glider $N_{\bullet}$ of length $1$ and glider $V_{\bullet}$ of length $0$ such that $[M_{\bullet}] = [N_{\bullet}] + [V_{\bullet}]$. Summarized, if $\Char(k) \nmid |G_0|$,
\begin{equation}\label{ontbinding Grothendieck ring voor length 1}
\operatorname{K_0^{split}}\big( \glid_{\leq 1} F(k[G]) \big) = \{ [M_{\bullet}] \mid M_{\bullet} \in \glid_0 F(k[G]) \} \op \langle [N_{\bullet}] \mid N_{\bullet} \text{ cyclic length }1\rangle
\end{equation}
as $\Z_2$-graded rings, where $\langle \cdot \rangle$ denotes the subring generated by.

Note that $\RepR_1(F(k[G]))$ is strictly smaller than $\operatorname{K_0^{split}}\big( \glid_{\leq 1} F(k[G]) \big)$ since for example the only gliders of length $0$ that $\RepR_1(F(k[G]))$ contains are $\langle \Irr_0 F(k[G_0]) \rangle \subseteq \Mod(k[G_0])$ . Interestingly, $M_{\bullet} \in \Irr_0 F(k[G])$ if and only if $M_0$ is a simple $k[G_0]$-module. In particular $\langle \Irr_0 F(k[G]) \rangle \cong \operatorname{K_0}^{\text{split}}\big(\Mod(k[G_0]/J(k[G_0]))\big)$ which in turn is equal to the Grothendieck group $\operatorname{K_0}(\Mod(k[G_0]/J(k[G_0])))$ due to semisimplicty of $k[G_0]/J(k[G_0])$. Hence using (\ref{ontbinding Grothendieck ring voor length 1}) we can also describe the $\Z_2$-graded ring-isomorphism type of the glider representation ring (of length $1$) in an interesting way:
\begin{equation}
\RepR_1(\wt{G}) \cong \operatorname{K_0}(\Mod(\frac{k[G_0]}{J(k[G_0])})) \op \RRepR_1(\wt{G}).
\end{equation}
Hence due to this description of $\RepR_1(\wt{G})$ we will focus on $\RRepR_1(\wt{G})$. 

\section{A structural result for glider representations rings and development of a toolbox}\label{sectie SES}

From now and until the end of the paper we will work over an algebraically closed field $K$ of characteristic $0$ and consider the chain $\{ e_G \} \lneq G$ with associated algebra filtration \begin{equation}\label{de 'triviale' 1 step filtratie}
F(K[G]):= \, \, K \subsetneq K[G].
\end{equation}
The goal of this section is to obtain a first description of the reduced glider representation ring $\RRepR_1(\wt{G})$ or rather the associated $\Q$-algebra $\Q(\wt{G}) := \Q\ot_{\Z}\RRepR_1(\wt{G})$. For this we will start in \Cref{subsectie length 1} to obtain a parametrization of the irreducible gliders of length $1$ and a realisation of $\RRepR_1(\wt{G})$ as an integral semigroupsring.  These descriptions will be a recurrent tool in the remainder of the paper. Thereafter in \Cref{sectie ab actie en obstructie modulen} we define a  $G^{ab}$-action on $R_1(\wt{G})$. Using this, we construct explicitly three $\Q[G^{ab}]$-submodules $P,R,E$ of $\Q(\wt{G})$ which will be key for the description obtained in our main \Cref{ses}. Herein $\Q\ot_{\Z}\RRepR_1(\wt{G})$ is described in terms of a short exact sequence of $\Q[G^{ab}]$-modules involving the glider representation ring of subnormal subgroups and the modules $P,R,E$. In \Cref{sectie decompositie zonder de obstructies}, under the assumptions $P=0=E$, the exact sequence is refined to a full description of the $\Q[G^{ab}]$-algebra isomorphism type of $\Q(\wt{G})/ J(\Q(\wt{G}))$, the reduced glider representation ring modulo its Jacobson radical. These vanishing assumptions will become more concrete in \Cref{sectie interpretaties} where an interpretation of $P$ and $R$ is obtained in terms of $\rep_{\C}(G)$.

\subsection{Parametrization of the irreducible gliders of length \texorpdfstring{$1$}{}} \label{subsectie length 1}\hspace{1cm}

The set $\Irr_1F(K[G])$ for the one-step filtration in (\ref{de 'triviale' 1 step filtratie}) was fully described in \cite{CVo2} (or \cite[Theorem 4.1.12.]{CvoBook}). Note that for this filtration a glider in $\glid_1^m F(K[G])$ consists of a $K[G]$-module $M$ and a linear subspace $M_{-1}$. By (\ref{cyclicity condition for irr}) this glider will be irreducible if and only if $\dim M_{-1} =1$, i.e. $M_{-1} = K\vec{a}$ for some $\vec{a} \in M$, and $M = K[G]M_{-1} = K[G]\vec{a}$.

\begin{theorem}[\cite{CVo2}]\label{irred}
Let $G$ be a finite group, $K$ an algebraically closed field with $\Char(K)=0$ and let $\{V_1,\ldots,V_q\}$ be a full set of irreducible $G$-representations. Then $M_{\bullet} \in \Irr_1F(K[G])$ if and only if it is of the form
$$ \left( K\vec{a} \subseteq M = \bigoplus_{i=1}^q V_i^{\oplus m_i} \right)$$
with $\vec{a} = (v_1^1, \cdots, v_{m_1}^1, v_1^2, \cdots, v_{m_2}^2,  \cdots, v_1^q, \cdots, v_{m_q}^q) \in M$ and $v^i_{j} \in V_i$ satisfying
\begin{enumerate}
\item[$(\mathbb{I}_1)$] $ m_i \leq \dim(V_i), $
\item[$(\mathbb{I}_2)$] $ \dim(\Span_K\{ v_1^i,\ldots,v_{m_i}^i\}) = m_i, $
\end{enumerate}
for all $1 \leq i \leq q$.
\end{theorem}

\begin{remark*} \hspace{-0,2cm}
\begin{enumerate}
\item For each $1 \leq j \leq m_i$, the element $v_j^{i}$ lives in a different copy of  $V_{i}$ and hence condition $(\mathbb{I}_2)$ may look redundant, however in condition $(\mathbb{I}_2)$ the different $v_j^{i}$ and the subspace generated by these are viewed inside one single copy. More formally: denote the $m_i$ different copies of $V_{i}$ by $V_{i}(1),\ldots, V_{i}(m_i)$ and fix an isomorphism $\varphi_{j}: V_{i}(1) \rightarrow V_{i}(j)$ for every $2 \leq j \leq m_i$. Then condition $(\mathbb{I}_2)$ demands that 
$$\dim( \Span_K \{ v_1^{i}, \varphi_{2}^{-1}(v_2^{i}), \ldots,  \varphi_{m_i}^{-1}(v_{m_i}^{i}) \}) =m_i.$$
\item We opted to formulate the theorem in terms of an external direct sum in order to emphasize that the $v_j^{i}$ live in different copies. However, this tends to make the defining of the element $\vec{a}$ more lengthy so for ease of notation usually we will work with internal direct sums and hence write $\sum_{i =1}^q v_1^{i} + \cdots + v_{m_i}^{i}$ and '$a$' instead.
\end{enumerate}
\end{remark*}

Different choices of the point $\vec{a}$ may yield isomorphic irreducible gliders. In order to parametrize the isomorphism classes we need the following generalization of \cite[Lemma 7.1]{CVo4}

\begin{lemma}\lemlabel{grass}
Let $G$ be a finite group, $U$ a $d$-dimensional irreducible $G$-representation and $m \leq d$. The irreducible $(K \subseteq K[G])$-glider representations $K(u_1, \ldots  ,u_m)\subseteq U^{\oplus m}$ and $K(v_1, \ldots , v_m)\subseteq U^{\oplus m}$ are isomorphic if and only if $\Span\{ u_1,\ldots, u_m\}$ and $\Span\{ v_1,\ldots, v_m\}$ determine the same point in the Grassmanian $\Gr(m,U)$.
\end{lemma}
\begin{proof}
Extend $\{u_1,\ldots,u_m\}$ and $\{v_1,\ldots,v_m\}$ to $K$-bases for $U$. Then there exists a base change matrix $B$ such that $Bu_i = v_i$ for $1 \leq i \leq m$ if and only if $\Span\{ u_1,\ldots, u_m\}$ and $\Span\{ v_1,\ldots, v_m\}$ determine the same point in the Grassmanian $\Gr(m,U)$.
\end{proof}
For an irreducible $G$-representation $U$ of dimension $d$ we denote $\Gr(U) = \sqcup_{j=1}^d \Gr(j,U)$ 
and we denote a point in $\Gr(j,U)$ by $(a_1,\ldots,a_j) \in \mathbb{P}^{d-1} \times \ldots \times \mathbb{P}^{d-1}$ (all $a_k$ different). For $j = d$, $\Gr(d,U)$ is a singleton which we denote by $\{ \ast_U\}$. We denote by 
\begin{equation}\label{SG}
\mc{S} = \mc{S}_G
\end{equation}
the set of subsets $B \subseteq \sqcup_{U \in \Irr{G}~\vline~\dim(U) > 1} \Gr(U)$, such that for all $U$ the intersection $B \cap \Gr(j,U)$ is non-empty for at most one $1 \leq j \leq \dim(U)$ and for this $j$ it is in fact a singleton. 
\begin{proposition}\label{parametrisatie irreducible}
Let $G$ be a finite group. There is a bijection 
$$\frac{ \Irr_1 F(K[G])}{\cong}  \xleftrightarrow{1-1} \{ (A,B) \in \mathcal{P}(G/G') \times \mathcal{S}_G \}$$
where $G' = [G,G]$ is the commutator subgroup of $G$ and $\mathcal{P}(G/G')$ the power set of $G/G'$. 
\end{proposition}
\begin{proof}
Recall that the $1$-dimensional representations of $G$ correspond to the character group $\widehat{G/G'} = \Hom_{\text{grp}}(G/G', K^{*})$ and moreover $\widehat{G/G'} \cong G/G'$. We fix such an isomorphism and use it to fix a correspondence between the $1$-dimensional representations and the elements of $G/G'$. For $z \in G/G'$ denote the corresponding $G$-representation by $T_{z}$. For every $z \in G/G'$, take an element $t_z \in T_z$. 

From \Cref{irred} we see that a glider $M_{\bullet} \in \Irr_1 F(K[G])$ corresponds to the numbers $m_i$ and the choices of elements $v_j^{i}  \in V_i$ with $1 \leq j \leq m_i$. In case $V_i$ is $1$-dimensional, every different chosen element $v_1^{i}$ yields an isomorphic glider. Hence the choice reduces whether to pick $V_i$ or not, in other words
there is a one-to-one correspondence 
$$A=\{z_1, \ldots, z_l\} \in \mathcal{P}(G/G')  \xleftrightarrow{1-1} \frac{\left( K (t_{z_1}, \ldots, t_{z_l}) \subseteq \bigoplus_{z \in A} T_z \right)}{\cong}.$$
This correspondence does not depend on the chosen elements $t_z$ because of \lemref{grass}. For the $V_i$ of dimension at least $2$, the choice corresponds by definition (and due to \lemref{grass} and \Cref{irred}) to a point of $\mathcal{S}_G$. So altogether we obtain the statement.
\end{proof}
Note that the point $(\emptyset, \emptyset)$ corresponds to the glider $(0 \subset K)$ which is of essential length $0$ and hence is equal to zero in $\RRepR_1(\widetilde{G})$. Let us now give an example of how the correspondence in \cref{parametrisatie irreducible} works.

\begin{example}\label{voorbeeld Q_8}
 Let $G = Q_8 = \langle i,j,k ~\vline ~i^2 = j^2 = k^2 = ijk\rangle$ be the quaternion group. The abelianization of $Q_8$ is $C_2 \times C_2 \cong \langle a,b ~\vline~ a^2 = b^2 = 1 \rangle$
and denote $ab = c$. Fix the isomorphism between $Q_8/Q_8'$ and the group of 1-dimensional $Q_8$-representations
$$1 \mapsto T_1 \quad a \mapsto T_i \quad b \mapsto T_j \quad c \mapsto T_k.$$
With fixed basis $\{e_1,e_2\}$ of the 2-dimensional representation $U$, the point $[\lambda: \mu] \in \mathbb{P}^1$ determines the glider $K(\lambda e_1 + \mu e_2) \subseteq U$. We have the correspondences
$$\chi_{(\{b,c\},\{[1:1]\})} \longleftrightarrow T_j \oplus T_k \oplus U \supseteq K(t_j, t_k , e_1 + e_2)$$
and
$$\chi_{(\{1\},\{\ast_U\})}  \longleftrightarrow T_1 \oplus U^{\oplus 2} \supseteq K(t_1, u_1, u_2),$$
where $\dim_K(\Span_K \{ u_1,u_2 \}) = 2$.
\end{example}

To end this section we point out in \Cref{voor most easy filtratie is semigroupsalgebra} some key properties in which the filtrations of the type (\ref{de 'triviale' 1 step filtratie}) are peculiar in. 

Let $(Km \subseteq M)$, $(Kn \subseteq N)$ be irreducible gliders of length $1$. Since, from this section on, $K[G]$ is semisimple one has a decomposition as in (\ref{decomposition glider length 1 into cyclic}). Consequently,
\begin{equation}\label{prod in reduced char ring}
[Km\subseteq M ] \cdot [Kn \subseteq N] = [K(m \ot n) \subseteq K[G](m\otimes n)]
\end{equation}
in $\RRepR_1(\widetilde{G})$. As pointed out in (\ref{cyclicity condition for irr}), $K(m \ot n) \subseteq K[G](m\otimes n)$ is an irreducible glider. In other words, $\Irr_1 F(K[G])$ viewed as subset of $\RRepR_1(\widetilde{G})$ forms a multiplicatively closed set. Combined with \Cref{theorem glider repr as decat} we in fact obtain that it is a basis. The following will be crucial and often used without further. The case where is $G$ abelian was obtained in \cite[section 5]{CVo4}.
\begin{proposition}\label{voor most easy filtratie is semigroupsalgebra}
Let $H \leq G$ be finite groups. For their corresponding filtration of type (\ref{de 'triviale' 1 step filtratie}) we have that $\Irr_1 F(K[G])$ is a monoid and $$\RRepR_1(\wt{G}) \cong \Z[\Irr_1 F(K[G])]$$
is an integral semigroup algebra. Furthermore, $\IInd^G_H(\cdot)$ and $\RRes^G_H(\cdot)$ are ring morphisms preserving the basis and induction is in fact a monomorphism.
\end{proposition}
\begin{proof}
The first part has be explained above. For the second part, due to \Cref{ind and res induce ring map} and the fact that $\RRes^G_H(\cdot)$ is always a monoidal additive functor, it remains to prove that the induction and restriction functors preserve irreducible gliders. For induction this follows from (\ref{cyclicity condition for irr}), \Cref{ind and res adjoint pair} and $\dim \IInd^G_H(M_{\bullet})_{-1} = 1$ for this filtration. For restriction the same arguments must be combined with the fact that $[\RRes^G_H(M_{\bullet})] = [(K u \subset KHu)]$ in $\operatorname{K^{\op}}(F(k[H]),1)$ (due to (\ref{decomposition glider length 1 into cyclic}) we may always consider the canonical cyclic subglider $C_{\bullet}^M$ as representative). Finally that induction a monomorphism is follows directly from \Cref{ind and res induce ring map} as now $H \leq G$.
\end{proof}
Except for the part about $\RRes^G_H(\cdot)$, the previous result holds more generally for every field $k$, but still only for the filtration (\ref{de 'triviale' 1 step filtratie}). It would be interesting to find a basis of $\RRepR_t(\wt{G})$ for more general filtrations. \vspace{0,2cm}
\begin{example*}
Consider $G=Q_8$ and $H= \{ 1, -1 \}$. With notations as in \cref{voorbeeld Q_8}, we have that 
$$\RRes^{Q_8}_{H}((K.1 \subset T_{i})) = \RRes^{Q_8}_{H}((K.1 \subset T_{j})) = (K.1 \subset T_{H}).$$
Thus  $\RRes^G_H(\cdot)$ is in general not a monomorphism.
\end{example*}

\noindent {\bf Recurrent notation.} Given a tuple $(A,B) \in \mathcal{P}(G/G') \times \mathcal{S}_G$ we will write
\begin{itemize}
\item $M_{(A,B)}$ for the image in both $\RepR_1(\wt{G})$ and $\RRepR_1(\wt{G})$ of the isomorphism class of the irreducible $(K \subseteq KG)$-glider corresponding to $(A,B)$ following \Cref{parametrisatie irreducible}. In particular due to \Cref{voor most easy filtratie is semigroupsalgebra} we may write $M_{(A,B)}.M_{(C,D)} = M_{(E,F)}$ in a uniquely determined way;
\item when we want to stress out that we are working in $\ov{R}_1(\wt{G})$, e.g. because a given (in)equality only holds therein, then we will use square brackets. Thus we will write $[M_{(A,B)}]$ instead of simply $M_{(A,B)}$.
\item $M_{A}$ instead of $M_{(A,\emptyset)}$ (in spirit of \cite{CVo4} where the abelian case was handled). However both notations will be in use.
\item As $\Irr_1 (K\subset K[G])$ is viewed as a semigroup in the split Grothendieck ring, we will often write $M_{(A,B)}^n$ instead of $M_{(A,B)}^{\ot n}$.
\end{itemize}

\subsection{The \texorpdfstring{$G^{ab}$}{}-action and the obstruction modules}\label{sectie ab actie en obstructie modulen}
For the chain (\ref{de 'triviale' 1 step filtratie}) we can identify $G_{\bullet}$ with $G$ and $\varphi_{\bullet}:H_{\bullet}\rightarrow G_{\bullet}$ with $\varphi: H \rightarrow G$ without confusion or lost of information. \Cref{ind and res induce ring map} can now be restated as the statement that $\RRepR_1(\cdot)$ gives raise to a monoidal additive functor from the category ${\bf Grp}$ of groups to the category ${\bf Rng}$ of (not necessarily unital) rings. In the remainder of this paper we will be interested in the associated $\Q$-algebra $\Q(\wt{G}) = \Q \ot_{\Z} \RRepR_1(\wt{G})$. We denote the composition of the functors $\RRepR_1(\cdot)$ and $\Q \ot_{\Z} (-)$ by
$$\Phi :{\bf Grp} \rightarrow {\bf Alg}_{\Q}.$$
The morphism $\Phi(\varphi) : \Q(\wt{H}) \rightarrow \Q(\wt{G})$ will always be denoted as $\Phi^G_H$ (in order to reflect more in the notation that $\Phi(\varphi)(M_{\bullet}) = \IInd^G_H(M_{\bullet})$).

Now note that the description of the kernel in \Cref{ind and res induce ring map} also implies that $\Phi$ preserves monomorphisms. In other words,  for any subgroup $H \leq G$ we obtain a ring monomorphism $\Phi^G_H:\mathbb{Q}(\wt{H}) \hookmapright{} \mathbb{Q}(\wt{G})$. We will now upgrade this morphism to a $\Q[G^{ab}]$-algebra map, where $G^{ab} = G/G'$ is the abelianization of $G$. For this we first need to explain what the $G^{ab}$-actions are. To start, as in the proof of \Cref{parametrisatie irreducible}, {\it we fix until the end of the paper an isomorphism}
$$G^{ab} \cong \Irr_1(G^{ab}) : z \mapsto T_z.$$
Next, identify $T_z$ with the irreducible glider $M_{(\{z\}, \emptyset)}$ in $\RRepR_1(\wt{G})$ and subsequently with $M_{(\{z\}, \emptyset)} \ot 1$ in $\Q(\wt{G})$ which we denote compactly as $M_z$. It is easily seen that this identification is in fact a group isomorphism. The results obtained will not depend on the chosen isomorphism above since all the proofs in fact purely work with the identification of $\Irr_1(G/G')$ in $\Q(\wt{G})$.

Now the action of $G^{ab}$ on $\Q(\wt{G})$ is given by
 $$z . [N_{\bullet}] = M_z . [N_{\bullet}] = [M_z \ot N_{\bullet}]$$
and the action of $G^{ab}$ on $\Q(\wt{H})$ by $z . [M_{\bullet}] = [\RRes^G_H(M_z) \ot M_{\bullet}]$. The push-pull formula (\ref{Push-pull formula}) now translates into the fact that $\Phi^G_H$ commutes with the $G^{ab}$-action.
\begin{proposition}\label{phi is algebra map}
The map $\Phi^G_H$ is a $\Q[G^{ab}]$-algebra monomorphism. 
\end{proposition}

Due to the transitivity of the induction functors, the morphisms $\Phi^G_H$ also enjoy a transitivity which will regularly be used: if $H \leq E \leq G$, then 
\begin{equation}\label{inc}
\Phi_H^G(\mathbb{Q}(\wt{H}))= \Phi_E^G(\Phi_H^E(\mathbb{Q}(\wt{H}))) \subseteq \Phi_E^G(\mathbb{Q}(\wt{E})).
\end{equation}

Our first main structural \Cref{ses} will be a short exact sequence describing $\Q(\wt{G})$ as a $\mathbb{Q}[G^{ab}]$-module. The main protagonist hereby will be $\Phi^G_H(\Q(\wt{H}))$ for a class of subnormal subgroups $H$ and three $\mathbb{Q}[G^{ab}]$-modules $P(G), R(G)$ and $E(G)$ which will be the content of the remainder of this section. Later on, in \Cref{sectie interpretaties} we will give an interpretation of these modules in terms of representation and group theoretical information. 

\subsubsection*{{\bf The module $P$.}} Let $M_{(A,B)} \in \mathbb{Q}(\wt{G})$ and consider the cyclic semigroup $\langle M_{(A,B)} \rangle$. Due to \Cref{parametrisatie irreducible} and \Cref{voor most easy filtratie is semigroupsalgebra}, the glider representation ring is infinite dimensional for non-abelian groups. Therefore it could be that the cyclic semigroup $\langle M_{(A,B)} \rangle \cong \mathbb{N}$. 

\begin{definition}[{\it Obstruction module $P$}]
We denote by $P(G)$ the $\mathbb{Q}$-vector spanned by the elements $M_{(A,B)}$ which generate an infinite cyclic semigroup. 
\end{definition}

 If $M_{(A,B)} \notin P(G)$, then $\langle M_{(A,B)}\rangle$ is finite. It is well-known and easy to show that finite semigroups contain a unique idempotent element $e$. If $n$ is the smallest integer such that $M_{(A,B)}^{\ot n} = e$, then $(M_{(A,B)} - M_{(A,B)}^{\ot n+1})^{\ot n}=0$. In other words, the difference $M_{(A,B)} - M_{(A,B)}^{\ot n+1}$ is nilpotent. Thus such gliders will yield torsion elements in $\Q(\wt{G})/N$, the glider representation ring modulo its nilradical.\vspace{0,1cm}

{\it Notation:} For $M_{(A,B)} \notin P(G)$, the unique idempotent in $\langle M_{(A,B)}\rangle$ is denoted by $e(A,B)$. 

\begin{proposition}
The vector space $P(G)$ is a $\mathbb{Q}[G^{ab}]$-submodule of $\mathbb{Q}(\wt{G}).$
\end{proposition}
\begin{proof}
Let $M_{(A,B)} \in P(G)$ and $z \in G^{ab}$.  If $z\cdot M_{(A,B)} \notin P(G)$, then there exists $n > 0$ such that $(z.M_{(A,B)})^{\ot n} = e$ is idempotent. But then $e= e^{|G^{ab}|} = M_{(\{z \},\emptyset)}^{\ot n.|G^{ab}|}\ot M_{(A,B)}^{\ot n.|G^{ab}|}=  M_{(A,B)}^{\ot n.|G^{ab}|} \in P$, a contradiction.
\end{proof}

Being idempotent gives strong restrictions on the tuple $(A,B)$ as is already reflected from the following. Later on this result will be recovered from a more general one (cf \cref{change multiplicity under tensor product}), but for expository reasons we already include this sub-case now.

\begin{proposition}\label{AB in E}
Let $M_{(A,B)}, M_{(C,D)} \in \RRepR_1(\wt{G})$ irreducible. Denote the unique irreducible representant of the product $M_{(A,B)}.M_{(C,D)}$ by $M_{(E,F)}$.  If $A \neq \emptyset \neq C$, then $A.C \subseteq E$. Consequently, if $M_{(A,B)}$ is an idempotent with $A$ non-empty, then $A \subseteq G^{ab}$ is a subgroup.
\end{proposition}
\begin{proof}
Let $M_{(A,B)}$ correspond to the glider $K(\sum_{a \in A} t_a + u) \subset \Big(\bigoplus_{a \in A} T_a \oplus U\Big) $ and let  $M_{(C,D)}$ correspond to the glider $K(\sum_{c \in C} t_c + v) \subset \Big(\bigoplus_{c \in C} T_c \oplus V \Big)$ where $U$ and $V$ are $K[G]$-modules with simple components all having dimension strictly bigger than 1. By definition, the multiplication $M_{(A,B)}M_{(C,D)}$ is given by the glider
$$ K\big( (\sum_{a \in A} t_a + u) \ot (\sum_{c \in C} t_c + v) \big) \subset K[G]\Big( (\sum_{a \in A} t_a + u) \ot (\sum_{c \in C} t_c + v)\Big)$$
which is again irreducible (e.g. see \Cref{voor most easy filtratie is semigroupsalgebra}). Consequently, by \Cref{irred}, in order to know what are the simple submodules of $K[G]\Big( (\sum_{a \in A} t_a + u) \ot (\sum_{c \in C} t_c + v)\Big)$  we need to describe the vectors '$v_j^{i}$' mentioned in the theorem.

Suppose now that $A \neq \emptyset \neq C$. We see that $(\sum_{a \in A} t_a + u) \ot (\sum_{c \in C} t_c + v) = \big( \sum_{a.c \in AC} t_a \ot t_c \big)+ w$ for some $w$. Since $G^{ab}$ is isomorphic to the group of $1$-dimensional representations of $G$ we have that $T_a \ot T_c \cong T_{ac}$ and the vectors $t_a \ot t_b$ satisfy condition $(\mathbb{I}_1)$ and $(\mathbb{I}_2)$ of  \Cref{irred} (which for $1$-dim modules simply demands that the vectors are non-zero). Thus indeed $AC \subseteq E$.

Finally, suppose that $M_{(A,B)}$ is idempotent, then by the above $A.A \subseteq A$. However $e \in A$, since $e = a^{o(a)} \in A^{o(a)} \subseteq A$ for any $a \in A$, which entails that  $A.A = A$ as needed.
\end{proof}

\begin{example}\label{voorbeelden van idempotenten}\hspace{-0,2cm}
\begin{enumerate}
\item By the previous lemma an element $M_{(A,\emptyset)}$ is an idempotent if and only if $A$ is a subgroup of $G^{ab}$. Moreover, $\langle M_{(A,\emptyset)} \rangle$ will always contain an idempotent. For example if $|A|=1$ then the idempotent will be $M_{(\{ e\}, \emptyset)}$ and if $1 \in A$ then it will be $M_{(\langle A \rangle, \emptyset)}$. However if $1 \notin A$ there seems to be no generic form for the idempotent (the example $G= Q_8$ and $A= \{ i,j \}$ is instructive. In this case the idempotent is $M_{(\{ \pm 1, \pm k\},\emptyset)}$.)
\item If $H\leq G$ and  $M_{(A,B)} \in \RRepR_1(\wt{H})$ is an idempotent, then by \cref{voor most easy filtratie is semigroupsalgebra} also $\IInd^G_H(M_{(A,B)})$ is idempotent. Thus combined with the previous example, this gives a first main method to produce idempotent elements in $\RRepR_1(\wt{G})$ (and hence in $\Q(\wt{G})$).
\end{enumerate}
\end{example}

More restrictions on $(C,D)$ for an idempotent $M_{(C,D)}$ will follow in \Cref{subsectie behaviour multiplicity vector under tensor}. 

\subsubsection*{{\bf The module $R$.}}
Consider an irreducible glider of the form $M_{(\{z\},\emptyset)}$. Such elements are never in $P(G)$ and furthermore the associated idempotent is equal to $M_{(\{e\},\emptyset)}$. More general the following elements will play a special role.

\begin{definition}[{\it Obstruction module $R$}]
The $\Q$-vector space generated by all elements $M_{(A,B)}$ for which the associated idempotent element $e(A,B)$ is of the form $M_{(\{e\},D)}$ is denoted $R(G)$.
\end{definition}
From the description above we have that $\mathbb{Q}[G^{ab}] \subseteq R(G)$. Moreover, for $z \in G^{ab}$ it follows that $(z \cdot M_{(A,B)})^{\ot n.\lvert G^{ab}\rvert}=(M_z \ot M_{(A,B)})^{\ot n.\lvert G^{ab}\rvert} =M_{(A,B)}^{\ot n.\lvert G^{ab}\rvert} =  M_{(\{e\},D)}$. Hence $R(G)$ is a $\mathbb{Q}[G^{ab}]$-module.
\begin{example}\label{voorbeeld element in R}
Every element of $R$ must be of the form $M_{(\{ z \},D)}$ or $M_{( \emptyset,D)}$. Indeed, suppose $M_{(A,B)}\in R$ with $A$ non-empty. If $|A| \geq 2$, then thanks to \Cref{AB in E} for all powers $M_{(A,B)}^{\ot n} = M_{(E_n,F_n)}$ with $ A^n \subseteq E_n$. Hence if an idempotent $M_{(\{ e\},D)}$ is reached, say at step $n_0$, then we need that $|A^{n_0}| = 1$. This will only be possible if $A$ is a singleton.
\end{example}

Note that the product of idempotents is also an idempotent, however $M_{(\{e\},D_1)} \ot M_{(\{e\},D_2)}$ is a priori not necessarily again an idempotent of the form $M_{(\{e\},D')}$. We pose the following question.
\begin{question}\label{wanneer is R een ring vraag}
For which finite groups $G$ is $R(G)$ a $\Q[G^{ab}]$-algebra?
\end{question}
It will follow from \Cref{R niet triviaal alleen voor non-nilpotent} that nilpotent groups are examples.  

\subsubsection*{{\bf The module $E$.}}\hspace{0,2cm}\vspace{0,1cm}

Finally we introduce the following,

\begin{definition}[{\it Obstruction module $E$}]
The $\Q$-vector space generated by all elements $M_{(A,B)}$ for which the associated idempotent element $e(A,B)$ is of the form $M_{(\emptyset,D)}$ is denoted $E(G)$.
\end{definition}
Note that the associated idempotent of $M_{(\{g\},\emptyset)}M_{(A,B)}$ is equal to $e(A,B)$  and hence also $E(G)$ is a $\mathbb{Q}[G^{ab}]$-module. Intriguingly, we were unable to find an example of an element of the form $M_{(\emptyset,B)}$ which is an idempotent. Thus, as with $P(G)$, this set might always be trivial.

{\it Notation.} When the group $G$ is clear from the context we will often simply write $P,R$ and $E$.

\subsection{A short exact sequence describing \texorpdfstring{$\RRepR_1(\wt{G})$}{} modulo its radical}

Let $J(R)$ be the Jacobson radical of a ring $R$ and $N(R)$ the nilradical. In our main theorem below we describe $\Q(\wt{G})/N(\Q(\wt{G}))$. The reason for using the nilradical  rather than the more typical object $\mathbb{Q}(\wt{G})/J(\mathbb{Q}(\wt{G}))$ is the following: a priori the morphisms $\Phi_H^G$, as they are not surjective, do not factorize over the Jacobson radical. However, as $\Q(\wt{H})$ is commutative, its nilradical which consists of the nilpotent elements is preserved by the ringmorphism $\Phi_H^G$. Consequently we can consider the induced monomorphism
$$\ov{\Phi_H^G} : \mathbb{Q}(\wt{H})/N(\mathbb{Q}(\wt{H})) \longrightarrow \mathbb{Q}(\wt{G})/N(\mathbb{Q}(\wt{G})).$$
In \Cref{sectie decompositie zonder de obstructies} we will show, under certain conditions, that the Jacobson and nilradical coincide. 

\begin{theorem}\thelabel{ses}\label{ses}
Let $G$ be a finite group. We have the following short exact sequence of $\mathbb{Q}[G^{ab}]$-modules
\begin{equation}\nonumber
\resizebox{1.0 \hsize}{!}{$\xymatrix{ 0 \ar[r] &\frac{E(G)}{E(G)\cap N} + \frac{P(G)}{P(G)\cap N} + \sum\limits_{G' \leq H \lneq G} \ov{\Phi_H^G}\big(\mathbb{Q}(\wt{H})/N(\Q(\wt{H}))\big) \ar[rr]^<<<<<<<<<\Psi &&\mathbb{Q}(\wt{G})/N  \ar[r] & \frac{R(G)}{R(G) \cap N} \ar[r] & 0}$}
\end{equation}
where $N= N(\Q(\wt{G}))$ and the map $\Psi$ denotes the embedding.
\end{theorem}
The last morphism in the exact sequence is the canonical map to the cokernel of $\Psi$. So among others we will proof that $\frac{R(G)}{R(G) \cap N} \cong \coker{\Psi}$. With this information at hand we see that in fact the sequence is split by the map
$$f: R/(R\cap N) \to \mathbb{Q}(\wt{G})/N,~ \ov{M_{(A,B)}} \mapsto \ov{M_{(A,B)}}.$$
Consequently, the theorem above can be reformulated as the following direct sum decomposition of $\Q[G^{ab}]$-modules:
\begin{equation} \label{r}
\mathbb{Q}(\wt{G})/N \cong R/(R \cap N) \oplus \big( \sum_{G' \leq H \lneq G} \ov{\Phi_H^G}(\mathbb{Q}(\wt{H})/N) + (\frac{P + E}{(P + E) \cap N})\big).
\end{equation}
We opted to formulate \Cref{ses} in terms of an exact sequence because in \Cref{sectie decompositie zonder de obstructies} we will pursue a direct decomposition as $\Q[G^{ab}]$-algebras. This will finally be obtained in \Cref{gencharring}, unfortunately only when the obstruction modules vanish. In fact, it will turn out that the main problem to obtain such a decomposition is the lack of understanding of \Cref{wanneer is R een ring vraag}. 
\begin{remark*}
In the sum of the theorem we run over all subgroups $H$ between $G'$ and $G$. However, due to (\ref{inc}) it is in fact enough to consider the maximal such subgroups.  
\end{remark*}

The origin of the class of groups appearing in the theorem is in fact the classical \Cref{tussenstap}. For this we need the map $A_{\iota}: \{ G' \leq H \leq G\} \rightarrow \{ N \leq G^{ab} \}$
defined by 
$$A_{\iota}(H)=\{ z \in G^{ab} ~\vline~ \RRes^G_H(T_z) \cong T_H {\rm~as~} H{\rm -representations}\}.$$
It is easily checked that $A_{\iota}(H)$ is indeed a subgroup. 
Conversely for $N\leq G^{ab}$ define $$\mc{L}(N) =  \bigcap\limits_{z\in N} \ker(T_z).$$
We consider both sets as lattices in the typical way (i.e. with the inclusion, intersection and product $H_1.H_2$ of (normal) subgroups).
\begin{lemma}\label{tussenstap}
Let $G$ be a finite group. Then 
\[
\begin{tikzcd}
	{\{ N \leq G^{ab} \}} \arrow[rrr, "{\mc{L}(\cdot)}", shift left=2] &  &  & {\{ G' \leq H \leq G\}} \arrow[lll, "{A_{\iota}(\cdot)}", shift left=1]\\
\end{tikzcd}\vspace{-0,5cm}
\]
are dually isomorphic as lattices.  
\end{lemma}
\begin{proof}
The equality $\mc{L}(A_{\iota}(H))= H$ is an instance of a more general result. Namely, see \cite[lemma 2.21]{IsaacsBook}, if $N \triangleleft G$, then $N = \cap_{\psi \in \Irr(G|N)} \ker(\psi)$ where $\Irr(G|N)= \{ \psi \in \Irr(G) \mid N \subseteq \ker(\psi)\}$, which corresponds via lifting to $\Irr(G/N)$. As the sets have equal finite cardinality we have that the functions are each other inverse and hence bijective. Clearly they are inclusion reversing and hence as posets the one is isomorphic to the dual of the other. We now consider the meet and join operations, i.e. 
\begin{equation}\label{meet and join for A and L maps}
A_\iota(H)A_{\iota}(E) = A_\iota(H \cap E) \text{ and } A_\iota(H)\cap A_{\iota}(E) = A_\iota(H. E)
\end{equation}
By the first part we can write $\mc{L}(C) = H, \mc{L}(D) = E$ for some $C,D \leq G^{ab}$. Proving the above equalities are equivalent to 
$$\mc{L}(C.D) = \mc{L}(C) \cap \mc{L}(D) \text{ and } \mc{L}(C\cap D)= \mc{L}(C).\mc{L}(D).$$ These equations are more transparent. For example for the first:
\begin{equation}\nonumber
H \cap E \subseteq \bigcap_{c \in C}\bigcap_{d \in D} \Ker(T_c \ot T_d) = \bigcap_{e \in CD} \Ker(T_e) \subseteq \bigcap_{c \in C} \Ker(T_c) \cap \bigcap_{d \in D} \Ker(T_d) = H \cap E.
\end{equation}
where we used that $\ker(T_c)\cap\ker(T_d) \subseteq \ker(T_c\ot T_d)$ and $C,D \subseteq CD$. The proof of the other operation is similar. 
\end{proof}

As an illustrative example we now handle the abelian case of \Cref{ses} in which many terms in fact vanish, but nevertheless a sketch of the general case already arise.  

\begin{example}\label{abelian case of SES}
 Let $G$ be abelian. In this case $\Irr_1(K \subset K[G])= \{ M_{(A,\emptyset)} \mid A\subseteq G^{ab}\}$ with in fact $G^{ab}=G$. In particular, by definition, $E = 0$ and $\mathbb{Q}(\wt{G})$ is finite dimensional which yields that $P=0$. By \Cref{AB in E}, the idempotent elements in $\Irr_1(K \subset K[G])$ are those of the form $M_{(C,\emptyset)}$ with $C\leq G^{ab}=G$ and $R= \langle M_{z} \mid z \in G^{ab} \rangle = \Q[G^{ab}]= \Q[G]$. Now take $M_{(A,\emptyset)}$ with idempotent $M_{(C,\emptyset)}$, say reached at the $n$-th power. 
 
By the push-pull formula (\ref{Push-pull formula}), $\IInd^G_{\mc{L}(C)} \RRes^G_{\mc{L}(C)} M_{(C, \emptyset)} \cong M_{(C,\emptyset)} \ot \IInd^G_{\mc{L}(C)}(M_{e})$ where the second term can be described as $\IInd^G_{\mc{L}(C)}(M_{e}) \cong M_{(C,\emptyset)}$ (this follows directly from the definition of $\mc{L}(C)$). Hence, $\IInd^G_{\mc{L}(C)} \RRes^G_{\mc{L}(C)} M_{(C, \emptyset)} \cong M_{(C, \emptyset)}$. Thus in combination with \Cref{tussenstap} we have that  $M_{(C,\emptyset)} \in \Ima(\Psi)$ if and only if $C\neq \{e\}$. In particular, $\Q[G^{ab}] = R \subseteq \coker(\Psi)$. 

Next a direct computation shows that $(M_{(A,\emptyset)} - M_{(A,\emptyset)}^{n+1})^n = M_{(C,\emptyset)} (\sum_{i=0}^n {i \choose n} (-1)^{i})=0$. Consequently, modulo the nilradical $$M_{(A,\emptyset)} \equiv M_{(A,\emptyset)}^{n+1} = M_{(A,\emptyset)} \ot M_{(C,\emptyset)} = \IInd^G_{\mc{L}(C)} \RRes^G_{\mc{L}(C)} (M_{(A, \emptyset)}).$$
Thus altogether we see that in the abelian case $\Q(\wt{G}) = \big(\sum_{H \leq G} \ov{\Phi}^G_H(\Q(\wt{H}))\big) \op \Q[G^{ab}]$, as desired.
\end{example}

\subsection{Behaviour multiplicities under tensor product and  \texorpdfstring{$\IInd^G_H\RRes^G_H$}{}}\label{subsectie behaviour multiplicity vector under tensor}

For the general case of \Cref{ses} we need a better understanding of the multiplicity vector of the tensor product of two irreducible gliders, which will be introduced in this subsection. \Cref{abelian case of SES} also indicates that we need to understand the behaviour of $\IInd^G_H(\RRes^G_H(M_{(A,B)}))$. This section will deliver the necessary tools.

Denote, as in \Cref{subsectie length 1}, by $\{V_1, \ldots, V_q \}$ a chosen set of isomorphism classes of simple $K[G]$-modules. Now as in \Cref{irred} we take subsets $\{ v_1^{i}, \ldots, v^i_{m_i}  \} \subset V_i$ and form the (external) direct sum $U := \bigoplus^{q}_{i=1} V_i^{\op m_i}$ for some (potentially zero) integers $m_i$.  We need to understand the simple components of the submodule $K[G]\vec{v}$ where 
$$\vec{v}=(v_1^{1}, \ldots, v^1_{m_1}, \ldots, v_1^{q}, \ldots, v_{m_q}^q) \in U.$$
The following lemma is a compact reformulation of several lemmas in \cite[Section 3]{CVo2} of which we give a sketch for the convenience of the reader.

\begin{lemma}\label{lemma with components of cyclic from vector}
With notations as above we have that
$$KG \vec{v} \cong \bigoplus^{q}_{i=1} V_i^{\op l_i}$$
with $l_i = \dim \Span_K \{ v_1^{i}, \ldots, v_{m_i}^{i} \}$. In particular, $l_i \leq \dim V_i$ and $l_i \neq 0$ if some $v^{i}_j$ with $1 \leq j \leq m_j$ is non-zero. 
\end{lemma}
\begin{proof}[Sketch of proof.]
Denote the $j^{th}$-copy of $V_i$ by $V_i(j)$. 
It is easy to see that the irreducible components of $K[G]\vec{v}$ are exactly those $V_i$ for which some $v^{i}_j \neq 0$. Indeed, such an element yields that the canonical projection $K[G]\vec{v} \rightarrow V_i(j): \vec{v} \mapsto v_j^{i}$ is a non-zero module morphism and thus $V_i$ is a component. The existence of such vector is also clearly necessary. 

Thus the statement is mainly about the formula for the multiplicities. Fix some $i$, we now work by induction on $m_i$. If $m_i= 1$, choose an $x \in \bigcap_{j \neq i } \ann (V_j)$ which is not in $\ann(V_i)$ (which exists as $K[G]$ is semisimple).  Then $K[G]x \vec{v}$ can be seen as a non-zero module of $V_i(1)$ and hence equals $V_i(1)$ as needed.

Next, consider $1 \leq j \leq m_i$ and the associated projection $\pi_{\widehat{(i,j)}}$ from $\bigoplus^{q}_{i=1} V_i^{\op m_i}$ which leaves out the component $V_i(j)$. Then $\ker\big( \restr{\pi_{\widehat{(i,j)}}}{K[G]\vec{v}} \big) = \{ \alpha \vec{v} \mid \alpha \in \bigcap_{(t,k) \neq (i,j)} \ann(v_k^t)\}$. In particular if $v^{i}_j \in \Span_K \{ v^{i}_k \mid k \neq j\}$, then $\pi_{\widehat{(i,j)}}(K[G]\vec{v}) \cong K[G]\vec{v}$ and subsequently induction hypothesis would yield that $l_i = \dim \Span_K \{ v_k^{i}\mid k \neq j \} = \dim \Span_K \{ v_1^{i}, \ldots, v_{m_i}^{i} \}$, as desired. Thus it remains to consider the case that $\{v_1^{i}, \ldots, v_{m_i}^{i} \}$ is linearly independent. For this complete the set to a basis of $\mathcal{B}_i$ of $V_i$. Now recall that by Wedderburn-Artin's theorem $K[G] \cong \prod_{i= 1}^q \End_K (V_i)$ as a ring (where we used that $K$ is algebraically closed and so $\End_{K[G]}(V_i)\cong K$ by Schur's lemma). It is now not hard to see that the $x$ from above can be chosen such that it acts trivially on (each copy of) $V_i$. 

Now consider $K[G]x\vec{v} \leq K[G]\vec{v}$. We claim that $K[G]x\vec{v} \cong V^{\op m_i}$, which would finish the proof. To see this, consider the short exact sequence of vector spaces 
$$0 \rightarrow \ann(x\vec{v}) \rightarrow K[G] \rightarrow K[G]x\vec{v} \rightarrow 0$$
which entails that $\dim K[G]x \vec{v} = |G| - \dim \ann(x\vec{v})$. Next, by having a closer look to the decomposition of $K[G]$ above and rewriting $\End_K(V_i) \cong M_{\dim V_i} (K)$ using the basis $\mathcal{B}_i$, we can describe explicitly $\ann(x\vec{v})$. Doing so, we would see that $\dim \ann(x\vec{v}) = |G| - m_i \dim (V_i)$. Thus altogether $\dim K[G]x\vec{v} = m_i \dim (V_i)$, but it could be considered as a submodule of $V^{\op m_i}$, therefore yielding the claim and finishing the proof.
\end{proof}

Now we introduce a new tool. Let, $M_{(A,B)} = (K\vec{v} \subset V)$ with  
$$K[G]\vec{v} = V \cong \bigoplus_{j=1}^q V_j^{\oplus m_j}.$$
Then we call $(m_1,\ldots,m_q) \in \mathbb{N}^q$ the \emph{multiplicity vector} of $M_{(A,B)}$ and will denote it by $\vec{m}_{(A,B)}$ or $\vec{m}(M_{(A,B)})$. Furthermore, $\dim(M_{(A,B)}) := \dim(K[G]\vec{v}) = \sum_{j=1}^q m_j \dim(V_j)$ will be called its \emph{total dimension}. Besides we will consider a partial order on $\N^q$ which is not the lexicographic one. Namely, take $\vec{m} =(m_1, \ldots, m_q)$ and $\vec{n} = (n_1, \ldots, n_q) \in \N^q$ then we define
$$\vec{m} \leq \vec{n} \text{ if and only if } m_i \leq n_i \text{ for all } 1 \leq i \leq q.$$

Using these notions we readily obtain from \cref{lemma with components of cyclic from vector} the following inequality on the growth of multiplicities, which will be crucial in the sequel. {\it The proof is notational lengthy to write down, however the philosophy is short:} by the previous lemma in order to compute the multiplicity vector of an irreducible glider one needs to look at the $-1$-level of the glider and find enough independent vectors. Consequently, when subsequently tensoring with two irreducible gliders, one contained in the other, then with the larger one the same vectors are found as with the smaller one but with potentially more vectors and hence multiplicities can only increase.

\begin{corollary}\label{change multiplicity under tensor product}
Let $(A,B) \in G^{ab} \times \mathcal{S}_G$. Then for any subsets $C \subseteq C' \subseteq G^{ab}$ and $D \subseteq D' \in \mathcal{S}_G$ we have that
$$\vec{m}(M_{(A,B)} \ot M_{(C, D)}) \leq \vec{m} (M_{(A,B)} \ot M_{(C',D')}).$$
In particular, if $M_{(C,D)}$ is idempotent and $C \neq \emptyset$ then $[\IInd^G_{\mc{L}(C)} \RRes^G_{\mc{L}(C)} (M_{(C,D)}) ] = [M_{(C,D)}]$. 
\end{corollary}
\begin{proof}
Write $M_{(A,B)} = (Ku \subset U)$ with $K[G]u = U := \bigoplus_{i=1}^q V_i^{\op m_i}$. By \Cref{irred} we have that $u = \sum_{i=1}^q u_1^{i}+ \cdots + u^{i}_{m_i}$ for some linearly independent subsets $\{ u^{i}_1, \ldots, u^{i}_{m_i} \} \subset V_i$. Next let $\{ w^{i}_1, \ldots, w^{i}_{n_i} \} \subset V_i$ be other sets of independent vectors corresponding to $D$. Then $M_{(C,D)} = (K w_{(C,D)} \subset \bigoplus\limits_{c\in C} T_c \op \bigoplus\limits_{i=|G|^{ab}+1}^q V_i^{\op n_i} )$ where $w_{(C,D)}= \sum\limits_{c \in C} t_c + \sum\limits_{|G^{ab}|\lneq i}^q w^{i}_1+ \cdots+ w^{i}_{n_i}$. Also let $M_{(C', D')} = (K w_{(C',D')} \subset W)$ with 
$$W = K[G]w_{(C',D')} = \bigoplus_{c\in C} T_{c} \op \bigoplus_{c'\in C'\setminus C} T_{c'}  \op \bigoplus_{i=|G|^{ab}+1}^q V_i^{\op n_i} \op V,$$
where $V$ corresponds to $D'\setminus D$.  In particular, we can write $w_{(C',D')} = w_{(C,D)} + v + \sum_{c' \in C'\setminus C} t_{c'}$ with $v \in V$. Note that by \lemref{grass} the chosen vectors in $T_{c}$ for $c\in C$ do not matter and so we may indeed assume without lose of generality that in $M_{(C,D)}$ and $M_{(C',D')}$ the vectors in $\oplus_{c\in C}T_c$ are the same.
 
Now note that the multiplicity of $V_j$ in 
$$
\big( \bigoplus_{i=1}^q V_{i}^{\op m_i} \big) \ot (K[G]w_{(C,D)}) \cong \bigoplus_{i=1}^q \bigoplus_{c\in C} (V_i \ot T_{c})^{\op m_i} \op \bigoplus_{\substack{1 \leq i \leq q \\ |G^{ab}| \lneq k}}(V_i \ot V_k)^{\op (m_i + n_j)} 
$$
 is equal to 
$$\sum_{i=1}^q m_i . |C_{i \rightarrow j}| + \sum_{i,k} (m_i + n_k). M_{i,k}(j) $$
 with  $C_{i \rightarrow j} = \{ c \in C \mid V_i \ot T_c \cong V_j \}$ and $M_{i,k}$ is the multiplicity of $V_j$ in $V_i \ot V_k$.  For $c \in C_{i \rightarrow j}$ let $\varphi_{i,j}^c: V_i \ot T_c \rightarrow V_j$ a $K[G]$-module isomorphism. Further for $1 \leq x \leq m_i$ and $1 \leq y \leq n_i$ denote by $\{ (u^{i}_{x} \ot w_{y}^{k})_{j_1}, \ldots, (u^{i}_{x} \ot w_{y}^{k})_{j_{M_{i,k}(j)}} \}$ the coordinates of $u^{i}_{x} \ot w_{y}^{k}$ in the $M_{i,k}(j)$ copies of $V_j$ in $V_i \ot V_k$.
 
 Due to \Cref{lemma with components of cyclic from vector} we know that $K[G] (u \ot w_{(C,D)}) \cong \bigoplus_{i=1}^q V_i^{\op l_i}$ with 
 \begin{equation*}
 \begin{split}l_i = \dim \Span_K\{ \varphi_{i,j}^c(v^{i}_k \ot t_c), (u^{i}_{x} \ot w_{y}^{k})_{j_1}, \ldots, (u^{i}_{x} \ot w_{y}^{k})_{j_{M_{i,k}(j)}}
 \mid 1 \leq k, x \leq m_i, 1 \leq y \leq n_i \\c \in C_{i \rightarrow j} \neq \emptyset \}.\end{split}
  \end{equation*}
  Due to the lemma of Schur this number does not depend on the chosen isomorphisms $\varphi_{i,j}^c$. The first part of the statement now follows by simply remarking that for $M_{(A,B)} \ot M_{(C',D')}$ exactly these vectors are also all taken but with potentially more (depending if $V_j$ is also a component of $V_i \ot V$ or $V_i \ot T_{c'}$), hence the multiplicity can only increase.
 
 For the second part, note that by the push-pull formula and \Cref{AB in E}
  $$\IInd^G_{\mc{L}(C)} \RRes^G_{\mc{L}(C)} (M_{(C,D)}) \cong M_{(C,D)} \ot M_{(C,\emptyset)} \cong M_{(E,F)}$$
  with $C \leq E$  and by the first part $\vec{m}_{(C,D)} \leq \vec{m}_{(E,F)}$. However, it is not only an inequality of multiplicities but even $D \subseteq F$. This follows from the fact that $1 \in C$ and hence $F$ contains by construction the subset  $\{ d \ot 1 \mid d\in D \}$ which however as points in the Grassmanian leads to the same points as $D$.  Finally, using again the first part $\vec{m}(M_{(C,D)} \ot M_{(C, \emptyset)}) \leq \vec{m}_{(C,D)}$ as $M_{(C,D)}$ is idempotent and thus $D =F$ and $C= E$, as desired. 
\end{proof}

Now as in \Cref{AB in E} consider a product $M_{(A,B)}.M_{(C',D')} = M_{(E,F)}$ of irreducible gliders in $\Q(\wt{G})$. As a direct consequence of \cref{change multiplicity under tensor product} (taking $(C,D) = (\{e\}, \emptyset)$) we see that if $1 \in A$, then $\vec{m}_{(C',D')} \leq \vec{m}_{(E,F)}$. In particular if also $1 \in C$, then
\begin{equation}\label{max van multiplicities}
\max\{\vec{m}_{(A,B)}, \vec{m}_{(C',D')} \} \leq \vec{m}_{(E,F)}.
\end{equation}
In particular, we now find another natural example of an idempotent.
\begin{example*}
The regular glider $k[G]_{\bullet}=(K .\, 1_G \subset K[G])$ corresponds to the element $(G^{ab}, \{ \ast_U \mid U \in \Irr(G), \, \dim(U) \geq 2 \})$ of $\mc{P}(G^{ab}) \times \mc{S}_G$ with maximal multiplicity vector. Consequently $k[G]_{\bullet}$ is an idempotent. 
\end{example*}

\begin{remark*}
When $1\notin A$ one can not deduce from \cref{change multiplicity under tensor product} that $\vec{m}(M_{(A,B)}) \leq \vec{m}(M_{(A,B)}^{\ot 2})$ and in particular also not that $\vec{m}(M_{(A,B)}) \leq \vec{m}(e(A,B))$ where $e(A,B)$ is the associated idempotent when $M_{(A,B)}\notin P$.  In fact these inequalities unfortunately do not hold as shown by \Cref{voorbeelden van idempotenten}. However when $1 \in A$ they do.
\end{remark*}
Next we consider the operation $\IInd^G_H\RRes^G_H$ also for non-idempotent gliders and for  any $H \leq G$. In these cases, using (\ref{Push-pull formula}) and the previous corollary, we have that 
\begin{equation}\label{indres doet multiplicity stijgen}
\vec{m}(M_{(A,B)}) \leq \vec{m}(\IInd^G_H \RRes^G_H(M_{(A,B)}))
\end{equation}
the operation usually yields a strictly larger irreducible glider. 

\begin{proposition}\label{irreducible under Ind(Res)}
Let $H \leq G$. For $M_{(A,B)}, M_{(C,B)} \in \Irr_1(F(K[G]))$ and $M_{(E,F)} \in \Irr_1(F(K[H]))$ the following hold
\begin{enumerate}
    \item $\RRes^G_H (M_{(A,B)})$ (resp. $\IInd^G_H (M_{(E,F)})$) is idempotent if $M_{(A,B)}$ (resp. $M_{(E,F)}$) is,\vspace{0,1cm}
    \item $[\RRes^G_{H} \IInd^G_{H}(M_{(E,F)})]=[M_{(E,F)}]$,\vspace{0,1cm}
    \item $[\IInd^G_{H} \RRes^G_{H}(M_{(A,B)})] = [M_{(A',B')}]$ with $A\subseteq A'$ and $B \subseteq B'\in \mc{S}_G$. Moreover, if $G' \leq H$ and $A \neq \emptyset$, then $A \, . \, A_{\iota}(H) = A'$. 
\end{enumerate} 
\end{proposition}
\begin{proof}
The first part follows directly from the fact that restriction and induction are ring morphisms on the reduced glider representation ring, \cref{voor most easy filtratie is semigroupsalgebra}. Next, the second assertion is a consequence of (\ref{decomposition glider length 1 into cyclic}), due to which $[\RRes^G_H(M_{\bullet})]= [(Ku \subset KHu)]$ in $\RRepR_1(\wt{H})$,  and the fact that irreducible gliders are cyclic.

Due to \Cref{voor most easy filtratie is semigroupsalgebra} we know that $[\IInd^G_{H} \RRes^G_{H}(M_{(A,B)})]$ is again irreducible and so is of the form $M_{(A',B')}$. Furthermore due to (\ref{Push-pull formula}) $[M_{(A',B')}] = [M_{(A,B)} \ot \IInd^G_{H}(M_{(\{T_H\}, \emptyset)})]$. It is not hard to see that $\IInd^G_{H}(M_{(\{T_H\}, \emptyset)}) = M_{(C,D)}$ with $C$ truly a subgroup of $G^{ab}$. Consequently, see before (\ref{max van multiplicities}), $\vec{m}_{(A,B)} \leq \vec{m}_{(A',B')}$ and thus $A\subseteq A'$.  That $B \subseteq B'$ is similar as at the end of the proof of \cref{change multiplicity under tensor product}.

Now suppose that $G' \leq H$. This entails that $\IInd^G_{H}(M_{(\{T_H\},\emptyset)}) = M_{(A_{\iota}(H),\emptyset)}$ (using Frobenius reciprocity one directly sees that only one-dimensional irreducible $G$-representations can lie over $T_{H}$ if $G' \leq H$). So if $A \neq \emptyset$ \Cref{AB in E} yields that $A\, . \, A_{\iota}(H) \subseteq A'$. However $U\ot T_a$ is irreducible with $\dim (U \ot T_a) = \dim (U) \geq 2$ for all $U \cap B \neq \emptyset$ and $a \in A_{\iota}(H)$. Thus no other $1$-dimensional simple submodules can appear and therefore $A\, . \, A_{\iota}(H) = A'$, as needed.
\end{proof}

Note that if in the third part $A_{\iota}(H) \subseteq A$ (e.g. if $A$ is a subgroup and $\mc{L}(A) \subseteq H$), then $A=A'$. To finish this section, whose main purpose is to provide new tools to understand irreducible gliders, we consider necessary conditions for $M_{(A,B)}$ to be in the image of $\Phi^G_H$ for some strict subgroup $H$.

\begin{lemma}\label{induced from subgroup lemma}
 Let $\mc{C}$ be a class of subgroups of $G$. If $M_{(A,B)} \in \sum_{H \in \mc{C}} \Phi^G_H(\Q(\wt{H}))$. Then, 
\begin{enumerate}
\item[(i)] $[\IInd^G_H\RRes^G_H(M_{(A,B)})] = [M_{(A,B)}]$ for some $H \in \mc{C}$;
\end{enumerate}
Furthermore for an $H$ as in (i) we have that
\begin{enumerate}
\item[(ii)] $[\IInd^G_{H_m}\RRes^G_{H_m}(M_{(A,B)})] =[M_{(A,B)}]$ for any larger subgroup $H \lneq H_m \in \mc{C}$;
\end{enumerate} 
\end{lemma}
\begin{proof}
Since in the split Grothendieck ring sum corresponds to a direct sum decomposition,  \Cref{theorem glider repr as decat} entails that every irreducible glider $[M_{(A,B)}]$ is indecomposable in $\Q(\wt{G})$. \Cref{voor most easy filtratie is semigroupsalgebra} now implies that  $[\IInd^G_H(M_{(E,F)})] = [M_{(A,B)}]$ for some $M_{(E,F)} \in \Irr_1(K \subset K[H])$ and  $H  \in \mc{C}$.  Using \Cref{irreducible under Ind(Res)} we see that $[M_{(E,F)}] = [\RRes^G_H(M_{(A,B)})]$ as claimed. 

For the second statement take $H$ as in part (i) or thus in combination with (\ref{Push-pull formula}) suppose that $[M_{(A,B)}] = [M_{(A,B)} \ot \IInd^G_H(M_{(\{T_H\},\emptyset)})]$.  Now for $H\lneq H_m \in \mc{C}$ we need to look at $M_{(A',B')}:= \IInd^G_{H_m}(\RRes^G_{H_m}(M_{(A,B)}))= M_{(A,B)} \ot \IInd^G_{H_m}(M_{(\{T_{H_m}\},\emptyset)})$. By (\ref{indres doet multiplicity stijgen}) $\vec{m}(M_{(A,B)}) \leq \vec{m}(M_{(A',B')})$. For the reverse inequality: $M_{(C',D')} := \IInd^G_H(M_{(\{T_H\},\emptyset)}) =\IInd^G_{H_m}\IInd^{H_m}_H (M_{(\{T_H\},\emptyset)})$ and so clearly if $M_{(C,D)} := \IInd^G_{H_m}(M_{(\{T_{H_m}\},\emptyset)})$, then $C \subseteq C'$ and $D \subseteq D'$. Consequently, using \cref{change multiplicity under tensor product}, $\vec{m}\big(\IInd^G_{H_m}\RRes^G_{H_m}(M_{(A,B)})\big) \leq \vec{m}(M_{(A,B)})$ and hence equality follows. Part (3) of \cref{irreducible under Ind(Res)} now yields that it is in fact an equality of gliders in $\Q(\wt{G})$.
\end{proof}

\subsection{Proof of \texorpdfstring{\cref{ses}}{} and a distinguished subalgebra}

 We now have all the ingredients to prove our first main structural theorem for a general finite group.

\begin{proof}[Proof of \Cref{ses}.]\hspace{0,1cm}\vspace{0,1cm}
Take $M_{(A,B)} \in \Q(\wt{G})\setminus P(G)$ and so there is an associated idempotent $M_{(C,D)} = M_{(A,B)}^{\ot n}$ for some $n \in \N_{0}$. \vspace{0,1cm}

{\bf Prelude.} If $C = \emptyset$, then by definition $M_{(A,B)} \in E(G)$. Next, assume that $C \neq \emptyset$. By \Cref{change multiplicity under tensor product} $$[\IInd^G_{\mc{L}(C)} \RRes^G_{\mc{L}(C)}(M_{(C,D)})] = [M_{(C,D)}]$$ is induced from an idempotent of $\mc{L}(C)$ which is a strict subgroup of $G^{ab}$ exactly when $C \neq \{ e \}$. Thus we already understand fully the idempotents not in $R(G)$. This module is the content of the next step.\vspace{0,1cm}

{\bf Step 1.} We claim that $R(G) \cap \sum_{G'\leq H \lneq G} \Phi_H^G(\Q(\wt{H})) = 0$.\vspace{0,1cm}

Consider $M_{(A,B)} \in R(G)$ with idempotent $M_{(\{e \}, D)}$ for some $D \in \mc{S}_G$. If the idempotent lies in $\sum_{G'\leq H \lneq G} \Phi^G_H(\Q(\wt{H}))$, then combining \Cref{irreducible under Ind(Res)} and \Cref{induced from subgroup lemma} we obtain that $\{e \} A_{\iota}(H) = \{e \}$ for some $G' \leq H \lneq G$. Hence $A_{\iota}(H) = \{ e \}$ or equivalently $H=G$, a contradiction.
 Since $\Phi^G_H(\cdot)$ is multiplicative, also $M_{(A,B)}$ can not lie in that sum (otherwise $M_{(\{e \}, D)}$ would also lie in it). Subsequently, also no linear combination can because $\IInd^G_H(\cdot)$ preserve irreducible gliders and $\Irr_1(F(K[G]))$ forms a basis of $\Q(\wt{G})$, see \Cref{voor most easy filtratie is semigroupsalgebra}.\vspace{0,1cm}

{\bf Step 2.} Let $M_{(A,B)} \in \Q(\wt{G}) \setminus (P \op E)$ with associated idempotent $M_{(A,B)}^{\ot n}  = M_{(C,D)}$.  In this step we will prove that
\begin{equation}\label{stap 2 bewijs SES}
\big(M_{(A,B)} - \IInd_{\mc{L}(C)}^G \RRes^G_{\mc{L}(C)}(M_{(A,B)}) \big)^{n} = 0
\end{equation}
in $\Q(\wt{G})$ (for the $n$ as in the definition of $M_{(C,D)}$).  \vspace{0,1cm}

Consider $y_i = M_{(A,B)}^{i} \IInd_{\mc{L}(C)}^G \RRes^G_{\mc{L}(C)}(M_{(A,B)})^{n-i}$ for $0 \leq i \leq n$. If $i= 1$ or $n$, then $y_i = M_{(C,D)}$ due to \Cref{change multiplicity under tensor product} and the multiplicativity of the functions. We will now show that this also holds for all other $i$. Using consecutively \Cref{change multiplicity under tensor product}, push-pull formula (\ref{Push-pull formula}) and (\ref{indres doet multiplicity stijgen}) we find that

\begin{equation}\nonumber
\begin{split}
\vec{m}(M_{(C,D)}) & = \vec{m}(M_{(A,B)}^{\ot n} \ot \IInd^G_{\mc{L}(C)}(M_{e}))\\
 & = \vec{m}\big(M_{(A,B)}^{\ot i} \ot \IInd^G_{\mc{L}(C)}(\RRes^G_{\mc{L}(C)}(M_{(A,B)}^{\ot n-i}))\big)\\
 & \leq \vec{m}\big( \IInd^G_{\mc{L}(C)} \RRes^G_{\mc{L}(C)}(M_{(A,B)}^{\ot n})\big)\\
 &= \vec{m}(M_{(C,D)}).\\
\end{split}
\end{equation}
Consequently, $\vec{m}(y_i) = \vec{m}(M_{(C,D)})$. However part (2) of \cref{irreducible under Ind(Res)} also yields that $[\RRes^G_{\mc{L}(C)}(y_i)] = [\RRes^G_{\mc{L}(C)}(M_{(C,D)})]$. Combining the both we obtain that in fact $y_i = M_{(C,D)}$, as asserted. Using this (\ref{stap 2 bewijs SES}) now follows readily
\begin{equation}\nonumber
\begin{split}
\big(M_{(A,B)} - \IInd_{\mc{L}(C)}^G \RRes^G_{\mc{L}(C)}(M_{(A,B)}) \big)^{n} & = \sum_{i = 0}^n {n\choose i} (-1)^{n-i} M_{(A,B)}^{i}\IInd_{\mc{L}(C)}^G \RRes^G_{\mc{L}(C)}(M_{(A,B)})^{n-i}\\
&= \sum_{i = 0}^n {n\choose i} (-1)^{n-i} y_i \\
& =M_{(C,D)} \sum_{i = 0}^n {n\choose i} (-1)^{n-i}  =0\\
\end{split}
\end{equation}

{\bf Conclusion.} By step $2$ whenever $\mc{L}(C) \neq G$, i.e. $C \neq \{ e \}$, or in other words if it is not in $R(G)$, we know that $M_{(A,B)} \in \Ima(\Psi)$.  Indeed either it is in $P\op E$ or $M_{(A,B)} \equiv \IInd_{\mc{L}(C)}^G \RRes^G_{\mc{L}(C)}(M_{(A,B)})$ modulo the nilradical $N(\Q(\wt{G}))$.  Thus combined with step $1$ we obtain that $\coker(\Psi) = R(G) / R(G) \cap N(\Q(\wt{G}))$. This finishes the proof.
\end{proof}

Unfortunately, \Cref{ses} gives no information about the nil or Jacobson radical of $\Q(\wt{G})$. The only information that we will obtain in this article is that if $E = P= 0$ and $R(G) = \mathbb{Q}[G^{ab}]$, then they coincide. We expect however that this is a general phenomena. In case $G$ is abelian or the quaternion group $Q_8$, then a generating set for the radical was obtained in \cite[Theorem 5.10. \& 7.6.]{CVo4} by Caenepeel-Van Oystaeyen.

\begin{remark}
If we denote by $\Irr^{id}(F[K[G]])$ the set of idempotents in $\Irr_1(F(K[G]))$, then this set is multiplicatively closed and in fact $\Q^{id}(\wt{G}) := \Q \ot_{\Z} \Z[\Irr^{id}(F(K[G]))]$ is a $\Q[G^{ab}]$-subalgebra of $\Q(\wt{G})$. Inspecting the proof of \Cref{ses} we see that we obtain the following decomposition as $\Q[G^{ab}]$-modules:
\begin{equation}
\Q^{id}(\wt{G}) = \big(E^{id} + \sum_{G' \leq H \leq G} \Phi^G_H(\Q^{id}(\wt{H})) \big) \op R^{id}
\end{equation}
where $E^{id}$ and $R^{id}$ are the submodules generated by the idempotent irreducible gliders. The fact that we don't need the nilradical for idempotent elements was already emphasized in the prelude and step 1. It simply remained to notice that when an idempotent lies in $\Ima(\Phi^G_H)$ then it is induced from an idempotent by \cref{irreducible under Ind(Res)}.
\end{remark}

\addtocontents{toc}{\protect\setcounter{tocdepth}{1}}
\section{Interpreting the obstructions with a representation eye}\label{sectie interpretaties}

The current definitions of the modules $E,P$ and $R$ might still seem a bit exotic, making it non-transparant how to check vanishing. The goal of this section is to adjust this by giving descriptions of $P$ and $R$ in terms of classical representation theory.

\subsection{The module \texorpdfstring{$R$}{}}\label{subsectie over R interpretatie}\hspace{1cm}\vspace{0,2cm}\newline
To understand the flavour of $R$ let us start by considering the example of $A_4$. In this case, $R$ can be strictly bigger than $\mathbb{Q}[G^{ab}]$. To see this, recall its character table
$$
\begin{array}{c|rrrr}
  \rm class&\rm1&\rm2&\rm3A&\rm3B\cr
  \rm size&1&3&4&4\cr
\hline
  \rho_{1}&1&1&1&1\cr
  \rho_{2}&1&1&\zeta_3&\zeta_3^2\cr
  \rho_{3}&1&1&\zeta_3^2&\zeta_3\cr
  U &3&-1&0&0\cr
\end{array}
$$
where the conjugacy classes $3A$ and $3B$ consist of the $3$-cycli and the second class is the set $\{ (12)(34),(13)(24),(14)(23) \}$. Together with $\{ 1 \}$ these cycli form the commutator subgroup $A_4'$, furthermore $C_3 \cong \{ 1, (123),(132) \}$ is an example of a non-normal maximal subgroup. Note that $\Res^{A_4}_{C_3}(U) = \op_{i=1}^3 \Res^{A_4}_{C_3}(\rho_i)$ is the decomposition as simple $K[C_3]$-modules. Therefore,
$$\IInd_{C_3}^{A_4}(M_{(\{ T_{C_3} \},\emptyset)}) = M_{(\{ T_{A_4}\}, \{ \ast_U\})},$$
which is again an idempotent element and hence sits in $R(A_4)$.

The example above is not a coincidence and in fact maximal subgroups $H \leq G$ which are not normal always yield idempotent elements in $R$, by considering $\Phi_H^G(M_{(T_H,\emptyset)})$. Even more the existence of elements in $R$ outside $\Q[G^{ab}]$ characterise having such maximal subgroups (which is equivalent to be non-nilpotent).

\begin{theorem}\label{R niet triviaal alleen voor non-nilpotent}
Let $G$ be a finite group. Then,
\begin{center}
$G$ is nilpotent if and only if $R = \Q[G^{ab}]$. 
\end{center}
Moreover, if $G$ is non-nilpotent and $H$ is a  non-normal maximal subgroup, then $\Phi_H^G(M_{(T_H,\emptyset)})= M_{(\{T_G\},D)} \in R \setminus \Q[G^{ab}]$.
\end{theorem}
\begin{proof}
Recall the well-known fact that a finite group is nilpotent if and only if every maximal subgroup is normal. 
Now suppose that $G$ is not nilpotent and let $H$ be a non-normal maximal subgroup. Then $\IInd_H^G(M_{(\{ T_H\} ,\emptyset)})= M_{(\{T_G\},D)}$. Indeed, from (the analogue of) Frobenius reciprocity (\ref{Frobenius reciprocity type iso}) we see that if another $T_z$ with $z \in G^{ab}$ would appear then $H \subseteq \ker(T_z)$.  As $H$ is not normal $G' \nsubseteq H$ and hence due to maximality $G' \, .\,  H=G$. This entails that $G \subseteq \ker(T_z)$ and thus indeed $T_z = T_G$. Furthermore, $D \neq \emptyset$ (since $ H \lneq G$) and hence $M_{(\{T_G\},D)} \in R \setminus \Q[G^{ab}]$ as needed. This finishes the second part and at the same time the necessity of the statement.

Now suppose that $G$ is nilpotent. As shown in \Cref{voorbeeld element in R} if $M_{(A,B)} \in R$, then $|A|\leq 1$.  Now, would the associated idempotent $e(A,B)$ be of the form $M_{(\{ e\}, \emptyset)}$ then is not hard to see from \Cref{lemma with components of cyclic from vector} that also $B = \emptyset$ and consequently $M_{(A,B)}= M_{(\{z\}, \emptyset)}$ for some $z \in G^{ab}$.
Therefore, in order to obtain that $R(G) = \Q[G^{ab}]$ it is enough to prove that if $M_{(\{ e\},D)}$ is an idempotent in $R(G)$, then $D = \emptyset$. We will now prove this by induction on the nilpotency class of $G$.

If $G$ is abelian, then all irreducible representations are linear and hence indeed $D = \emptyset$. Now we consider the case of class two separately. Let $Kv \subset V$ be the glider associated to $M_{(\{e\},D)}$ and suppose that $D$ is non-empty with $U$ a simple submodule of $V$ of dimension at least $2$. By decomposing $V$, in virtue of \Cref{irred}, we may find $u \in U$ that appears in the corresponding decomposition of $v$. Due to the upcoming \Cref{characterization class 2}, there exists an $n > 0$ such that $U^{\ot n}$ completely linearizes. Because $M_{(\{e\},D)}$ is idempotent and by \Cref{irred} it follows that $KG(u \ot \cdots \ot u) \cong T$ which contradicts $\dim(U) \geq 2$. Hence $D = \emptyset$.

Now assume that $G$ is nilpotent of class $n \geq 3$ and consider the lower central series $e = G_n < G_{n_1} < \ldots < G_1 < G$. For all $U \in \Irr(G)$ there exists $0 \leq n_i < n$ such that $G_{n_i} \subseteq Z(U)$ and $G_{n_i-1} \not\subset Z(U)$. Indeed, since $G_n = e \subset \Ker(U)$ it follows that $G_{n-1} \subset Z(U)$. By the $n = 2$ case we know that no $U$ with $n_i = 1$ do not occur in $D$. Suppose that $U$ occurs with $n_i > 1$. There exists $m$ such that all components $V$ of $U^{\ot m}$ are such that $G_{n_i-1} \subset Z(V)$.  Since $M_{(\{e\},D)}$ is idempotent, this shows that at least one $V$ with $n_i-1$ occurs in $D$, contradiction. Hence $D = \emptyset$.
\end{proof}

\begin{conclusion*}
The $\Q[G]^{ab}$-module $R(G)$ is directly connected with the nilpotent property. Besides, the $\Z$-linear combinations of the 'permutation gliders' $\IInd^G_H(M_{e})$ for $H$ a subgroup in the set
$$\mc{S}_{non-l.} = \{ H \leq G \mid H \nsubseteq  \ker(T_z) \text{ for all } z \in G^{ab}\setminus \{e \} \}$$
play a special role. Indeed, for such $H$ the same argument as in the proof of \Cref{R niet triviaal alleen voor non-nilpotent} works and hence $\IInd^G_H(M_{e})$ produces an element in $R \setminus \Q[G^{ab}]$.
\end{conclusion*}

In fact, for such subgroup $H$, for any $z\in H^{ab}$ the permutation glider $\IInd^G_H(M_{z})$ would produce an element in $R(G)$ which however is not idempotent but whose $o(z)$-tensor power would be an idempotent of the form $M_{(\{e\},D)}$. Therefore, it would be especially interesting to try to describe $R$ for monomial groups. Recall that a group $G$ is {\it monomial} if every irreducible character is induced from a linear character of a subgroup. For example, nilpotent and supersolvable groups are monomial \cite[Theorem 6.22]{IsaacsBook}. The following question now arises naturally. \vspace{0,1cm}

\begin{question} \label{question R for monomial} Let $G$ be a monomial group. Is then 
$$R(G) = \Span_{\Q} \{ \IInd^G_H(M_{z}) \mid H \in \mc{S}_{non-l.} \text{ and } z \in H^{ab} \}?$$
\end{question}

\vspace{0,1cm}\noindent Note that due the push-pull formula, for $z \in G^{ab}$ and $h \in H^{ab}$
$$M_z \ot \IInd^G_H(M_h) \cong \IInd^G_H(\RRes^G_H(M_z) \ot M_h) = M_{h'h}$$ 
for some $h' \in H^{ab}$. In other words the right hand side in \Cref{question R for monomial} is also a $\Q[G^{ab}]$-module.

To finish this part we record an immediate consequence of \Cref{R niet triviaal alleen voor non-nilpotent} which follows from the fact that nilpotency is inherited by subgroups.
\begin{corollary}\label{G nilpotent all R are abs}
Let $G$ be a finite nilpotent group and $H\leq G$. Then $R(H) \cong \Q[H^{ab}]$.
\end{corollary}
This corollary will be important from \cref{sectie decompositie zonder de obstructies} on.

\subsection{The module \texorpdfstring{$P$}{}}\label{sectie intepretatie P}

Given a module $U$, taking the sequence of tensor powers $(U^{\ot n})_{n \in \N}$ yields modules with increasing dimension. Now suppose that $U$ is cyclic, say $U = KG u$. Then, with irreducible gliders in mind, we are rather interested in the sequence of modules 
\begin{equation}\label{our sequence of u-tensor powers}
\big( K[G](\underbrace{u \ot \ldots \ot u)}_{n-times} \big)_{n\in \N}
\end{equation}
which lives in the sequence of symmetric powers $(S^n(U))_{n\in \N}$ of $U$. In contrast to these powers, due to \Cref{irred}, the dimensions in (\ref{our sequence of u-tensor powers}) are bounded by above and this sequence might stabilise (e.g. if the associated glider $[(Ku \subset U)]$ is an idempotent). As will be explained below, the stabilising of this sequence is exactly the content of the $K[G^{ab}]$-module $P(G)$.

\subsubsection*{Interpretation}

As illustrated by our parametrisation, \Cref{parametrisatie irreducible}, for non-abelian groups the number of irreducible gliders $M_{(A,B)}$ is infinite. However, by \Cref{irred} the number of multiplicity vectors is still finite and the total dimension of '$M_0$' bounded by above. This fact will now allow us to give a first set of conditions on $A$ or $B$ such that $M_{(A,B)}\notin P$.

First recall that a representation $U$ is said to {\it completely linearize} if it is a direct sum of $1$-dimensional representations.
\begin{proposition}\label{hoe semigroups zonder idempotent vinden}\proplabel{groupcondition}
Let $(A,B) \in \mathcal{P}(G/G') \times \mathcal{S}_G$. If one of the following conditions is satisfied then $\langle M_{(A,B)} \rangle$ contains an idempotent:
\begin{enumerate}
\item $A \neq \emptyset$
\item there exists $V \in \Irr(G)$ such that $B \cap \Gr(V) \neq \emptyset$ and $V^{\otimes n}$ completely linearizes for some $n$
\end{enumerate}
\end{proposition}
\begin{proof}
Suppose first that $A \neq \emptyset$, then by \cref{AB in E} there exists some $n$ (e.g. $|G|$) such that $M_{(A,B)}^{\ot n} = M_{(C,D)}$ with $1 \in C$. Consider now the sequence
$$M_{(C,D)}, M_{(C,D)}^{\ot 2}, M_{(C,D)}^{\ot 4}, \dots, M_{(C,D)}^{\ot 2^n},\ldots$$
Again by \Cref{AB in E} the multiplicity vector $\alpha(n)$ of $M_{(C,D)}^{\ot 2^n}$ is an increasing function in $n$. However, since there are only a finite number of multiplicity vectors, this sequence must stabilize. Now if $D= \emptyset$ then we obtain the idempotent $M_{(C,D)}$. If $\emptyset \neq D \cap \Gr(j,V) = \{a_1,\ldots, a_j\}$ for some $V \in \Irr(G)$, then because $T$ appears in $M_{(C,D)}$ we have that at least $\{a_1,\ldots,a_j\}$ appears in $M_{(C,D)}^{\ot 2^n}$ for all $n$. Therefore from the moment on that the sequence stabilize we obtain an idempotent $M_{(E,F)} \in \langle M_{(A,B)} \rangle$.

For the second part it suffices to consider the case that $A = \emptyset$.  Consider now $U \in \Irr(G)$ and $n$ as in the statement. Then $M_{(A,B)}^{\ot n} = M_{(A',B')}$ with $A' \neq \emptyset$ and hence we are finished by the first part.
\end{proof}

For an irreducible glider $(Ku \subset U)$ The proof of the previous statement shows the importance of detecting the presence of the trivial representation in the submodule $KG(u \ot \cdots \ot u)$ of some tensor power $U^{\ot n}$. Recall by Burnside-Brauer's theorem that for $N = \ker(U)$ one has that
\begin{equation}\label{some tensor power has trivial}
\langle T_G, U^{\ot n} \rangle_G = \langle T_{G/N}, \ov{U}^{\ot n} \rangle_{G/N} \neq 0
\end{equation}
for some (potentially large) $n$ since $\ov{U}$ is faithful as $K[G/N]$-module (see \cite[Theorem 4.3.]{IsaacsBook}). Combined with the methods of the proof of \propref{groupcondition} we now obtain the following interpretation of the module $P$.

\begin{obstruction}
Given $U \in \Irr(G)$, there exists by (\ref{some tensor power has trivial}) an $n \in \mathbb{N}$ such that the trivial $G$-representation appears in the decomposition of $U^{\ot n}$. Working with $(K \subseteq KG)$-glider representations, however, requires keeping track of a vector $u \in U$ and by definition 
 $$\left[ \Big( U \supseteq Ku\Big)\right]^{\ot n} =\left[ \Big(KG(u\ot \cdots \ot u) \supseteq Ku \ot \cdots \ot u\Big) \right].$$ In general, $KG(u \ot \cdots \ot u) \subsetneq U^{\ot n}$. If nevertheless we can ensure that $T_G$ appears in the decomposition of $KG(u \ot \cdots \ot u)$, then $M_{(U \supseteq Ku)} \notin P$. 
 \end{obstruction}
 
In the interpretation we could as well have replaced $T_G$ by another one-dimensional $G$-representation $S$. 

\begin{corollary}\corlabel{P=0}
Let $G$ be a finite group and $M_{(A,B)} \in \mathbb{Q}(\wt{G})$. If there exists $U \in \Irr(G)$ such that $B \cap \Gr(U) = \{\ast_U\}$, then $M_{(A,B)} \notin P$.
\end{corollary}
\begin{proof}
Take $n$ as in (\ref{some tensor power has trivial}) such that $T_G \leq U^{\ot n}$. Because we have the liberty of choosing vectors $u_1, \ldots, u_{\dim(U)}$ in $U^{\oplus \dim(U)}$, we can choose them appropriately such that $T_G$ indeed appears in the decomposition of
$$KG(u_1 + \cdots + u_{\dim(U)})^{\ot n}.$$
\end{proof}

Intriguingly, we were unable to find a group with non-vanishing $P$. Hence,

\begin{question}
Does there exist a finite group $G$ such that $P(G) \neq 0$? 
\end{question}

We expect for solvable groups the answer to be positive, however in general we are dubious. Natural candidates for non-vanishing $P$ are simple groups since the trivial representation is the only $1$-dimenionsal representation.

\subsubsection*{Applications of the interpretation}\hspace{0,2cm}\vspace{0,1cm}\newline
We can embed $\mathbb{Q}(\wt{H})$ in $\mathbb{Q}(\wt{G})$ via $\Phi_H^G$. Due to this we directly obtain that vanishing of the module $P$ is inherited by subgroups. 
\begin{proposition}\label{vanishing P inherited by subgroups}
Let $G$ be a finite group such that $P(G) = 0$, then $P(H) = 0$ for all subgroups $H \leq G$.
\end{proposition}

Now we will give a first non-trivial application of the interpretation of the module $P$ obtained earlier. More concretely, 
\begin{proposition}
Let $G$ be a group with an abelian subgroup $H$ of index $2$. Then $P(G) = 0$.
\end{proposition}
\begin{proof}
By Ito's theorem, \cite[Theorem 6.15.]{IsaacsBook}, we know that $\dim(U) \mid [G:H] = 2$ for any $U \in \Irr(G)$. Let $U \in \Irr(G)$ be $2$-dimensional and decompose it in its symmetric and antisymmetric part:  $U \ot  U = S(U \ot U) \oplus A(U\ot U)$

We know that $u \ot u \in S(U\ot U)$ so $KG(u \ot u) \subseteq S(U\ot U)$. Recall that $\dim S^r
(U) = {\dim(U) +r-1 \choose r}$. In particular, $\dim S(U \ot U) = 3$. Hence, either it is the direct sum of three $1$-dimensional subrepresentations, in which case \Cref{hoe semigroups zonder idempotent vinden} yields the desired conclusion, or 
$$S(U \ot U) \cong T_1 \oplus V.$$
with $\dim V = 2$. It remains to consider the later case. For this, fix a basis of $U$ such that 
$$\Res^G_H(U) \cong S \oplus S' = Ks \oplus Kt$$
as $H$-representations. By \cite[Proposition 20.5]{JL} $S \not\cong S'$, since $[G:H] = 2$. \vspace{0,1cm}

{\bf Claim:} $S \ot S \not\cong S' \ot S'$ as $H$-representations.
\begin{proof}
Suppose that $S \ot S \cong S' \ot S'$. In that case $\Res^G_H(U\ot U) \cong (S \ot S)^{\op 2} \op (S \ot S')^{\op 2}$. Now write $\Res^G_H(V) = T' \oplus T''$ as $H$-representation and as before we know that $T' \not\cong T''$.  Whence 
$$\Res^G_H(V) \cong (S \ot S) \oplus (S \ot S').$$
Consequently, $V$ and $T_1$ lie over the $H$-representation $S \ot S'$, which by \cite[Proposition 20.11 \& 20.12]{JL} implies that $\dim(V) = 1$, a contradiction.
\end{proof}

Let $u = \lambda s + \mu t$, then 
$$u \ot u = \lambda^2 s \ot s + \lambda\mu(s \ot t + t \ot s) + \mu^2 t \ot t.$$
If $\lambda\mu \neq 0$, then $KG(u \ot u)$ must be 3 dimensional and it reaches a one-dimensional representation, which is sufficient to show that $\langle M_{(Ku \subset U)} \rangle$ contains an idempotent. If $\lambda\mu = 0$, then, say, $KHu \cong S$. In this case $KG(u\ot u)$ is 2 dimensional so isomorphic to $V$. Decompose $V$ as $H$-representation
$$\Res^G_H(V) = W \oplus W' = Kw \oplus Kw'.$$
We remark that this decomposition is unique: for $h \in H$, write $h \cdot w = c(h)w$ and $h \cdot w' = d(h)w'$. Since $W \not\cong W'$, there exists $h \in H$ such that $c(h) \neq d(h)$. If $\alpha w + \beta w'$ is such that $KH(\alpha w + \beta w') \cong W$, then on the one hand we have
$$h \cdot (\alpha w + \beta w') = \alpha c(h) w + \beta d(h)w'.$$
and on the other hand
$$h \cdot (\alpha w + \beta w') = \gamma \alpha(w) + \gamma \beta w'.$$
Up to rescaling, it follows that $\gamma = c(h) = d(h)$, contradiction. Since $H$ is abelian, we can represent $W$ and $W'$ by elements $h$ and $h'$ of $H$. If $h^2 = h'^2$, then by the second claim $S(V \ot V)$ must be of the form $T_1 \oplus T_2 \oplus T_3$ and we see that $(Ku \subset U )^{\ot 4} = (Kv \subset V)^{\ot 2}$ reaches a one-dimensional representation, which suffices to conclude the existence of an idempotent element. If $h^2 \neq (h')^2$ then one looks at $S(V \ot V)$ and restarts the reasoning: if $S(V \ot V) = T_1 \oplus T_2 \oplus T_3$, one concludes. In the other case, $S(V \ot V) = T' \oplus V'$, we check the dimension of $KG(v \ot v)$. If it is 3, we are done, if it is 2, then $KG(v \ot v) = V'$ and 
$$\Res^G_H(V') \cong W^{\ot 2} \oplus (W')^{\ot 2}.$$
But these $H$-representations correspond to $h^4, (h')^4$ respectively. If both elements are equal, one concludes, otherwise one restarts. Since $H$ is finite abelian, there exists $n \geq 1$ such that $h^{2^n} = (h')^{2^n}$ so the above argument stops and we conclude.

For an arbitrary glider representation $(Kv \subset V)$ we know that if an irreducible representation $U$ of dimension 2 appears in the decomposition of $V$, a certain tensor power reaches a one dimensional representation and we can deduce the existence of an idempotent element in $\langle M_{(Kv \subset V)} \rangle$. If all appearing representations in $V$ are 1 dimensional, we are working in $\mathbb{Q}[G^{ab}]$, which is finite dimensional. Hence we have shown that $P=0$.
\end{proof}

\begin{remark}
Amitsur \cite{Am} classified all groups having all irreducible representations of dimension bounded by $2$. His classification consists of three subclasses: (1) abelian groups; (2) certain groups of nilpotency class $2$ and (3) groups having an abelian subgroup of index $2$. In \cref{sectie class 2} we will handle arbitrary groups of nilpotency class $2$. Hence the groups in (3) remain and this was one of the original motivations to apply the interpretation to the groups above. 
\end{remark}

\subsection{The module \texorpdfstring{$E$}{}}\label{sectie intepretatie E}

As mentioned earlier, we do not know whether idempotent elements of the form $M_{(\emptyset,D)}$ can exist. What we do have is the following.

\begin{proposition}\label{E subgroups}
If $E(G) = 0$ then $E(H) = 0$ for all subgroups $H \leq G$.
\end{proposition}
\begin{proof}
Suppose that $M_{(\emptyset,D)}$ is an idempotent in $\mathbb{Q}(\wt{H})/N$. Since no one-dimensional $G$-representation can lie over an irreducible $H$-representation of dimension $>1$, $\IInd_H^G(M_{(\emptyset,D)})$ would be in $E(G)$, contradiction.
\end{proof}

Unfortunately another interpretation of $E(G)$ seems to be out the reach of our current methods.
\begin{question}
Does there exist a group with $E(G)$ non-trivial? If yes, what is an interpretation in terms of $\rep_K(G)$?
\end{question}

\addtocontents{toc}{\protect\setcounter{tocdepth}{2}}
\section{Precise description semisimple part \texorpdfstring{$\mathbb{Q}\otimes \RRepR_1(\widetilde{G})$}{} under vanishing obstructions}\label{sectie decompositie zonder de obstructies}

Let $G$ be a finite group and $E,P,R$ the $\mathbb{Q}[G^{ab}]$-modules from \Cref{ses}. The main aim of this section is to prove the following result
\begin{theorem*}
Let $G$ be a finite nilpotent group with $P = 0 = E$. Then there is an isomorphism
$$\varphi: \bigoplus_{H \in \Sub(G)} \mathbb{Q}[H^{ab}] \to \mathbb{Q}(\wt{G})/N $$
of $\Q[G^{ab}]$-algebras for a certain class of subnormal subgroups $\Sub(G)$. Moreover, $N( \mathbb{Q}(\wt{G})) = J( \mathbb{Q}(\wt{G}))$ coincide.
\end{theorem*}
This result will be achieved in \Cref{gencharring}. To construct the morphism $\varphi$ we have to introduce new objects. To start, the class $\Sub(G)$ will be defined in \Cref{subsectie indres chain en def HCD} along with other tools. More precisely, to any idempotent $M_{(C,D)} \notin E(G)$ we will associate a chain of groups $\Chain(C,D)$ and a set of idempotents $\Idemp(C,D)$. Subsequently in \Cref{subsectie ortho idemp} we will connect to such a set of idempotent another idempotent $\epsilon(C,D)$. All elements constructed in this way will form an orthogonal family of idempotents and will finally allow us to define the isomorphism $\varphi$.

In \Cref{sectie class 2} we will prove that for nilpotent groups of class $2$ indeed the condition $P=0=E$ is satisfied. However we expect these vanishing conditions to be fulfilled for any nilpotent group.  Remark also that the $\Q(\wt{G})\setminus P$ is multiplicatively closed and in fact is again a $\Q[G^{ab}]$-algebra. In fact the methods of this section describe this algebra (which equals $\Q(\wt{G})$ whenever $P=0$).

\subsection{ \texorpdfstring{$\IInd \RRes$}{}-chain connected to an idempotent}\label{subsectie indres chain en def HCD}

Let $M_{(C,D)} \in \Irr_1(F(K[G]))$ be an idempotent with $C \neq \emptyset$. In particular $1 \in C \leq G^{ab}$ is a subgroup by \ref{AB in E}. We will now explain a procedure that associates a chain $\Chain(C,D)$ of subgroups of $G$ whose members will have associated smaller idempotents (in the sense of total dimension). 

To start, if $C = \{e\}$ then $M_{(C,D)} \in R(G)$ and $\Chain(C,D) = \{ \mc{L}(C)= G \}$. If  $C \neq \{e\}$, then $G' \leq \mc{L}(C) \lneq G$ and using \cref{change multiplicity under tensor product} we form
$$M_{(C,D)} = \IInd_{\mc{L}(C)}^G(\RRes_{\mc{L}(C)}^G(M_{(C,D)})) = \IInd_{\mc{L}(C)}^G(M_{(C_1,D_1)}).$$
Note that, \cref{irreducible under Ind(Res)}, $M_{(C_1,D_1)} \in \Q(\wt{\mc{L}(C)})$ is an idempotent with $ \emptyset \neq C_1 \leq \mc{L}(C)^{ab}$. Now for notational simplicity denote $H_1 := \mc{L}(C)$.  If $C_1 = \{e\} $ we stop and $\Chain(C,D) = \big( G \gneq  H_1 \big)$. If $C_1 \neq \{e\}$, then $(H_1)' \leq \mc{L}(C_1) \lneq H_1$ and we form
$$ M_{(C_1,D_1)} =  \IInd_{\mc{L}(C_1)}^{H_1}(\RRes_{\mc{L}(C_1)}^{H_1}(M_{(C_1,D_1)})) = \IInd_{\mc{L}(C_1)}^{H_1}(M_{(C_2,D_2)}) $$
where $M_{(C_2,D_2)} \in \Q(\wt{\mc{L}(C_1)})$ with $\emptyset \neq C_2 \leq \mc{L}(C_1)^{ab}$. Denote $H_2 = \mc{L}(C_1)$. If $C_2 = \{ e \}$, then we stop and $\Chain(C,D)= \big( G \gneq H_1 \gneq H_2 \big)$ and otherwise we continue this procedure. Since $(H_{i-1})' \leq H_{i} \lneq H_{i-1}$ the order of $H_{i}$ decreases at every step, so at some point $C_i$ must become $\{e\}$. Suppose that it stops after $m$ steps, then we obtained a descending chain
$$\Chain(C,D) := \big( G=H_0 \gneq H_1 \gneq \cdots \gneq H_m \big)$$
with an associated sequence of idempotents
$$\Idemp(C,D) := \{ M_{(C,D)}=M_{(C_0, D_0)} , M_{(C_1, D_1)}, \ldots, M_{(C_m,D_m)} \}$$
such that $\emptyset \neq C_i \leq H_i^{ab}, \, H_{i-1}' \leq H_i = \mc{L}(C_{i-1}) \lneq H_{i-1}$ and $C_m = \{ e \}$. The last member of the chain $H_m$ will be denoted $H(C,D)$. Collecting all such subgroups we obtain the set
$$\Sub(G) := \{ H(C,D) ~\vline~ M_{(C,D)} \in \mathbb{Q}(\wt{G})\setminus E(G) {\rm~idempotent}\}.$$

Note that by construction all subgroups $H(C,D)$ are subnormal in $G$. Moreover, inspecting every step we see easily that the idempotents produced above enjoy the following properties.
\begin{proposition}\label{eigenschappen leden chain C-D}
With notations as above, we have that
\begin{enumerate}
\item[(i)] $M_{(C_i,D_i)} = \RRes^{H_{i-1}}_{H_i} (M_{(C_{i-1},D_{i-1})})=\RRes^G_{H_i}(M_{(C,D)})$,\vspace{0,1cm}
\item[(ii)] $\dim (M_{(C_{i-1},D_{i-1})}) \geq \dim (M_{(C_{i},D_{i})})$,\vspace{0,1cm}
\item[(iii)] $M_{(C,D)} = \IInd_{H_i}^G (M_{(C_{i},D_{i})}) = \IInd_{H_i}^G\RRes_{H_i}^G(M_{(C,D)})$,\vspace{0,1cm}
\item[(iv)] $M_{(C_m,D_m)} \in R(H_m)$ (i.e. $C_m = \{ T_{H_m} \}$)
\end{enumerate}
for all $1 \leq i \leq m$.
\end{proposition}
\begin{proof}
Concerning assertion (i), $M_{(C_i,D_i)} = \RRes^{H_{i-1}}_{H_i}(M_{(C_{i-1},D_{i-1})})$ is by definition and now using this recursively, via transitivity of restriction, we see that also $M_{(C_i,D_i)} = \RRes^G_{H_i}(M_{(C,D)})$. Part (ii) follows from (i) (note that dimensions are not equal as $K[H]u$ is often smaller than $\Res^G_H(K[G]u)$). Part (iv) is by definition of $m$ (the last step of the construction).

For (iii), recall from \cref{change multiplicity under tensor product} that $M_{(C_i,D_i)} = \IInd_{\mc{L}(C_i)}^{H_{i}}\RRes_{\mc{L}(C_{i})}^{H_{i}}(M_{(C_i,D_i)}) $. Using this recursively, combined with (i), we find that 
\begin{equation}
\begin{split}
M_{(C,D)} & \equiv  \IInd_{H_1}^G(M_{(C_1,D_1)}) \\
& \equiv \IInd_{H_1}^G \IInd_{H_2}^{H_1} \RRes_{H_2}^{H_1} (M_{(C_1,D_1)}) \\
& \equiv \IInd^G_{H_2} (M_{(C_2,D_2)})=  \IInd_{H_2}^G\RRes^G_{H_2}(M_{(C,D)})\\
& \equiv \ldots \\
&  \equiv \IInd_{H(C,D)}^G\RRes^G_{H(C,D)}(M_{(C,D)})\\
\end{split}
\end{equation}
this finishes the proof of (iii)
\end{proof}

Now consider an element $M_{(A,B)} \in \mathbb{Q}(\wt{G})$ whose associated idempotent is $M_{(C,D)}$. Then $M_{(A_i,B_i)} := \RRes^G_{H_i}(M_{(A,B)})$ is such that it has an associated idempotent, namely $M_{(C_i,D_i)}$. Note that also $M_{(A_i,B_i)} := \RRes^{H_{i-1}}_{H_i}(M_{(A_{i-1},B_{i-1})})$. By (\ref{stap 2 bewijs SES}) 
$$M_{(A_i,B_i)} \equiv \IInd_{\mc{L}(C_i)}^{H_{i}}\RRes_{\mc{L}(C_{i})}^{H_{i}}(M_{(A_i,B_i)}) \mod N(\mathbb{Q}(\wt{H_i}))$$ 
for all $0 \leq i \leq m-1$ where $\mc{L}(C_{i}) = H_{i+1}$. With a reasoning as in the proof of part (iv) in \Cref{eigenschappen leden chain C-D} one proves that for all $i$
\begin{equation}\label{irr glider indres vanuit leden keten}
\begin{split}
M_{(A,B)} & \equiv \IInd_{H_i}^G\RRes^G_{H_i}(M_{(A,B)})\\
&  \equiv \IInd_{H(C,D)}^G\RRes^G_{H(C,D)}(M_{(A,B)})\\
\end{split}
\end{equation}

Now define the $\Q$-vector space map
\begin{equation}\label{the map as modules}
\psi :\bigoplus_{H \in \Sub(G)} R(H) \rightarrow \Q(\wt{G})/N :  \sum_{H \in \Sub(G)}  M_{(A_H,B_H)} \mapsto \sum_{H \in \Sub(G)}  \IInd_H^G(M_{(A_H,B_H)})
\end{equation}
where $M_{(A_H,B_H)} \in R(H)$. It follows from \cref{phi is algebra map} that it is a morphism of $\Q[G^{ab}]$-modules. Now if $E=0=P$ it follows from (\ref{irr glider indres vanuit leden keten}) and \cref{eigenschappen leden chain C-D} that $\psi$ is surjective. If now we moreover assume that $G$ is nilpotent, then $R(H) = \Q[H^{ab}]$ by \cref{G nilpotent all R are abs} and in particular $\dim_{\Q} \Q(\wt{G})/N \leq \dim_{\Q }\bigoplus_{H \in \Sub(G)} R(H) < \infty$. Under this extra assumption it is not difficult to prove that $\psi$ is in fact an isomorphism, however it is not a ring morphism. To resolve this problem we will now introduce in the next section a family of orthogonal idempotents.

\subsection{A family of orthogonal idempotents of \texorpdfstring{$\Q(\wt{G})$}{}}\label{subsectie ortho idemp}
Let $G' \leq H \leq G$ be a normal subgroup and consider the set of all maximal subgroups of $H/G'$:
$$\mc{M}(H/G') = \{ G' \leq L  \lneq H \mid L \text{ maximal in } H \}.$$ 
In other words $\mc{M}(H/G')$ consists of all subgroups $L$ such that $A_\iota(L)$ is minimal over $A_\iota(H)$. Associated to this set we define the element
$$\epsilon(H,G) = \prod\limits_{L \in \mc{M}(H/G')}(M_{(A_\iota(H),\emptyset)} - M_{(A_\iota(L),\emptyset)}) \in \RRepR_1(\wt{G}).$$
which is easily seen to be idempotent. When the group $G$ is clear from the context {\it we simply write $\epsilon(H)$.}  We will use this idempotent writing in the following form:
$$\epsilon(H,G) =M_{(A_\iota(H),\emptyset)}\, . \, \prod_{L \in \mc{M}(H/G')}(M_{(\{ T_G\}, \emptyset)} - M_{(A_\iota(L),\emptyset)}).$$

Now, let $M_{(C,D)} \in \Irr_1(K\subset K[G])$ be an idempotent with associated chain of subgroups
$$\Chain(C,D) = \big( G = H_0 \gneq H_1\gneq \ldots \gneq H_m = H(C,D) \big).$$
Recall that $(H_i)' \leq H_{i+1} < H_i$, for $0 \leq i \leq m-1$, and thus $\epsilon(H_{i+1}, H_i) \in \mathbb{Q}(\wt{H_i})$ and so for these elements the overgroup $H_i$ is intrinsic to the definition and so we write compactly $\epsilon(H_{i+1})$. To this data we attach the idempotent
$$\epsilon(C,D) = \prod_{i=0}^{m-1} \IInd_{H_{i}}^G(\epsilon(H_{i+1})) \IInd_{H_{m}}^G(\epsilon(H_{m}, H_m)).$$
Observe that $\IInd^G_{H_0}(\epsilon(H_1)) = \epsilon(H_1)$ and that the final factor of the product equals
$$ \IInd_{H_{m}}^G(\epsilon(H_{m}, H_m))= \IInd_{H_m}^G\big( \prod_{L \in \mc{M}(H_m/H_m')}(M_{(\{T_{H_m}\}, \emptyset)} - M_{(A_\iota(L), \emptyset)}) \big).$$
We will now prove that the elements $\epsilon(C,D)$ for different subgroups $H(C,D)$ are orthogonal. Afterwards we will obtain a kind of canonical form in terms of the basis.

\begin{proposition}\label{orthogonal idempotents}
Let $H(C,D),H(C',D') \in \Sub(G)$. Then
$$H(C,D) \neq H(C',D')\,  \, \, \text{ if and only if }\, \, \, \epsilon(C,D)\epsilon(C',D') = 0$$
\end{proposition}
\begin{proof}
Since $\epsilon(C,D)$ is an idempotent, orthogonality is clearly a sufficient condition to distinguish $H(C,D)$ and $H(C',D')$. Now, for the converse denote the subgroups appearing in $\Chain(C,D)$ (resp. $\Chain(C',D')$) by $H_i$ with $0 \leq i \leq n$ (resp. by $F_j$. with $0 \leq j \leq m$).

 Without loss of generality we may assume that $n\leq m$. Clearly there is a smallest non-zero $i$ such that $F_i = H_i$ and $F_{i+1} \neq H_{i+1}$. To start suppose that $i < n$.

We have that $H_i' \leq H_{i+1} \neq F_{i+1} \lneq H_i$. We may choose a subgroup $Y \in \mc{M}(H_{i+1}/H_i')$ such that $H_{i+1} \cap F_{i+1} \leq Y \lneq H_i$. Or in other words $A_\iota(Y) \leq A_\iota(H_{i+1} \cap F_{i+1})$. 

 {\it We claim that:}
 $$M_{(A_\iota(H_{i+1}),\emptyset)}\, . \, M_{(A_\iota(F_{i+1}),\emptyset)} ( M_e - M_{(A_\iota(Y),\emptyset)})( M_e - M_{(A_\iota(Z), \emptyset)}) = 0.$$
 in $\Q(\wt{H_i}) =\Q(\wt{F_i})$ (and where $e = T_{H_i} = T_{F_i}$). Consequently, $\epsilon(H_{i+1})\epsilon(F_{i+1}) =0$ and hence the product $\epsilon(C,D)\epsilon(C',D')$ which contains the factor $$\IInd_{H_i}^G(\epsilon(H_{i+1}))\IInd_{F_i}^G(\epsilon(F_{i+1})) = \IInd_{H_i}^G(\epsilon(H_{i+1})\epsilon(F_{i+1})) = 0.$$
 is also zero, as desired.

{\it To prove the claim:} We have that $M_{(A_\iota(H_{i+1}),\emptyset)}\, . \, M_{(A_\iota(F_{i+1}),\emptyset)} = M_{(A_\iota(H_{i+1})A_\iota(F_{i+1}), \emptyset)}$ and so by (\ref{meet and join for A and L maps}) equal to $M_{(A_\iota(H_{i+1} \cap F_{i+1}), \emptyset)}$. Due to the choice of $Y$ and (\ref{meet and join for A and L maps}), $M_{(A_\iota(H_{i+1} \cap F_{i+1}), \emptyset)}( M_e - M_{(A_\iota(Y),\emptyset)}) = 0$ which is a stronger version of the claim.\vspace{0,1cm}

Finally consider the case that $i = n$ (which can only occur if $n \lneq m$). Hence $F_n = H_n$ but $H_n \neq H_{n+1}$. In the product of $\epsilon(C,D)\epsilon(C',D')$ there appears 
$$\IInd_{F_n}^G(\epsilon(F_n))\IInd_{H_n}^G(\epsilon(H_{n+1})) = \IInd_{H_n}^G(\epsilon(H_n)\epsilon(H_{n+1}))$$
 which equals
$$\IInd_{H_n}^G\big(\prod_{L \in \mc{M}(H_n/H'_n)}(M_{(\{e\},\emptyset)} - M_{(A_\iota(L),\emptyset)})\prod_{Q \in \mc{M}(H_{n+1}/H'_n)}(M_{(A_\iota(H_{n+1}),\emptyset)} - M_{(A_\iota(Q),\emptyset)})\big).$$
Now consider a maximal subgroup $L$ of $H_n$ containing $H_{n+1}$. Then for any $Q \in \mc{M}(H_{n+1}/H'_n)$ we have that
$$\big( M_{(\{e\},\emptyset)} - M_{(A_\iota(L),\emptyset)} \big) . \big( M_{(A_\iota(H_{n+1}),\emptyset)} - M_{(A_\iota(Q),\emptyset)} \big) = 0.$$
\end{proof}

Under an extra assumption that $G$ is nilpotent we now obtain a canonical decomposition of $\epsilon(C,D)$ where we see that it contains $M_{(C,D)}$ as a distinguished summand. This will be key in the proof of the injectivity of \Cref{gencharring}. 

\begin{proposition}\label{form of idempotents}
Let $G$ be a finite nilpotent group and $M_{(C,D)} \in \Irr_1(K \subset K[G])\setminus(E \op P)$ an idempotent. Then, 
$$\epsilon(C,D) = M_{(C,D)} + \sum_i^{'} c_iM_{(C_i,D_i)}$$
with $c_i \in \mathbb{Z}$ and $C \lneq C_i$.
\end{proposition}
\begin{proof}
Denote $\Chain(C,D) = (G = H_0 \gneq H_1 \gneq \ldots \gneq H_m = H(C,D))$ where $H_i' \leq H_{i+1}= \mc{L}(C_i) \lneq H_i$. For $0 \leq i \leq m-1$ we will now rewrite 
$$\epsilon(H_{i+1}) = M_{(A_{\iota}(H_{i+1}),\emptyset)} \prod_{L \in \mc{M}(H_{i+1}/H_i')} (M_{(\{T_{H_i}\},\emptyset)} - M_{(A_{\iota}(L),\emptyset)}).$$
To start denote $\mc{M}(H_{i+1}/H_i') = \{L_1, \ldots, L_{t_i} \}$. Then using  (\ref{meet and join for A and L maps}), we can work out the product as 
$$\prod_{L \in \mc{M}(H_{i+1}/H_i')} (M_{(\{T_{H_i}\},\emptyset)} - M_{(A_{\iota}(L),\emptyset)}) = M_{(\{T_{H_i}\},\emptyset)} \, +  \sum_{\substack{1 \leq j_1, \ldots, j_k \leq t_i \\ 1 \leq k \leq t_i}} z_{j_1,\ldots,j_k} M_{(A_{\iota}(L_{j_1} \cap \ldots \cap L_{j_k}), \emptyset)}$$
with all coefficients $z_{j_1,\ldots,j_k} \in \Z$. Note that $A_\iota(H_{i+1}) \lneq A_{\iota}(L_{j_1} \cap \ldots \cap L_{j_k})$ for all possible choices of $L_j$. Hence 
\begin{equation}\label{ontbinding epsilon van de H-i}
\epsilon(H_{i+1}) = M_{(A_{\iota}(H_{i+1}),\emptyset)} \, +  \sum_{\substack{1 \leq j_1, \ldots, j_k \leq t_i \\ 1 \leq k \leq t_i}} z_{j_1,\ldots,j_k} M_{(A_{\iota}(L_{j_1} \cap \ldots \cap L_{j_k}), \emptyset)}.
\end{equation}
Now recall that $M_{(C,D)}=\IInd^G_{H_i}(M_{(C_i,D_i)})$ by \Cref{eigenschappen leden chain C-D} and thus $\IInd^G_{H_i}(M_{(C_i,\emptyset)})= M_{(C,F_{i})}$ for some $F_i \in \mc{S}_G$. With methods as in the proof of \cref{irreducible under Ind(Res)} it follows that moreover   $F_i \subseteq D$.
Therefore
\begin{equation}\label{ontbinding ind van epsilon van de H-i}
\IInd_{H_i}^G(\epsilon(H_{i+1})) = M_{(C,F_{i+1})} + \sum_{C\leq C' \leq G^{ab}} z_{C'} M_{(C',F')}
\end{equation}
with $F_{i+1} \subseteq D$ and some $F' \in \mc{S}_G$ depending on $C'$ and $H_i$. For $i=0$ we can say more, indeed $\IInd^G_{H_0}(\epsilon(H_1)) = \epsilon(H_1)$ and hence  the decompositions (\ref{ontbinding epsilon van de H-i}) and (\ref{ontbinding ind van epsilon van de H-i}) coincide. In particular, for $i=0$ only $C'$ strictly larger than $C$ appear in the summation of (\ref{ontbinding ind van epsilon van de H-i}). Note that at least one $z_{C'}$ is non-zero for $i=0$ (namely the terms corresponding to $M_{(A_{\iota}(L_j),\emptyset)}$).

Finally, the last term $\IInd_{H_m}^G(\epsilon(H_m, H_m))$ can be rewritten in exactly the same form as (\ref{ontbinding ind van epsilon van de H-i}) but with the difference that the first summand will be $M_{(C,D)}$. Indeed, as $G$ is nilpotent, $\RRes^G_{H_m}(M_{(C,D)}) \in R(H_m) = \Q[H_m^{ab}]$ by \Cref{G nilpotent all R are abs}. Hence $M_{(C_m,D_m)} = M_{(\{T_{H_m}\}, \emptyset)}$ and consequently $\IInd_{H_m}^G(M_{(\{T_{H_m}\}, \emptyset)}) = M_{(C,D)}$ by \Cref{eigenschappen leden chain C-D}. If we now take the product of all expressions of the type (\ref{ontbinding ind van epsilon van de H-i}) found, using the extra properties of $\epsilon(H_1)$ and $\IInd^G_{H_m}(\epsilon(H_m,H_m))$ we find that $\epsilon(C,D)$ has indeed the stated form. To see this, use \Cref{change multiplicity under tensor product} to deduce that $M_{(C,D)} M_{(C,F_{i+1})} = M_{(C,D)}$ and for $C \lneq C'$ and $F'' \subseteq D$ also $M_{(C', \emptyset)} . M_{(C'', F'')} = M_{(A,B)}$ with $C \lneq C'\, .\, C'' \leq A$.
\end{proof}

\subsection{Proof of \texorpdfstring{\Cref{gencharring} and the Jacobson radical}{}}

Assume that $G$ is nilpotent. Then $R(H) = \mathbb{Q}[H^{ab}]$ for all $H \leq G$ by \Cref{G nilpotent all R are abs}. In particular it has only one idempotent, namely $M_{(\{T_H\},\emptyset)}\in \mathbb{Q}(\wt{H})$ and so if $M_{(C,D)}\notin E$ is idempotent then $M_{(C,D)} = \IInd^G_{H(C,D)}(M_{T_{H(C,D)}})$. On the other hand if $H \in \Sub(G)$, then there exists an idempotent $M_{(C,D)}$ such that $H(C,D) = H$ and by the previous in fact there is a unique possibility for $M_{(C,D)}$, namely $ \IInd^G_H(M_{T_H})$. In other words, we have the following bijection:
\begin{equation}\label{bijectie tussen sub en idempotenten}
\arraycolsep=1.6pt\def\arraystretch{1.5}
\begin{array}{rcl}
\Sub(G) & \xleftrightarrow{1-1} & \{ M_{(C,D)} \notin E(G) \text{ idempotent } \} \\
H & \longmapsto & \IInd^G_H(M_{(\{T_H\},\emptyset)})\\
\end{array}
\end{equation}
Now consider two idempotents $M_{(C,D)}$ and $M_{(E,F)}$ not in $E(G)$ and such that $\epsilon(C,D)= \epsilon(E,F)$. From the canonical form (i.e. \cref{form of idempotents}) and the fact that $\Irr_1(K \subset K[G])$ is a basis of $\Q(\wt{G})$ we readily deduce that the latter is only possible if in fact $M_{(C,D)} = M_{(E,F)}$. So we have one bijection more:
$$
\Sub(G)\xleftrightarrow{1-1}  \{ M_{(C,D)} \notin E(G) \text{ idempotent } \} \xleftrightarrow{1-1} \{ \epsilon(C,D)\}.
$$
We denote the element $\epsilon(C,D)$ corresponding to $H \in \Sub(G)$ through the composition of these bijections by $\epsilon_H$. We are now ready to state the exact form of the second main structural theorem of this paper. 

\begin{theorem}\label{gencharring}
Let $G$ be a finite nilpotent group with $P = 0 = E$. Then  
$$\varphi: \bigoplus_{H \in \Sub(G)} \mathbb{Q}[H^{ab}] \longrightarrow \mathbb{Q}(\wt{G})/N: \sum_{H \in \Sub(G)}  \alpha_H \mapsto \sum_{H \in \Sub(G)}  \IInd_H^G(\alpha_H)\epsilon_H$$
where $\alpha_H = \sum\limits_{h \in H^{ab}} q_h M_{(\{ h\},\emptyset)} \in R(H)=\mathbb{Q}[H^{ab}]$, is an isomorphism of $\Q[G^{ab}]$-algebras.
\end{theorem}
\begin{proof}
As $\Q[H^{ab}]$ must be interpreted as $R(H) \subset \Q(\wt{H})$ and the action on the direct sum $\bigoplus_{H \in \Sub(G)} \mathbb{Q}[H^{ab}]$ is the diagonal one, the assertion that $\varphi$ is a $\Q[G^{ab}]$-algebra map follows by definition (or directly from \Cref{phi is algebra map}).

Suppose now that $x \in \Ker(\varphi)$, then also $\varphi(x)\epsilon_H=0$ for all $H\in \Sub(G)$. Hence by \Cref{orthogonal idempotents} we may assume that $x \in \mathbb{Q}[H^{ab}]$. \vspace{0,1cm}

{\it We claim that} if $\IInd_H^G(x)\epsilon_H = 0$ in $\Q(\wt{G})$, then $x=0$.\vspace{0,1cm}

\noindent To start denote $\IInd_H^G(M_{(\{ h\},\emptyset)}) = M_{(A_h,B_h)}$ for $h \in H^{ab}$ and write $M_{(C,D)}$ for $h=1$. Note that $e(A_h,B_h) = M_{(C,D)}$ for each $h$. Furthermore,
$$M_{(A_h,B_h)}M_{(C,D)} = \IInd^G_H(M_{(\{ h\},\emptyset)}.M_{(\{ T_H\},\emptyset)}) = M_{(A_h,B_h)}.$$
Next recall the canonical form of $\epsilon_H$ obtained in \Cref{form of idempotents}. Remark that the proof in fact yields that the elements $M_{(C_i,D_i)}$ are idempotent. Now consider the idempotent $\RRes^G_H(M_{(C_i,D_i)}) := M_{(E_i,F_i)} \in \Q(\wt{H})$. Since $C \lneq C_i$ and due to the bijection (\ref{bijectie tussen sub en idempotenten}) we know that $\{ T_H \} \neq E_i \leq H^{ab}$.  All this entails that $\RRes^G_H(M_{(A_h,B_h)} . M_{(C_i,D_i)}) = M_{(\{ h\},\emptyset)} . M_{(E_i,F_i)} \notin R(H)$ and therefore 
\begin{equation}\label{prod niet weer in R}
M_{(A_h,B_h)} . M_{(C_i,D_i)} \neq M_{(A_{h'}, B_{h'})}
\end{equation}
 for any $h' \in H^{ab}$. All together, using \Cref{form of idempotents} and denoting $x = \sum_{h\in H^{ab}}q_h M_{(\{ h\},\emptyset)}$, we see that 
\begin{equation}\nonumber
\IInd_H^G(x)\epsilon_H = \IInd_H^G(x) + \sum_{\substack{C\lneq C_i \\ h \in H^{ab}}} z_i q_h M_{(A_h,B_h)}M_{(C_i,D_i)}  \in \Q(\wt{G})
\end{equation}
for some $a_i \in \Z$. This decomposition combined with (\ref{prod niet weer in R}) and the fact that $\Irr_1(K\subset K[G])$ is a basis of $\Q(\wt{G})$ implies that $\IInd^G_H(x) =0$. Consequently by \Cref{voor most easy filtratie is semigroupsalgebra} also the claim follows.

Now since  $(\IInd_H^G(x)\epsilon_H)^{n} = \IInd_H^G(x^n) \epsilon_H$, the claim implies that $x \in N(\Q[H^{ab}]) =0$. Thus $\varphi$ is a monomorphism. Next, by (\ref{the map as modules}),
$$ \dim_{\Q} \Q(\wt{G})/N \leq \dim_{\Q} \bigoplus_{H \in \Sub(G)}\Q[H^{ab}]  < \infty.$$  Therefore, $\varphi$ must even be an isomorphism.
\end{proof}

As a consequence we see that under the above assumptions the Jacobson and nil radical coincide.

\begin{corollary}\label{both radical equal under assumption coro}
Let $G$ be a finite nilpotent group such that $E(G) = 0 = P(G)$. Then  $J(\Q(\wt{G})) = N (\Q(\wt{G}))$ and $\Q(\wt{G})/J$ is semisimple.
\end{corollary}
\begin{proof}
By \Cref{R niet triviaal alleen voor non-nilpotent} we have that $R(H)= \Q[H^{ab}]$ for all $H \leq G$. The theorem of Perlis-Walker tells us that $\Q[H^{ab}]$ is a direct product of cyclotomic fields. In particular, it contains no nilpotent elements and hence intersects trivially with $N(H)$. \Cref{gencharring} now yields that $\Q(\wt{G})$ is a direct product of a finite number of group algebras of the form $\Q[H^{ab}]$, which are each semisimple, and consequently  $\Q(\wt{G})$ also. This implies that $J(G) \subseteq N(G)$. However the converse inclusion holds for any commutative ring, thus $J(G) = N(G)$ as needed. 
\end{proof}

It is well known that for finitely generated commutative rings the Jacobson and nilradical coincide and hence understanding this property is of special interest. Any of the following properties would be especially useful to obtain in general:

\begin{question}\label{vraag over f.g and Noetherian en etc}
Let $G$ be a finite group. Is $\Q(\wt{G})/N$ finitely generated? Is it Noetherian or even semisimple?
\end{question}

For any class of groups for which $\Q(\wt{G})/J$ is semisimple it would be interesting to investigate whether $\Q(\wt{G})$ satisfies some kind of Wedderburn-Malcev decomposition (i.e. it contains a maximal semisimple subalgebra isomorphic to $\Q(\wt{G})/J$).
 
In the next section we will have a closer look at nilpotent groups of class $2$. More precisely we will prove that in that case the assumptions $E =0 =P$ are fulfilled and furthermore we will prove that $\Sub(G)$ consists of all the subgroups.

\addtocontents{toc}{\protect\setcounter{tocdepth}{1}}
\section{A look at nilpotent and isocategorical groups}\label{sectie nilp en isocat grps}

In this section we will apply the short exact sequence of \theref{ses} to groups of nilpotent class $2$ and to certain isocategorical groups. More precisely, we will prove that if $G$ has nilpotency class $2$ then the obstruction modules vanish (i.e. $E = 0 =P$) and thus we may apply \Cref{gencharring}. We will also prove that $\Sub(G)$ in this case is the set of all subgroups, hence obtaining a full description in the case of class $2$. Afterwards we will apply this description to distinguish a class of isocategorical groups and point out some natural group-theoretical invariants which are unfortunately not monoidal Morita invariants.

\subsection{Nilpotent groups of class 2}\label{sectie class 2}\hspace{1cm}\vspace{0,2cm}\newline
Combined with the results of \Cref{sectie interpretaties}, the following characterization of groups of nilpotency class $2$ will directly yield the vanishing phenomena. This characterization might be known to some experts however we were unable to find this in the literature.

\begin{proposition}\label{characterization class 2}
Let $G$ be a finite group. Then the following are equivalent
\begin{enumerate}
\item $G$ is nilpotent of class $2$
\item for every $V \in \Irr(G)$, there exists $n \geq 1$ such that $V^{\ot n}$ completely linearizes.
\end{enumerate}
\end{proposition}
\begin{proof}
Suppose that the nilpotency class of $G$ is larger than 2. Then there exists $g \in G' \setminus Z(G)$. Since $Z(G) = \bigcap_{\chi \in \Irr(G)}Z(\chi)$, there exists an irreducible complex character $\chi$ such that $|\chi(g)| < |\chi(e)|$. If there would exist an $n > 1$ such that $\chi^n$ is a positive linear combination of linear characters of $G$, then  $\chi^n(g) = \chi^n(e)$, because $g \in G'$ and \Cref{tussenstap}. On the other hand $|\chi^n(g)| < |\chi^n(1)|$, which gives a contradiction. Conversely, suppose that $G$ is of nilpotency class at most 2, i.e. $G' \subseteq Z(G)$ and let $U$ be an irreducible $G$-representation. Then, $\Res^G_{G'}(U) \cong S^{\oplus \dim(U)}$ for some one-dimensional $G'$-representation $S$. There exists $n \geq 1$ such that $S^{\ot n}$ is the trivial $G'$-representation $T_{G'}$. Hence
$$ (\Res^G_{G'}(U))^{\ot n} \cong T_{G'}^{\oplus n\dim(U)}.$$
In other words, $U^{\ot n}$ decomposes into irreducible $G$-representations which all lie over the trivial $G'$-representation which entails that in fact $U^{\ot n}$ is a sum of one-dimensional representations.
\end{proof}

Consequently we may apply \Cref{hoe semigroups zonder idempotent vinden} to obtain that $P$ vanishes. Concerning $E(G)$ if $V$ appears with non-zero multiplicity in $M_{(\emptyset,D)}$, then $M_{(\emptyset,D)} = K[G](v + w)$ for some $v \in V$. If $V^{\ot n}$ linearizes, then $M_{(\emptyset,D)}^{\ot n} =K[G](v+w)^{\ot n}$ with $v^{\ot n}$ generating a $1$-dimensional $K[G]$-module. Hence by \Cref{lemma with components of cyclic from vector} $M_{(\emptyset,D)}$ can not be an idempotent, i.e. $E(G)=0$. 

\begin{corollary}\label{P=0}
If $G$ is nilpotent of class $2$, then $P(G) = 0 = E(G)$.
\end{corollary}

Thus all conditions are fulfilled to apply \Cref{gencharring}. Moreover, from \Cref{both radical equal under assumption coro} we know that both radicals coincide. Hence, in order to obtain a full description of the maximal semisimple quotient of the glider representation ring $\Q(\wt{G})$, it remains to describe $\Sub(G)$ precisely. We will now prove that for nilpotent groups of class $2$ in fact $\Sub(G)$ contains all the subgroups of $G$. 
\begin{theorem}\label{represention ring for class 2}
Let $G$ be a finite nilpotent group of class $2$. Then
$$\mathbb{Q}(\wt{G})/J \cong \bigoplus_{H \leq G} \mathbb{Q}[H^{ab}].$$
where the sum runs over all subgroups of $G$.
\end{theorem}
\begin{proof}
It remains to prove that one indeed obtains all the subgroups. Let $H \leq G$ be a subgroup. First assume that $G' \leq H$. Then by definition of $A_\iota(H)$ the idempotent $M_{(A_\iota(H),\emptyset)}$ is such that $\Chain(A_\iota(H_1),\emptyset) = (G \geq H)$, i.e. $H= H(A_\iota(H),\emptyset)$ and $H \in \Sub(G)$

If $G' \not\subset H$ we form $G'H$. If $G$ is class $2$, then $G' \subseteq Z(G)$ and using this we easily compute that $(G'H)' = H'$. In this way created a chain $G \geq G'H \geq H$ such that $G' \leq G'H \leq G$ and $(G'H)' \leq H \leq G'H$. Now consider more generally any chain $G \gneq H_1 \gneq H_2$ such that $G'\leq H_1$ and $H'_1 \leq H_2$. {\it We claim that }
$$M_{(C,D)}:= \IInd^G_{H_1}(M_{(A_\iota(H_2),\emptyset)}) \text{ is such that } H(C,D) = H_2.$$ 
By the preceding construction this would imply that also an $H$ with  $G' \not\subset H$ would be in $\Sub(G)$ and therefore $\Sub(G) = \{H \leq G \}$, as desired.

{\it Prove of the claim:}
By construction $C = \{ c \in G^{ab} \mid \Res^G_{H_1}(T_c) \in A_\iota (H_2) \}$, where the condition $\Res^G_{H_1}(T_c) \in A_\iota (H_2)$ can be rephrased as $H_2 \subseteq \ker(\Res^G_{H_1}(T_c))$. Thus $C = \{ c \in G^{ab} \mid H_2\subseteq \ker(T_c) \}$
 and hence $A_\iota(H_1) \leq C$. Now we need to consider $\RRes^G_{\mc{L}(C)}(M_{(C,D)}) = \RRes^{H_1}_{\mc{L}(C)} (M_{(A_\iota(H_2), \emptyset)})$ (where we used part (2) of \cref{irreducible under Ind(Res)}). Denote the latter by $M_{(C_1, \emptyset)}$. Due to the above, $H_2 \leq \mc{L}(C_1)$. On the other hand 
$$H_2 = \bigcap_{d \in A_\iota(H_2)} \ker (T_d) \geq \bigcap_{d \in A_\iota(H_2)} \ker (\Res^{H_1}_{\mc{L}(C)}(T_d)) = \mc{L}(C_1).$$
Thus $\mc{L}(C_1) = H_2$ and we obtain that $\Chain(C,D) = (G \gneq \mc{L}(C) \geq H_2)$ as claimed.
\end{proof}

Note that the proof of the claim above is valid for any finite group. This raises the question whether similar methods might work for arbitrary chains $G \gneq H_1 \gneq \cdots \gneq H_m$ with $H'_i\leq H_{i+1} \leq H_i$. If yes, then for any solvable group we would obtain a positive answer to the following question.
\begin{question}\label{the set Sub(G)}
Let $G$ be finite group. Is $\Sub(G)$ equal to the set of all subnormal subgroups of $G$?
\end{question}
At this point it is interesting to recall that for nilpotent groups all subgroups are subnormal. To end this subsection, we give an example of how the glider representation ring can be used to distinguish certain groups by filtering certain group theoretical invariants from the decomposition theorem above.

\begin{example}
There are two non-abelian groups of prime cube order $p^3$, namely $C_{p^2}\ltimes C_p$ and $H_p$ the Heisenberg group. For instance, if $p=2$ these are simply $D_8$ and $Q_8$. The groups $C_{p^2}\ltimes C_p$ and $H_p$ have the same character table. However they are nilpotent of class $2$ and have $\mathcal{Z}(G) = G' = C_p$. It follows that the glider representation rings are non-isomorphic since they have a different number of subgroups.
\end{example}

\subsection{Isocategorical groups and (non-)monoidal Morita invariants}\hspace{1cm}\vspace{0,2cm}\newline\label{subsectie isocategorical groups}
Recall that two groups $G_1$ and $G_2$ are called isocategorical if $\rep_{\C}(G_1)$ and $\rep_{\C}(G_2)$ are equivalent as tensor category without consideration of the symmetric structure. It was proven by Etingof-Gelaki \cite[Lemma 3.1.]{EtGe} that if $G_1$ and $G_2$ are isocategorical, then there exists a Drinfeld twist $J$ such that $\mathbb{C}[H]^J$ is isomorphic as Hopf algebra to $\mathbb{C}[G]$. In fact, all groups isocategorical to a given group $G$ can be explicitely classified in group theoretical terms. We now briefly recall this description, for details see \cite{EtGe}.
 
Let $A$ be a normal abelian subgroup of $G$ of order $2^{2^{m}}$ for some $m \in \mathbb{N}$ and write $Q= G/A$. Let $R: \widehat{A} \rightarrow A$ be a $G$-invariant skew-symmetric isomorphism between $A$ and its character group $\widehat{A}$. This form induces a $Q$-invariant cohomology class $[\alpha]$ in $H^2(\widehat{A}, K^{*})^Q$ (where the action of $\widehat{A}$ on $K^*$ is the trivial one).  By definition, $q\alpha/\alpha$ is a trivial $2$-cocycle for any $q\in Q$. Hence there exists a $1$-cochain $z(q): \widehat{N} \rightarrow K^{*}$ such that $ \partial (z(q)) = q\alpha / \alpha$. Define the cochain
$$b( p, q) := \frac{z(pq)}{z(p) z(q)^p}.$$
One can check that it has trivial coboundary and hence $b(p,q) \in \widehat{\hat{A}} \cong A$. In other words $b(p,q) \in Z^2(Q,N)$.
Define now the group $G_b$ to be equal to $G$ as a set, but with multiplication defined by
$$g \cdot_b h = b(\overline{g}, \overline{h}) gh.$$
 
In \cite[Theorem 1.3.]{EtGe} Etingof-Gelaki proved that if $G_2$ is isocategorical to $G_1$, then $G_2 \cong (G_1)_b$ for $b$ some cocycle obtained as in the procedure above. In particular,  \cite[Corollary 1.4.]{EtGe}, if a group $G$ does not have a normal abelian $2$-subgroup equipped with a $G$-invariant alternating form then it is categorically rigid, i.e. no other group is isocategorical equivalent to it. This holds for example if $4$ does not divide $|G|$.

In \cite[Section 4]{MaHi} an infinite family of pairs of non-isomorphic, yet isocategorical groups $G^m$ and $G^{m}_{b}$, for $3 \leq m \in \mathbb{N}$, was constructed. As proven by Goyvaerts-Meir \cite{GoMe} the case $m=3$ yields the smallest non-isomorphic, but isocategorical, groups (which are thus of order $64$). 

\begin{proposition}\label{class of isocat have non iso glider rep ring}
Let $3 \leq m \in \mathbb{N}$ and $G^m, G^{m}_{b}$ be the isocategorical groups from \cite{MaHi}. Then their representation rings over $\mathbb{C}$ are isomorphic rings, however the glider representation rings $\RRepR_1(G^m)$ and $\RRepR_1(G^m_b)$ are non-isomorphic rings.
\end{proposition}

More generally, suppose that $G$ and $H$ are isocategorical. Thus there exists a monoidal equivalence $F: \rep(G) \rightarrow \rep(H)$. Then $F$ clearly induces an isomorphism between the Grothendieck rings $K_0(\rep(G))$ and $K_0(\rep(H))$. Thus the first part of \Cref{class of isocat have non iso glider rep ring} is a general statement about isocategorical groups and hence follows from \cite[Theorem 1.3.]{EtGe} and the construction of the groups.

For the second part of \Cref{class of isocat have non iso glider rep ring} we start by recalling the construction in case $m=3$. Let $G = N \rtimes H$ with $N \cong \mathbb{Z}_4 \times \mathbb{Z}_4$ and $H = \langle \left( \begin{array}{ll}
1 & 2 \\ 0 & 1
\end{array}\right) \rangle \times \langle \left( \begin{array}{ll}
1 & 0 \\ 2 & 1
\end{array}\right) \rangle = \langle h_1 \rangle \times \langle h_2 \rangle$ (where the entries of the matrices are taken modulo $4$). Denote the generators of $N$ by $n_1$ and $n_2$. The action of $H$ on $N$ is as follows, where $T$ denotes transposition,
$$\left( n^h \right)^T = h^{-1} n.$$ For example $\left( n_1^{h_1} \right)^T= \left( \begin{array}{ll}
1 & -2 \\ 0 & 1
\end{array}\right) \, \left( \begin{array}{l}
1 \\ 0 
\end{array}\right) = \left( \begin{array}{l}
1 \\ 0 
\end{array}\right).$ In other words $n_1^{h_1} = n_1$. 

Now we need the coycle $b$ to twist $G$. The action of $H$ on $\widehat{N}$ is given by $(h \omega)(n) := \omega (n^h)$. Define
$$b(h_1^{t_1} h_2^{t_2}, h_1^{r_1} h_2^{r_2}) = n_1^{l_1} n_2^{l_2},$$
with $l_i = \delta_{1,t_i} \delta_{1, r_i}$. With easy computations one can check the following.

\begin{lemma}
With notations as above we have that $G = \langle n_1,n_2,h_1,h_2 \mid \mathcal{R}_1 \rangle$ and $G_b = \langle n_1,n_2,h_1,h_2 \mid \mathcal{R}_2\rangle$ with
\begin{itemize}
\item $\mathcal{R}_1 =\{ n_i^4=1, h_i^2=1, (h_1,h_2)=1, (n_1, n_2)=1, n_1^{h_1}=n_1, n_2^{h_1}= n_1^2 n_2, n_2^{h_2} = n_2, n_1^{h_2} = n_2^2 n_1 \}$
\item $\mathcal{R}_2 = \{ n_i^4=1, h_i^2=n_i^2, (h_1,h_2)=1, (n_1, n_2)=1, n_1^{h_1}=n_1, n_2^{h_1}= n_1^2 n_2, n_2^{h_2} = n_2, n_1^{h_2} = n_2^2 n_1 \}.$
\end{itemize}
\end{lemma}

In fact the groups above are those with SmallGroup ID's [64, 232], resp. [64, 236]. Note that both $G$ and $G_b$ are nilpotent of class $2$ and in fact their centers equal their commutator subgroups (e.g. $G' = \langle n_1^2, n_2^2 \rangle= \mathcal{Z}(G)$). Furthermore, the set of isomorphism classes of subgroups of $G$ and $G_b$ are isomorphic. However, both the number of subgroups and the number of subgroups of a given isomorphism type (for certain non-normal subgroups) are different. In case of groups of nilpotency class $2$, \Cref{gencharring} shows that the number of subgroups is determined by $\Q(\wt{\cdot})$. Consequently in our case, $\RRepR_1(\widetilde{G}) \ncong \RRepR_1(\widetilde{G_b})$ (since otherwise the same would hold after extension of scalars to $\mathbb{Q}$ and taking the quotient by the Jacobson radical). Thus we have proven \Cref{class of isocat have non iso glider rep ring} in the case $m=3$.  The case of a general $m$ is analogue but notational more cumbersome.

The way we distinguished the groups is especially interesting in light of the following result of Shimizu \cite{Shi}: if $G$ and $H$ are isocategorical groups then the number of elements of order $n$, for any $n$, in $G$ and $H$ coincide. An invariant with such a property is called a {\it monoidal Morita invariant}. 

\begin{corollary}
The following numbers:
\begin{enumerate}
\item total number of subgroups,
\item number of subgroups of a given isomorphism type,
\end{enumerate} 
are not monoidal Morita invariants.
\end{corollary} 

It would be interesting to know certain isomorphism types for which their number is actually a monoidal Morita invariant. In upcoming work we will describe, in a more systematic way, invariants that are determined by $\RRepR_1(\widetilde{G})$ but which is not necessarily detected by $\rep_{\C}(G)$ viewed as tensor category.

\bibliographystyle{plain}
\bibliography{GlidRep}

\end{document}